%% file: mec.tex
\documentclass [12pt,a4paper]{amsart}
\pdfoutput=1

\usepackage[utf8]{inputenc}
\usepackage[T1]{fontenc}
\usepackage{lmodern}
\usepackage{microtype}
\usepackage[svgnames]{xcolor}

\usepackage {amssymb}
\usepackage {amsmath}
\usepackage {amsthm}
\usepackage{mathtools}
\usepackage{xfrac}
\usepackage{cancel}
\usepackage{mathrsfs}
\usepackage{tikz,graphicx}
\usetikzlibrary{arrows.meta,patterns}

\usepackage{etoolbox}
\AtBeginEnvironment{itemize}{\everydisplay{\displaywidth=\linewidth}}

\usepackage{needspace}

\usepackage{thmtools}
\usepackage{thm-restate}

\usepackage{scrextend}

\usepackage[top=1.16in, bottom=1.16in, left=1.67in, right=1.67in]{geometry}

\usepackage{url}

\usepackage{bm}

\usepackage{enumitem}

\usepackage{graphicx}

\usepackage{xfrac}

\usepackage{hyperref}
\hypersetup{
    colorlinks=true,
    linkcolor=blue,
    filecolor=magenta,
    urlcolor=cyan,
    citecolor=magenta
}
\usepackage[all]{hypcap}

\usepackage{aliascnt}
\usepackage{cleveref}

\newcommand {\td} {\mathrm{td}}

\newcommand {\cl} {\mathrm{cl}}
\newcommand {\pr} {\mathrm{pr}}

\newcommand {\SL} {\mathrm{SL}}
\newcommand {\PSL} {\mathrm{PSL}}
\newcommand {\GL} {\mathrm{GL}}
\newcommand {\PGL} {\mathrm{PGL}}

\newcommand {\alg} {\mathrm{alg}}

\newcommand {\rk} {\mathrm{rk}}

\newcommand {\seq} {\subseteq}
\newcommand {\C} {\mathbb{C}}
\newcommand {\R} {\mathbb{R}}
\newcommand {\Rbar} {\mathbf{R}}

\newcommand {\Z} {\mathbb{Z}}
\newcommand {\Q} {\mathbb{Q}}

\newcommand{\h}{\mathbb{H}_{\pm}}
\newcommand{\uh}{\mathbb{H}_+}
\newcommand{\lh}{\mathbb{H}_-}

\newcommand{\p}{\mathbb{P}}

\newcommand{\zbar}{\mathbf{z}}
\newcommand{\xbar}{\mathbf{x}}
\newcommand{\ybar}{\mathbf{y}}
\newcommand{\rbar}{\mathbf{r}}
\newcommand{\wbar}{\mathbf{w}}
\newcommand{\fbar}{\mathbf{f}}
\newcommand{\Fbar}{\mathbf{F}}
\newcommand{\gbar}{\mathbf{g}}
\newcommand{\Wbar}{\mathbf{W}}
\newcommand{\habar}{\mathbf{h}}
\newcommand{\cbar}{\mathbf{c}}
\newcommand{\vbar}{\mathbf{v}}

\newcommand{\ubar}{\mathbf{u}}

\newcommand{\pbar}{\mathbf{p}}
\newcommand{\Pbar}{\mathbf{P}}

\newcommand{\Schw}{\mathcal{S}}
\newcommand{\Dbar}{\overline{D}}
\newcommand{\iu}{\mathrm{i}}

\newcommand{\etabar}{\ensuremath{\bm{\eta}}}
\newcommand{\zetabar}{\ensuremath{\bm{\zeta}}}
\newcommand{\taubar}{{\ensuremath{\bm{\tau}}}}
\newcommand{\xibar}{{\ensuremath{\bm{\xi}}}}
\newcommand{\thebar}{\ensuremath{\bm{\theta}}}

\newcommand{\gr}{\mathrm{Graph}}

\DeclareMathOperator{\im}{Im}
\DeclareMathOperator{\re}{Re}

\theoremstyle {plain}
\newtheorem {theorem}{Theorem}[section]

\newtheorem {lemma} [theorem] {Lemma}
\newtheorem {fact} [theorem] {Fact}
\newtheorem {proposition} [theorem] {Proposition}
\newtheorem {corollary} [theorem] {Corollary}
\newtheorem {conjecture} [theorem] {Conjecture}
\newtheorem {claim} {Claim}
\newtheorem* {claim*}{Claim}

\newtheorem{innercustomthm}{Theorem}
\newenvironment{customthm}[1]
  {\renewcommand\theinnercustomthm{#1}\innercustomthm}
  {\endinnercustomthm}

\theoremstyle {definition}

\newtheorem{example} [theorem] {Example}

\newtheorem*{definition*}{Definition}

\theoremstyle {remark}
\newtheorem {remark} [theorem] {Remark}
\numberwithin{equation}{section}

\newenvironment{claimproof}{\bgroup\begin{proof}}{\end{proof}\egroup}

\calclayout

\begin{document}

\title{Modular Existential Closedness}

\author{Vahagn Aslanyan}

\address{Department of Mathematics, University of Manchester, Manchester, UK}
\email{vahagn.aslanyan@manchester.ac.uk}

\date{}

\thanks{This work is supported by EPSRC Open Fellowship EP/X009823/1. For the purpose of open access, the author has applied a Creative Commons Attribution (CC BY) licence to any Author Accepted Manuscript version arising from this submission.}



\begin{abstract}
 We prove the Modular Existential Closedness conjecture for the modular $j$-function extended from the upper half-plane to $ \mathbb{C} \setminus \mathbb{R}$ by Schwarz reflection: systems of equations involving the field operations and $j$ have solutions in $\mathbb{C}$ unless they are overdetermined. Roughly, a system is overdetermined if independent equations outnumber variables, possibly after applying the $\operatorname{PGL}_2(\mathbb Q)$-transformation rules for $j$. Geometrically, our theorem states that any \textit{broad} and \textit{free} irreducible algebraic variety $V \subseteq \mathbb{C}^n \times \mathbb{C}^n$ contains a point of the form $(\tau_1, \ldots, \tau_n, j(\tau_1), \ldots, j(\tau_n))$ with $\tau_i \in \mathbb{C} \setminus \mathbb{R}$. Broadness and freeness are geometric conditions ruling out overdetermined systems; they are motivated by the Modular Schanuel conjecture.

In the proof, we construct $\operatorname{PGL}_2(\mathbb Z)$-transforms of the given equations converging to simpler ones. Ax--Schanuel prevents degeneracies and guarantees an isolated solution of the system of limiting equations. Rouché’s theorem then gives a solution of the original system.

For this construction, we establish a simultaneous modular approximation result using Ratner’s theorems on unipotent flows in powers of $\operatorname{PSL}_2(\mathbb{R})$. For a suitable finite family of holomorphic functions real-valued on a real box, applying common modular transformations to the inputs and separate ones to the outputs makes the functions approach prescribed real M\"obius maps, while the inputs tend to any given point of the box.
\end{abstract}

\maketitle

\setcounter{tocdepth}{1}

\tableofcontents

\newpage

\section{Introduction and overview}

In this paper, we study the solvability of systems of equations involving the field operations and the modular $j$-function -- a classical function central to the theory of modular forms and elliptic curves. We show that, with $j$ defined on both upper and lower half-planes, such a system has solutions in $\C$ unless there is a natural obstruction given by the algebraic and transcendental properties of $j$. For instance, we show that the equation \[ z^3+j(z) + j(\iu z)^2 + j(z+j(\iu z)) = 0\] has infinitely many complex solutions (see \S\ref{subsec: simple example}).\footnote{We use upright $\iu$ for the imaginary unit.}  This problem lies naturally at the boundary between model theory and complex geometry. 

The Modular Existential Closedness conjecture, proposed in \cite{Aslanyan-Kirby-blurred-j}, predicts when such systems of equations have solutions. However, it is formulated in terms of $j$ as a function on the upper half-plane. As we explain in \S\ref{subsec: extend to h counterex}, the conjecture in its original form is false because of a trivial obstruction arising from restricting the arguments to the upper half-plane. Extending $j$ to the lower half-plane via Schwarz reflection is canonical (see \S\ref{subsec: setup}) and removes this obstruction. 

Our main result gives an unconditional proof of the conjecture for the extended $j$-function. In this paper, \textit{Modular Existential Closedness} refers to this extended version.

\subsection{Setup and the main results}\label{subsec: setup}

Let \[ \uh := \{ z \in \C: \im z > 0 \} \text{ and } \lh := \{ z \in \C: \im z < 0 \}\] denote the upper and lower half-planes of the complex plane, respectively, and let $\h := \uh \cup \lh = \C \setminus \R$. The modular $j$-function is defined on $\uh$, but we extend it to $\h$ by Schwarz reflection, i.e.\ \[ \text{for } z \in \lh \text{ set } j(z) := \overline{j(\bar{z})}. \]

The extension of $j$ to $\h$ is canonical: it satisfies $j(z)\equiv j(-z)$ (see \S\ref{subsec: j prelim}) and for any $\tau \in \h$ the lattices $\Z + \tau \Z$ and $\Z + (-\tau)\Z$ agree, so on either half-plane $j(\tau)$ is the $j$-invariant of the elliptic curve $\C / (\Z + \tau \Z)$. Moreover, the extended $j$ realises the level-one uniformisation attached to the full modular Shimura datum $(\GL_{2,\Q},\h)$ (see \cite[\S\S1.6, 1.8, 3.16]{Deligne1971}, \cite[p.~280]{ET14}, and \cite[\S2.1]{DH17}). This extended function also appears in a number of papers such as \cite[\S\S1,~2.2]{Eterovic18}, \cite[\S2.1]{Aslanyan-Eterovic-Kirby-closure-op}, and \cite[\S2]{JKSS26}.

The function $j \colon \h \to \C$ is invariant under the action of $\PGL_2(\Z)$ on $\h$ via linear fractional transformations, and has ``algebraic'' behaviour under the action of $\PGL_2(\Q)$. There is a countable collection of irreducible polynomials $\Phi_N(X,Y) \in \Z[X,Y],\, N \in \Z_{>0}$, known as \emph{modular polynomials}, such that for any $z_1, z_2 \in \h$
\[ \exists g \in \PGL_2(\Q) : z_2= gz_1 \Longleftrightarrow \exists N \in \Z_{>0} : \Phi_N(j(z_1), j(z_2)) = 0. \]
See \S\ref{sec: prelim} for more details.

For any positive integer $n$, which will always be clear from the context, we define \[ \gr(j) := \bigl\{ \bigl(z_1, \ldots, z_n, j(z_1), \ldots, j(z_n)\bigr): z_i \in \h \text{ for all } i \bigr\} .\]

Observe that understanding which systems of equations involving $+, \cdot, j$ have solutions in $\C$ is equivalent to understanding which algebraic varieties intersect the graph of $j$.  For instance, the equation \[ j(j(z^2)) = z\] has a solution if and only if the variety $V \seq \C^3 \times \C^3$ defined by the equations \[ x_2=x_1^2,~ x_3 = y_2,~ y_3 = x_1 \] intersects $\gr(j)$. This gives interchangeable algebraic and geometric interpretations of the problem.

\begin{definition*}
Let $(\xbar, \ybar):= (x_1, \ldots, x_n, y_1, \ldots, y_n)$ denote the coordinates of $\C^{n}\times \C^n$. Let $V \seq \C^n \times \C^n$ be an (irreducible) algebraic variety. 
\begin{itemize}
    \item $V$ is said to be \emph{broad} if for any non-empty subset $I \seq \{ 1, \ldots, n \}$, the projection of $V$ to $(\xbar_I, \ybar_I) := \bigl((x_i)_{i \in I}, (y_i)_{i \in I}\bigr)$ has dimension at least $|I|$. In particular, $\dim V \geq n$. Broadness roughly means that no subsystem of the given system of equations has more independent equations than variables.

    \item $V$ is said to be \emph{free} if no coordinate is constant on $V$, and for any $1\leq i< i' \leq n$, no relation of the form $x_i = g x_{i'}$ with $g \in \PGL_2(\Q)$  holds on $V$, and no relation of the form $\Phi_N(y_i, y_{i'})=0$, where $\Phi_N$ is a  modular polynomial, holds on $V$. In terms of systems of equations, freeness means the equations do not fix the value of any variable, and do not imply any algebraic equation coming from the transformation rules of $j$.
\end{itemize}
\end{definition*}

The following is our main result (varieties are assumed to be irreducible).

\begin{theorem}[Modular Existential Closedness -- MEC]
\label{thm: EC full}
   Let $V \seq \C^n \times \C^n$ be a broad and free variety. Then $V \cap \gr(j) \neq \emptyset$.
\end{theorem}

Broadness and freeness are purely algebro-geometric conditions -- the definitions are motivated by properties of $j$, but do not refer to it -- capturing the algebraic and (hypothetical) transcendental properties of $j$ that do not allow solutions for certain systems of equations. These notions are thoroughly discussed and justified in \S\ref{subsec: broad free}. But first we observe that \Cref{thm: EC full} implies a seemingly stronger version of itself.

\begin{corollary}\label{cor: zar dense ec}
    Let $V \seq \C^n \times \C^n$ be a broad and free variety. Then $V \cap \gr(j)$ is Zariski dense in $V$.
\end{corollary}
\begin{proof}
    Given a non-empty Zariski open subset $U \seq V$, we must show that $U \cap \gr(j) \neq \emptyset$. It suffices to show this when $U$ is defined by a single inequation $P\neq 0$ inside $V$. Consider the variety $V' \seq V \times \C^2 \seq \C^{2(n+1)}$ given by $x_{n+1}y_{n+1}P=1$ (this is known as the Rabinowitsch trick). It is easy to verify that $V'$ is broad and free and its projection to $V$ is equal to $U$. \Cref{thm: EC full} asserts that $V'$ contains a point from the graph of $j$, whose projection then belongs to $U \cap \gr(j)$.
\end{proof}

\begin{remark}
    For a reducible Zariski closed set $V$, assuming it has a broad and free irreducible component ensures that $V \cap \gr(j) \neq \emptyset$. Freeness and broadness of $V$ itself do not imply that it has a broad and free component.
\end{remark}

Our methods also prove the following result (cf. \cite[Prop.~3.9, Thm.~3.17]{Gallinaro-raising-to-powers}). Note that a variety $L \seq \C^n$ is said to be $\PGL_2(\Q)$-\textit{free} if $L \times \C^n$ is free in the above sense, i.e.\ the coordinates on $L$ are non-constant and satisfy no $\PGL_2(\Q)$-relations.

\par\Needspace{4\baselineskip}
\begin{restatable}{theorem}{secthm}
\label{thm: secondary thm density}
    Let $L \seq \C^n$ be an algebraic variety containing a Zariski dense set of real points. Then for any variety $W \seq \C^n$ for which $L\times W$ is broad and free, the set $j(L \cap \h^n)\cap W$ is dense in $W$ in the Euclidean topology. In particular, if $L$ is $\PGL_2(\Q)$-free then $j(L\cap \h^n)$ is dense in $\C^n$ in the Euclidean topology.
\end{restatable}

For example, the set $\{ (j(z), j(z^2)): z \in \h \setminus \iu \R \}$ is Euclidean dense in $\C^2$ and in fact meets any algebraic curve in $\C^2$ other than a horizontal line, a vertical line, or a modular curve (defined by a modular equation) in a Euclidean dense subset of that curve.

\subsection{Broadness and freeness}\label{subsec: broad free}

Let us now justify why broadness and freeness are natural conditions in \Cref{thm: EC full}. A system of equations of the form 
\begin{equation}\label{eqn: j(z)=z=j(z^2)}
    j(z) = z,~ j(\sqrt[3]{2}z) = z^3
\end{equation}
is not expected to have a solution, for it is \textit{overdetermined} -- there are more independent equations than variables. To understand whether solutions exist, we observe that \eqref{eqn: j(z)=z=j(z^2)} implies $\td_{\Q}\Q(z, \sqrt[3]{2}z, j(z), j(\sqrt[3]{2}z)) \leq 1$, where $\td$ stands for transcendence degree. In this situation the Modular Schanuel Conjecture imposes strong restrictions on $z$.

\begin{conjecture}[Modular Schanuel Conjecture -- MSC]
    If $z_1, \ldots, z_n \in \h$ are non-quadratic numbers in distinct $\PGL_2(\Q)$-orbits, then \[ \td_{\Q} \Q(z_1, \ldots, z_n, j(z_1), \ldots, j(z_n)) \geq n.\]
\end{conjecture}

This follows from the Grothendieck--Andr\'e generalised period conjecture; see \cite[\S6.3]{Aslanyan-Eterovic-Kirby-closure-op} and the references therein.

In the above example, MSC implies that either $z$ or $\sqrt[3]{2}z$ is quadratic or $\sqrt[3]{2}z = \frac{az+b}{cz+d}$ for some $a,b,c,d \in \Q$ with $ad-bc \neq 0$. In all cases, $z \in \Q^{\alg}$, hence $j(z)=z \in \Q^{\alg}$ and $j(\sqrt[3]{2}z) = z^3 \in \Q^{\alg}$. By Schneider's theorem \cite{Sch37} (which is the $n=1$ case of MSC), $z$ and $\sqrt[3]{2}z$ must both be quadratic irrationals, which is not possible.

While MSC is considered out of reach,\footnote{Well... maybe not for much longer.} 
the above example still shows that we should not expect overdetermined systems to have solutions.\footnote{MSC is used only to motivate broadness/freeness; it is not assumed in our proof of \Cref{thm: EC full}.} 

In the same spirit,  a system like
\[ j(z_1) = z_1,~ j(\sqrt[3]{2}z_1) = z_1^3,~ z_1+z_2+z_3 = j(z_3) \]
has three equations and three variables, so does not seem to be overdetermined. Nevertheless, the subsystem of the first two equations is overdetermined, so the whole system is not expected to have solutions and should be regarded as overdetermined.

Geometrically, this means that we do not expect $V$ to intersect the graph of $j$ unless $\dim V \geq n$ and a similar inequality holds for all projections of $V$ to any number of $\xbar$-coordinates and the corresponding $\ybar$-coordinates. This is exactly broadness.

Now we turn our attention to freeness. A system like \[ z_1 = \iu,~ (j(z_1) - 1728)z_2 = 1 \] has no solutions (even though it is broad) because $j(\iu) = 1728$. So we do not allow constant coordinates. Similarly, the system \[ z_2 = 2z_1,~ \Phi_2(j(z_1), j(z_2))=1 \] has no solution since $\Phi_2(j(z), j(2z)) \equiv 0$. Therefore, we do not allow $\PGL_2(\Q)$-relations on $\xbar$ and modular relations on $\ybar$. 

In other words, a broad but non-free variety need not remain broad after adding new equations obtained from the transformation rules of $j$ under the action of $\PGL_2(\Q)$. So, if a variety is not free, we would have to analyse compatibility with $\PGL_2(\Q)$-relations and corresponding modular relations. Instead, we simply require that no $\PGL_2(\Q)$-relation and no modular relation hold on our variety, and no coordinate be constant, which is exactly freeness.

Thus, Modular Existential Closedness (MEC), that is, \Cref{thm: EC full}, can be seen as a dual statement to the Modular Schanuel Conjecture (MSC). In the terminology of arithmetic geometry, MSC is an unlikely intersections problem -- if two geometric objects (a variety and the graph of $j$) intersect when they are not supposed to, there must be a special reason for that. MEC is the matching likely intersections problem -- two geometric objects do intersect when they are supposed to.

We must note, however, that broadness and freeness are by no means necessary conditions for a system to have a solution. For instance, the system \[ j(z_1) = z_2-\iu+1728,~ j(z_2) = z_1-\iu+1728,~ j(z_1+z_2) = 287496 \]
is neither broad nor free, but it has a solution, namely, $z_1=z_2= \iu$. Of course, the coefficients in this system have been carefully chosen to produce a solution. As explained above, for general systems of equations we do not expect solutions or, at least, we do not expect a Zariski dense set of solutions in the absence of broadness and freeness. A problem known as \textit{Strong Modular Existential Closedness} asks when varieties contain points from the graph of $j$ that are generic (in the Zariski topology) over a finitely generated field of definition. In this case it is possible to show, assuming MSC (possibly a uniform version), that some form of broadness is necessary for the existence of generic solutions or for the existence of a Zariski dense set of (not necessarily generic) solutions. See \cite{Eterovic-Generic,Aslanyan-ECandZPsurvey} for details.

\subsection{Restricting $j$ to $\uh$}\label{subsec: extend to h counterex}

Modular Existential Closedness was first proposed in \cite{Aslanyan-Kirby-blurred-j} as a conjecture.

\begin{conjecture}[{\cite[Conj.~1.2]{Aslanyan-Kirby-blurred-j}}]\label{conj: MEC for upper half}
     Let $V \seq \C^n \times \C^n$ be a broad and free variety such that $V \cap (\uh^n \times \C^n) \neq \emptyset$. Then there are $z_1, \ldots, z_n \in \uh$ such that $(z_1, \ldots, z_n, j(z_1), \ldots, j(z_n)) \in V$.
\end{conjecture}

However, this conjecture is false.\footnote{This was spotted by ChatGPT 5.6 Sol.} Indeed, consider the system of equations
\[  z_1 = z_3+ \iu R,~ z_2 = z_4 + \iu R,~ j(z_1) j(z_2) = 1  \]
where $R>0$ is a real number. This is broad and free. Since $z_3$ and $z_4$ must be in $\uh$, the imaginary parts of $z_1$ and $z_2$ are greater than $R$. Since $|j(z)| \to \infty$ as $\im z \to + \infty$, if $R$ is large enough, $|j(z_1)|> 1$ and $|j(z_2)|>1$, so we cannot have $j(z_1) j(z_2) = 1$. Of course, letting $z_1, \ldots,z_4$ range over $\h$ eliminates this issue. 

In \cite[Conj.~3.5]{Eterovic-Generic}, it is required that $V \cap (\uh^n \times \C^n)$ be Zariski dense in $V$, with the extension of $j$ suggested as an alternative in the discussion following the conjecture. However, when  $V \cap (\uh^n \times \C^n)$ is non-empty, it is relatively Euclidean open and hence Zariski dense in $V$, so the requirement of Zariski density does not make any difference. Therefore, extending $j$ to the lower half-plane is essential.

The above example is actually \textit{strongly broad}, i.e.\ the dimension of the projection to any non-empty set of coordinates $(\xbar_I, \ybar_I)$ is strictly larger than $|I|$. For just broad and free varieties it is possible to find a counterexample already for $n=2$, e.g.\ the system $z_2=z_1+\iu R, z_2j(z_2)=1$ has no solution in $\uh^2$ for large enough $R$.

As we pointed out above, the extension of $j$ to $\h$ is canonical, so it is natural to consider MEC for this extended version. However, one may still ask which broad and free varieties $V$ with $V \cap (\uh^n \times \C^n) \neq \emptyset$ intersect $\gr_+(j)$ -- the graph of $j$ restricted to $\uh$. \Cref{thm: EC full} immediately implies that for any broad and free $V$, the image of $V$ under a  map of the form $x_i \mapsto \varepsilon_i x_i$ with suitable choices of $\varepsilon_i \in \{ \pm 1 \}$ meets $\gr_+(j)$. Further, $V \cap \gr_+(j) \neq \emptyset$ when $V$ has dominant projection to the first $n$ coordinates or when $V$ splits as a product of a real M\"obius subvariety (defined by $\PSL_2(\R)$-equations) and some variety in $\C^n$, due to \cite{eterovic-herrero} and \cite{Gallinaro-raising-to-powers}, respectively. Our proof establishes certain positive results generalising both of these cases. Indeed, we prove in \Cref{thm: weak MEC for upper half} that if we can choose a convenient transcendence basis of coordinate functions on $V$, then $V \cap \gr_+(j) \neq \emptyset$ under some semi-algebraic conditions. 

However, the assumptions of this proposition are quite restrictive. While broadness and freeness are not necessary conditions for $V \cap \gr(j)$ to be non-empty, the above discussion shows that they are natural assumptions matching MSC, and the question can often be reduced to a broad and free situation. On the other hand, the conditions of \Cref{thm: weak MEC for upper half} exclude a large class of varieties $V$ that do intersect $\gr_+(j)$, including those varieties whose projection to the first $n$ coordinates is not defined over $\R$. Finding reasonably general conditions which would guarantee non-empty intersection with $\gr_+(j)$ seems to be a problem of quite a different nature that would possibly require non-semi-algebraic conditions (not unlike understanding when a variety would meet the graph of real $\exp$ as opposed to complex $\exp$). While this is only a speculation, it is supported by \Cref{Prop: no semialg condition for upper half}, which states that excluding countably many semi-algebraic relations (i.e.\ a semi-algebraic version of freeness) can never guarantee a non-empty intersection with $\gr_+(j)$. More precisely, we show that for any countable collection $\Sigma$ of semi-algebraic sets of positive codimension in $\C^8$, there is a broad and free variety that meets $\uh^4 \times \C^4$, is contained in no member of $\Sigma$, and is disjoint from $\gr_+(j)$. See \S\ref{sec: j restricted to upper half} for further details.

\subsection{Overview of Existential Closedness}\label{subsec: Overview of EC}

In the early 2000s, in his work on the model theory of complex exponentiation \cite{Zilb-pseudoexp}, Zilber proposed a conjecture on the solvability of systems of exponential equations, which is known as \emph{Exponential (Algebraic) Closedness}  or \emph{Existential Closedness} for $\exp$. A substantial body of work has emerged around Exponential Closedness since then \cite{marker-remark-pseudoexp,brown-masser,daquino--fornasiero-terzo,Gallinaro-exp-sums-trop,Aslanyan-Kirby-Mantova,mantova-masser,Gallinaro-raising-to-powers}, although the full conjecture  remains open. 

Since the late 2010s, analogous problems and conjectures, seeking natural sufficient conditions for systems involving transcendental functions to have complex solutions, have been studied for other classical transcendental functions, notably the modular $j$-function. These are collectively referred to as \emph{Existential Closedness} problems/conjectures (or \textit{EC} for short), with MEC as an instance of EC. These statements can be viewed as a complex analytic counterpart to Hilbert's Nullstellensatz which is about the solvability of systems of polynomial equations in the complex numbers (or any algebraically closed field). 

In \cite{eterovic-herrero}, Eterovi\'c--Herrero proved that if a variety $V \seq \C^{2n}$ has dominant projection to the first $n$ coordinates then $V \cap \gr_+(j) \neq \emptyset$, obtaining the first partial result towards MEC. A differential analogue of MEC was posed as a conjecture by Aslanyan in \cite{Aslanyan-adequate-predim} (where the notions of broadness and freeness were introduced) and proved by Aslanyan--Eterovi\'c--Kirby in \cite{Aslanyan-Eterovic-Kirby-Diff-EC-j}. A weak version of MEC for ``blurrings'' of the $j$-function was established by Aslanyan--Kirby in \cite{Aslanyan-Kirby-blurred-j}, where the full MEC statement was formally proposed as a conjecture (see \S\ref{subsec: extend to h counterex}). Gallinaro obtained another special case of MEC for ``raising to powers'' in \cite{Gallinaro-thesis,Gallinaro-raising-to-powers}, which is an important source of inspiration for the current paper and is discussed in more detail in \S\ref{subsec: n=2 case}. Some of these results were later extended to Shimura varieties by Eterovi\'c--Zhao \cite{eterovic-zhao}. A special case of a version of MEC involving the $j$-function and its derivatives was established by Aslanyan--Eterovi\'c--Mantova in \cite{aslanyan-eterovic-mantova-MECD1}.

\Cref{thm: EC full} settles the MEC conjecture in full for the extended $j$-function. 

Let us now explain the terminology. In model theory, a model $M$ of a first-order theory $T$ is said to be \textit{existentially closed} if for any quantifier-free formula $\varphi(\xbar)$ (roughly, a system of equations and inequations) with parameters from $M$, whenever $\varphi$ has a realisation (solution) in some larger model $M'$ of $T$, it already has a solution in $M$. In other words, a system of equations and inequations always has a solution in $M$, provided that it is consistent. In this case consistency means that a solution exists in a possibly larger model. In the context of Zilber's analytic and pseudo-analytic structures, such a structure (e.g.\ the complex field equipped with a transcendental analytic function) is not usually existentially closed in the first-order sense, but it is often existentially closed with respect to so-called \textit{strong} extensions, within a suitable class of structures. This idea has been formalised in some cases, e.g.\ for exponential functions \cite{Zilb-pseudoexp,Bays-Kirby-exp}, but not for $j$. One difficulty is that $j$ is not defined on the whole complex plane, which makes it harder to treat it model-theoretically (e.g.\ if we equip $\C$ with the graph of $j$ then the reals and integers would both be definable). However, the philosophy of existential closedness -- solutions exist unless there is a natural reason for them not to exist -- is still valid, hence these analytic solvability problems are referred to as Existential Closedness problems/conjectures.

The aforementioned obstacle to a model-theoretic treatment of $j$ disappears when we work with its differential equation in a differentially closed field. The first-order theory of such a structure is tame and includes the differential variants of MSC and MEC as axiom schemes, and the first-order existential closedness in strong extensions also holds. See \cite{Aslanyan-adequate-predim} for details.

\subsection{Overview of the proof of \Cref{thm: EC full} and the approximation theorem}\label{subsec: overview of proof} 

The main idea of the proof is to approximate the given system of equations by a simpler system which is guaranteed to have solutions, and to apply Rouch\'e's theorem from complex analysis to transfer these solutions back to the original system.

Broadness allows us to choose and work with a triangular system of equations. By applying suitable modular transformations to the independent input variables in the system, we get a sequence of systems where these independent variables move towards a real point. Some of the inputs of $j$ in the system may then approach non-real points, where $j$ is defined and holomorphic. Others may approach real points, where $j$ is undefined. For these inputs, we approximate the corresponding function with one of the form $j(hz)$ where $h$ is a real M\"obius transformation. This is possible due to \Cref{thm: density of analytic orbits general version}, which is discussed separately below. While doing this, we employ the Ax--Schanuel theorem to avoid potential degeneracies.

We end up with a limiting system of equations. Ax--Schanuel then comes into play again with a different role to show that this limiting system has an isolated solution. Rouch\'e's theorem is then applied to show that the systems of equations in the constructed sequence eventually have solutions, which we then transfer back to the original system.

In this argument, the values of the independent variables can be chosen to be in $\uh$. However, the other variables that can be expressed in terms of the independent ones may take values in $\lh$. This is where we use the domain $\h$ for $j$.

The following approximation theorem, which may be of independent interest in homogeneous dynamics, plays a crucial role in the proof.

\begin{theorem}[Modular approximation theorem]
\label{thm: density of analytic orbits general version}
Let $F_1, \ldots, F_N$ be non-constant holomorphic functions of variables $z_1, \ldots, z_m$ defined on some open domain $\Omega \seq \C^m$ with $\Omega \cap \R^m \neq \emptyset$. Let $B \seq \Omega \cap \R^m$ be a non-empty open box. Assume the functions $z_1, \ldots, z_m, F_1, \ldots, F_N$ are in distinct $\PGL_2(\Q)$-orbits, and $F_k(B) \seq \R$ for all $k$. Then there is a map $\sigma \colon \{ 1, \ldots, N \} \to \{ 1, \ldots, m \} $ with the following property: for any $\rbar:=(r_1, \ldots, r_m) \in B$ and for any $h_k \in \PGL_2(\R),\, 1 \leq k \leq N,$ there are sequences of matrices $(\gamma_{k,s})_s, (\zeta_{l,s})_s \seq \PGL_2(\Z),\, 1\leq k \leq N,\, 1\leq l \leq m,$ such that  
    \begin{itemize}
        \item for all $l$, as $s \to \infty$, we have $\zeta_{l,s} z \to r_l$ locally uniformly for $z \in \h$;

        \item for all $k$, as $s \to \infty$, we have $\gamma_{k,s} F_k(\zeta_{1,s} z_1, \ldots, \zeta_{m,s} z_m ) \to h_k z_{\sigma(k)}$  locally uniformly on $\h^m$.
    \end{itemize}
\end{theorem}

\begin{remark}
Since $\zeta_{l,s}z \to r_l$ locally uniformly on $\h$, for any compact set $K \seq \h^m$, for $\zbar \in K$ the tuple $(\zeta_{1,s} z_1, \ldots, \zeta_{m,s} z_m )$ will eventually land in a neighbourhood of $\rbar$ contained in $\Omega$. So, the function $F_k(\zeta_{1,s} z_1, \ldots, \zeta_{m,s} z_m )$ is well defined on $K$ for all sufficiently large $s$.
\end{remark}

For instance, given any pair of real M\"obius transformations, \Cref{thm: density of analytic orbits general version} applied to the functions $z^2$ and $z^3$ gives sequences of matrices $\gamma_{1,s}, \gamma_{2,s}, \zeta_s \in \PGL_2(\Z)$ such that the map $(\gamma_{1,s} ((\zeta_s z)^2), \gamma_{2,s} ((\zeta_s z)^3))$ converges locally uniformly to that pair.

We use Ratner's orbit closure and equidistribution theorems in powers of $\PSL_2(\R)$ to prove \Cref{thm: density of analytic orbits general version}. For example, given a single function $F(z)$ holomorphic around a suitably generic (outside a meagre set) real point $r$ and mapping a real neighbourhood of $r$ to $\R$, we approximate $F$ by a M\"obius transformation $g \in \PGL_2(\R)$ up to order 2. By replacing $F$ with $-F$ if necessary, we may assume $g \in \PSL_2(\R)$. We consider the orbit of $g$ traced out by the horocycle flow $\begin{pmatrix} 1 & t \\ 0 & 1 \end{pmatrix},\, t \in \R$. Since this flow moves any point in $\C$ towards $\infty$, we conjugate it by an appropriate M\"obius transformation and get a flow moving all points towards $r$. Ratner's orbit closure and equidistribution theorems, applied to this conjugated flow, provide sequences $\gamma_s, \zeta_s$ such that $\gamma_s \circ g \circ \zeta_s$ approaches any prescribed M\"obius transformation $h$. We then use the second-order approximation of $F$ by $g$ to establish that $\gamma_s \circ F \circ \zeta_s$ also approaches $h$. Here the second-order approximation is essential: it gives a third-order error for $F-g$ which is amplified by a second-order factor arising from modular transformations, leaving a first-order error tending to 0. Finally, we use a diagonal argument to pass from a suitably generic $r$ to an arbitrary real point.

When dealing with a function $F$ of several variables, we use modular transformations to move all the inputs to a real point. In fact, we pick a suitable variable (which is what the map $\sigma$ does in \Cref{thm: density of analytic orbits general version}) and ensure it approaches the appropriate real value more slowly than the other inputs. This allows us to treat $F$ as a function depending only on the chosen variable, so we apply the one-variable argument described above. 

For several functions we use a similar argument in powers of $\PSL_2(\R)$.

We provide more detailed overviews of the proofs of \Cref{thm: EC full,thm: density of analytic orbits general version} in \S\ref{sec: proof of main thm} and \S\ref{sec: dynamics} before presenting the full arguments. 

The two major ingredients in the proof of \Cref{thm: EC full} are the Ax--Schanuel theorem \cite{Pila-Tsim-Ax-j} and Ratner's orbit closure/equidistribution theorems \cite{Ratner1991OrbitClosure}. We expect our proof to generalise, with some further work, to some other settings where these two ingredients are available (and can produce an analogue of \Cref{thm: density of analytic orbits general version}), e.g.\ automorphic functions of cofinite Fuchsian groups (including non-arithmetic ones) and uniformising maps of Shimura varieties. Note that Ax--Schanuel in these cases is due to \cite{BlazquezSanzCasaleFreitagNagloo2026,MokPilaTsimerman2019}.

\subsection{Some applications of MEC and the approximation theorem.\protect\footnotemark\nopunct}\label{subsec: applications}

\subsubsection{Strong MEC}
In \cite{Eterovic-Generic}, Eterovi\'c explores the problem of finding points in $V \cap \gr_+(j)$ which are generic in $V$ over a given finitely generated field. He shows that, assuming a version of Modular Schanuel for $j$ together with its derivatives and an arithmetic-geometric conjecture known as Modular Zilber--Pink, given a broad and free variety $V$ defined over some finitely generated field $K \seq \C$, if $V\cap \gr_+(j)$ is Zariski dense in $V$ then it contains a point whose transcendence degree over $K$ is equal to $\dim V$ (i.e.\ it is generic in $V$ over $K$). This is \cite[Thm.~1.1]{Eterovic-Generic}. The proof shows that the same statement holds if we replace $\gr_+(j)$ with $\gr(j)$. Alternatively, given such a variety $V$, if $V\cap \gr(j)$ is Zariski dense in $V$ then $V$ contains a Zariski dense set of points from the graph of $j$ restricted to some product of half-planes. For each input coordinate whose half-plane is $\lh$, we can apply the transformation $x_i \mapsto -x_i$, and get a variety $V'$ whose intersection with $\gr_+(j)$ is Zariski dense in $V'$. Then we can apply Eterovi\'c's result to $V'$ and find a generic point, and move it back to $V$.

\footnotetext{I have discussed the applications in \S\S\ref{subsubsec:appl-cycles}--\ref{subsubsec:appl-orbit-closure} with ChatGPT, which suggested the example of 3-cycles at the end of \S\ref{subsubsec:appl-cycles}. The other examples, statements, and arguments are mine.}

Thus, \Cref{cor: zar dense ec} implies that, assuming Modular Schanuel with Derivatives and Modular Zilber--Pink, every broad and free variety contains a point from $\gr(j)$ which is generic over a finitely generated field of definition of $V$.

Eterovi\'c also establishes an unconditional result. There is a countable algebraically closed field $C_j$ which is the closure of $\Q$ under some closure operator respecting the $j$-function (see \cite{Aslanyan-Eterovic-Kirby-closure-op} and \cite[\S4.1]{Eterovic-Generic}). A broad variety $V\seq \C^n \times \C^n$ is said to have \textit{no $C_j$-factors} if, for any non-empty $I \seq \{ 1, \ldots, n \}$, the projection of $V$ to $(\xbar_I, \ybar_I)$ either is not defined over $C_j$ or has dimension $>|I|$. 

Theorem 1.2 of \cite{Eterovic-Generic} then asserts that if $V$, defined over a finitely generated field $K$, is broad and free with no $C_j$-factors and $V \cap \gr_+(j)$ is Zariski dense in $V$ then it contains a generic point of $V$ over $K$. As above, we can replace $\gr_+(j)$ with $\gr(j)$. Thus, \Cref{cor: zar dense ec} implies that any broad and free variety with no $C_j$-factors (in particular, any free and strongly broad variety) defined over $K$ contains a point from $\gr(j)$ which is generic over $K$.

\subsubsection{Cycles of holomorphic functions}\label{subsubsec:appl-cycles}
MEC can also be applied to study periodic orbits of functions $f$ that can be expressed in terms of $j$. Here, we consider functions of the form $f_c(z) = j(z)+c$ for constant $c \in \C$. Given $\tau \in \h$, the orbit $\tau, f_c(\tau), f_c(f_c(\tau)), \ldots$ (assuming all of these values are defined) is called \textit{periodic} (or a \textit{cycle}) if for some $n>0$ we have $f_c^{\circ n}(\tau) = \tau$. The smallest such $n$ is the \textit{period} of the cycle, which is then referred to as an \textit{$n$-cycle}.

Given $n \geq 2$, applying MEC to the variety \[ x_2 = y_1+c, \ldots, x_n = y_{n-1}+c, x_1 = y_n+c, \]
we get a Zariski dense intersection with $\gr(j)$. Projecting to $\xbar$ and removing all the points satisfying  a relation $x_i = x_{i'}$ for some $1\leq i < i' \leq n$, the remaining set will be exactly the set of $n$-cycles of $f_c$. Thus, the set of $n$-cycles is Zariski dense in $\C^n$. Here we only need MEC for varieties with dominant projection to the first $n$ coordinates. Similar examples are considered in \cite[\S6.2]{Aslanyan-Eterovic-Kirby-closure-op} and \cite[\S4.6]{Eterovic-Generic}.

Full MEC allows us to impose an algebraic relation on a cycle. Namely, given a $\PGL_2(\Q)$-free hypersurface $L \seq \C^n$ (i.e.\ no constant coordinates and no $\PGL_2(\Q)$-relations), the set \[ \mathscr{C}:= \{ (c, \taubar) \in \C\times L: \taubar= (\tau, f_c(\tau), \ldots, f_c^{\circ (n-1)}(\tau)) \text { is an $n$-cycle of } f_c \} \]
is Zariski dense in $\C\times L$. In particular, there are infinitely many $c\in \C$ for which there is $\tau \in \h$ of exact period $n$ such that $(\tau, f_c(\tau), \ldots, f_c^{\circ (n-1)}(\tau)) \in L$.

To see this, consider the algebraic variety $V \seq \C^{n+1}\times \C^{n+1}$ (with coordinates $(x_0, \ldots, x_n, y_0, \ldots,y_n)$) defined by the equations
\[ x_2 = y_1+y_0, \ldots, ~x_n = y_{n-1}+y_0,~ x_1 = y_n+y_0,~ P(x_1, \ldots,x_n) = 0, \]
where $P=0$ is a polynomial equation defining $L$.

Clearly, $V$ is broad, and it is also free because $L$ is $\PGL_2(\Q)$-free. By \Cref{cor: zar dense ec}, $V\cap \gr(j)$ is Zariski dense in $V$. If we remove all the points from this set satisfying  a relation $x_i = x_{i'}$ for some $1\leq i < i' \leq n$, then the remaining set is still Zariski dense in $V$. The projection of this set to $(y_0, x_1, \ldots, x_n)$ is exactly $\mathscr{C}$. So the latter is Zariski dense in the projection of $V$ which is $\C\times L$.

Note that the $\PGL_2(\Q)$-freeness of $L$ is automatic when $P$ is an irreducible polynomial depending on at least three variables. For example, for $P(x_1,x_2,x_3) = x_1-2x_2+x_3$ we conclude that the set of triples $(c,d,\tau)$ for which $(\tau, \tau+d, \tau+2d)$ is an exact 3-cycle of $f_c$ is Zariski dense in $\C^3$.

\subsubsection{An orbit closure theorem}\label{subsubsec:appl-orbit-closure}

    \Cref{thm: density of analytic orbits general version} can be used to characterise the closures of certain orbits of analytic functions under the action of $\PGL_2(\Z)$. Let $\Omega \seq \C$ be a disc centred at a real point and $\tilde{\Omega}$ a domain containing $\overline{\Omega}$. Let $F_1, \ldots, F_N \colon \tilde{\Omega} \to \C$ be holomorphic functions such that $z,F_1(z),\ldots,F_N(z)$ are in distinct $\PGL_2(\Q)$-orbits, $F_k(\overline{\Omega} \cap \R) \seq \R$ for all $k$ and $F_k'(r) \neq 0$ for all $k$ and $r \in \overline{\Omega} \cap \R$. Let $D \seq \h$ be a bounded domain with $\Dbar \seq \h$. Define     
    \[ \mathcal{O}(\Fbar):= \{ (\gamma_k \circ F_k \circ \zeta)_{k=1}^N : \gamma_k, \zeta \in \PGL_2(\Z),\, \zeta(D) \seq \Omega, \, F_k(\zeta(\Dbar)) \seq \h \text{ for all }k\}.  \]
    Note that under these conditions $\gamma_k \circ F_k \circ \zeta$ is holomorphic on $D$ and has no poles on $\Dbar$. Write $\overline{\mathcal{O}(\Fbar)}$ for the closure of $\mathcal{O}(\Fbar)$ with respect to the uniform norm on the space of continuous functions $C(\Dbar, \h^N)$. Then  
    \[  \overline{\mathcal{O}(\Fbar)}  = \mathcal{O}(\Fbar) \cup \PGL_2(\R)^N.\]
    Here, we think of $\PGL_2(\R)$ as the family of real M\"obius functions restricted to $\Dbar$. If we considered $\mathcal{O}(\Fbar)$ inside $C(\Dbar, \C^N)$, then we would also get real-valued constant functions in the closure.

    It is also possible to characterise the orbit closure without assuming the functions are $\PGL_2(\Q)$-independent. Since this is merely an application of \Cref{thm: density of analytic orbits general version}, rather than a central part of the paper, we do not consider the general case and limit ourselves to the following observation for a single function $F$:
    \[ \overline{\mathcal{O}(F)} = \begin{cases} \mathcal{O}(F) & \text{ if } F \text{ is a } \PGL_2(\Q)\text{-transformation,}\\
    \mathcal{O}(F) \cup \PGL_2(\R) & \text{ otherwise.}
    \end{cases}   \]
    
    We postpone the proofs of these statements to \S\ref{subsec: orbit closure}.

\subsection{Structure of the paper}

\begin{itemize}
    \item Section~\ref{sec: prelim} contains some preliminaries that we need throughout the paper. 

    \item In Section~\ref{sec: special cases and examples} we demonstrate our strategy on several examples to give the reader some intuition for the main ideas of the proof. In particular, we provide a detailed proof sketch for MEC for $n=2$ (modulo \Cref{thm: density of analytic orbits general version}), which is already a new result. We also look at an example where we closely follow our general strategy for the proof of \Cref{thm: EC full} and demonstrate the key ideas of that proof in an explicit context.

    \item To keep the natural flow of the paper, we then immediately proceed to the proof of \Cref{thm: EC full} in Section \ref{sec: proof of main thm}, assuming \Cref{thm: density of analytic orbits general version}. \Cref{thm: secondary thm density} is proved at the end of that section.

    \item \Cref{sec: j restricted to upper half} provides a brief discussion of Modular Existential Closedness for $j$ restricted to the upper half-plane.

    \item \Cref{thm: density of analytic orbits general version} is quite different in nature from the rest of the paper, and is proved in Section~\ref{sec: dynamics}. This section does not depend on the earlier ones, so a reader interested specifically in \Cref{thm: density of analytic orbits general version} may proceed directly to  \Cref{sec: dynamics}.

    \item In Appendix \ref{appendix} we provide domain colouring plots for two complex functions illustrating the distribution of their zeros.
\end{itemize}

\subsection*{AI statement}
While working on this project, I have used Gemini 3.1 Pro and ChatGPT 5.6 Sol/6 Astra  (alongside \cite[\S5.4]{Gallinaro-thesis} and \cite{Morris--Ratner}) to learn about and develop some intuition for homogeneous dynamics, particularly Ratner's theorems. The example in \S\ref{subsec: extend to h counterex}, the proofs of Lemmas \ref{lem: broad pregeom} and \ref{lem: connected subgps of SL_2(R)}, and some technical steps in the proof of \Cref{prop: Gamma g_0...g_N closure} are due to ChatGPT. Note that Lemmas \ref{lem: broad pregeom} and \ref{lem: connected subgps of SL_2(R)} are known results (at least to experts in the area), but I included proofs for completeness and for lack of precise references. I have discussed the applications in \S\S\ref{subsubsec:appl-cycles}--\ref{subsubsec:appl-orbit-closure} with ChatGPT and also used it to find several references throughout the paper. \Cref{fig:fundamental-domain,fig:j,fig:key,fig:j(z)-j(z^2)-1,fig:z^3+j(z)+...} and \Cref{tab:recursive-notation} were generated by ChatGPT 6 Astra. 

The paper was entirely written by me (including the aforementioned proofs based on AI-generated arguments). I also used ChatGPT 6 Astra and Claude Opus 5.5 to review earlier drafts, and implemented the suggested changes myself. These included a simpler argument for the last step of the proof of \Cref{claim: approximation thm}.

The footnotes throughout the manuscript provide further details on AI use where appropriate. Apart from the AI contributions acknowledged above, every aspect of this work -- including the mathematical ideas, overall strategy, proofs, examples, and exposition -- is human-generated. 

I am a human...

\section{Preliminaries}\label{sec: prelim}

\subsection{Notation and conventions}

\begin{itemize}
    \item We use upright $\iu$ for the imaginary unit and italic $i$ for indices.

    \item We use boldface letters (e.g.\ $\xbar, \ybar, \zbar, \wbar, \Fbar$) to denote tuples/vectors.

    \item For a complex number $z$ we write $\bar{z}$ for its complex conjugate. 
    
    \item For a subset $X$ of a topological space, $\overline{X}$ denotes its closure.

    \item The algebraic closure of a field $K$ is denoted $K^{\alg}$.
    
    \item All varieties are complex algebraic varieties, identified with their $\C$-points, and assumed to be irreducible.

    \item For a field extension $L\seq K$, we write $\td_LK$ for the transcendence degree.

    \item When working with algebraic functions or equations, we usually use $\xbar$ and $\ybar$ to denote coordinates and variables. When working with complex analytic functions, we use $\zbar$ and $\wbar$ instead.

    \item For a tuple $\pbar = (p_i)_{i \in I}$ and $K \seq I$, we write $\pbar_K$ for the tuple $(p_i)_{i \in K}$, i.e.\ the projection of $\pbar$ to $K$, with $\pbar_{\emptyset}$ being the empty tuple.
\end{itemize}

\subsection{Linear fractional transformations}
Let $\uh$ denote the complex upper half-plane $\{z\in \C :\im z>0\}$. The group $\mathrm{GL}_2^+(\R)$ of $2\times 2$ real matrices with positive determinant acts on $\uh$ via linear fractional transformations:
\[gz:=\frac{az+b}{cz+d} \quad \text{for} \quad g=\begin{pmatrix}a & b\\ c & d\end{pmatrix}\in\mathrm{GL}_2^+(\R).\]
However, this action is not faithful, so instead we work with projective linear groups (obtained by quotienting a linear group by its centre). Thus $G:= \PSL_2(\R) := \GL_2^+(\R) / \R^{\times}I = \SL_2(\R)/\{ \pm I \}$ acts faithfully on $\uh$ so each element of this group corresponds uniquely to a M\"obius transformation. 

Note that here we identify a matrix class in $G$ containing some matrix $g\in \GL_2^+(\R)$ with the normalised representative $\frac{1}{\sqrt{\lvert \det(g) \rvert}} g$ which has determinant $1$ and, up to $\pm I$, does not depend on the choice of $g$.

The \textit{modular group} $\Gamma :=\PSL_2(\Z) := \SL_2(\Z) / \{ \pm I \}$ is a discrete subgroup of $G$. Again, we think of the elements of this group as matrices with determinant $1$.  

We also need to consider the action of $\PGL_2^+(\Q)$ on $\uh$. However, given a matrix $g \in \GL_2^+(\Q)$, the normalised matrix $\frac{1}{\sqrt{\lvert \det(g) \rvert}} g$ may have irrational entries. Nevertheless, the map $g \mapsto \frac{1}{\sqrt{\lvert \det(g) \rvert}} g$ gives an embedding of $\PGL_2^+(\Q)$ into $G$. We write $\mathcal{C}$ for the image of this embedding and identify $\PGL_2^+(\Q)$ with $\mathcal{C}$. Note that $\mathcal{C}$ is the \textit{commensurator} of $\Gamma$, i.e.\ the subgroup of all $g \in G$ such that the intersection of $g^{-1}\Gamma g$ and $\Gamma$ has finite index in both groups (see \cite[Lem.~4.5.1]{Miyake2006}).

Now, let $\lh:= \{ z \in \C: \im z < 0 \}$, and define $\h := \uh \cup \lh = \C \setminus \R$.

The group $\hat{G}:= \PGL_2(\R):=\GL_2(\R)/\R^{\times}I$, whose elements can be represented by matrices with determinant $\pm 1$, acts on $\h$. Matrices with positive determinant preserve $\uh$ and $\lh$, while those with negative determinant swap them. As before, we identify a matrix class in $\hat{G}$ containing some matrix $g\in \GL_2(\R)$ with the representative $\frac{1}{\sqrt{\lvert \det(g) \rvert}} g$ which has determinant $\pm 1$ and is unique up to sign.

Similarly, the group $\hat{\Gamma}:=\PGL_2(\Z)$ acts on $\h$. If $g \in \GL_2(\Z)$, then $\frac{1}{\sqrt{\lvert \det(g) \rvert}} g$ still has integral entries (in fact, it is equal to $g$), for $\det (g) = \pm 1$. 

\iftrue
\begin{figure}[htbp]
  \centering
  \resizebox{\linewidth}{!}{%
    \input{gl2z-domains-tikz.tex}%
  }
   \caption[Fundamental domains]{The action of $\PGL_2(\Z)$ on $\h$. Every shaded hyperbolic triangle, including $\mathcal{F}_+$ and $\mathcal{F}_-$, is a fundamental domain.\protect\footnotemark}\label{fig:fundamental-domain}
\end{figure}
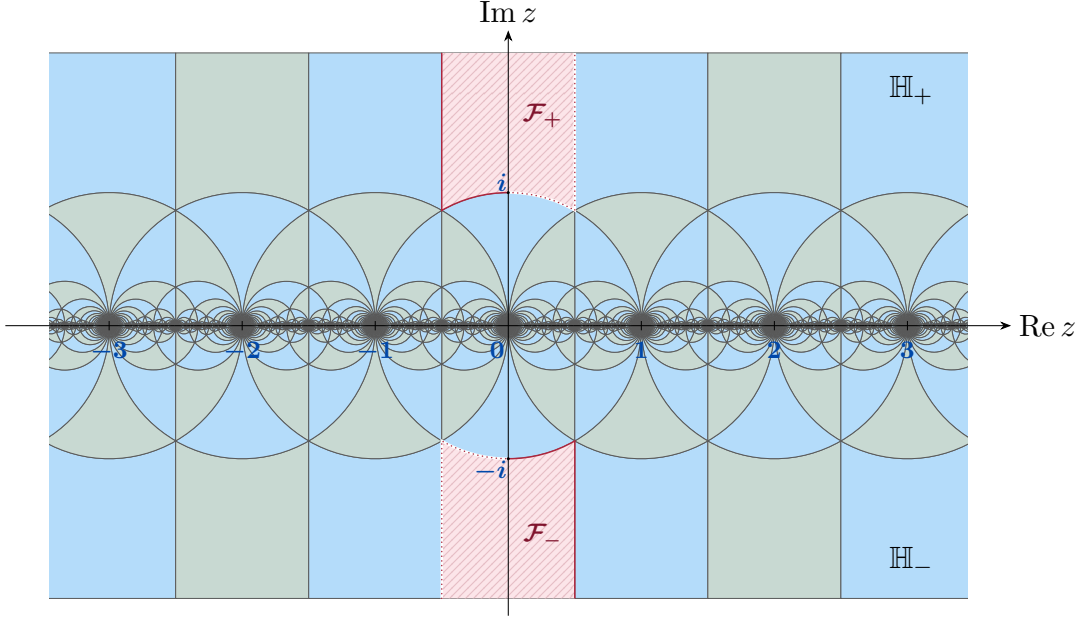
\fi

The group $\PGL_2(\Q)$ also acts on $\h$. The map $g \mapsto \frac{1}{\sqrt{\lvert \det(g) \rvert}} g$ gives an embedding of $\PGL_2(\Q)$ into $\hat{G}$. We write $\hat{\mathcal{C}}$ for the image of this embedding, and identify $\PGL_2(\Q)$ with $\hat{\mathcal{C}}$. In this case, $\hat{\mathcal{C}}$ is the commensurator of $\hat{\Gamma}$.

\footnotetext{This figure was generated by ChatGPT 6 Astra, inspired by \cite[Fig.~1]{aslanyan-eterovic-mantova-MECD1}.}

The notation $G, \Gamma, \mathcal{C},$ etc. will mostly be used in \S\ref{sec: dynamics}. In the rest of the paper we will use explicit notation such as $\PSL_2(\R), \PSL_2(\Z)$, etc.

We refer the reader to \cite[Ch.~1]{Miyake2006} for an in-depth discussion of Fuchsian groups and their action on the upper half-plane.

\subsection{The $j$-function}\label{subsec: j prelim}

The \emph{modular $j$-function} is defined as the unique function $j\colon \uh \to \C$ with the following properties:
\begin{itemize}
    \item $j$ is holomorphic on $\uh$;
    \item $j$ is $\PSL_2(\Z)$-invariant, i.e.\ for all $\gamma \in \PSL_2(\Z)$ we have $j(\gamma z) \equiv j(z)$;
    \item $j(\iu) = 1728, j\left(e^{\frac{2 \pi \iu }{3}}\right) = 0$;
    \item $j$ has a simple pole as $\im z \to + \infty$, i.e.\ the $q$-expansion below has a simple pole at $q=0$.
\end{itemize}

In particular, $j$ is 1-periodic. Its Fourier expansion (also known as the $q$-expansion) is of the form
\begin{equation}\label{eq:j-fourier-expansion}
j(z)=q^{-1}+744+\sum_{k=1}^{\infty}a_kq^k,\, \text{ with }q:=\exp(2\pi \iu z)  \text{ and }a_k\in \Z,
\end{equation}
which has a simple pole at $q=0$ as mentioned above. We refer the reader to \cite[Ch.~3, \S3 and Ch.~4, \S1]{Lang1987EllipticFunctions} for more details on the $j$-function.

As explained in the introduction, we need to extend $j$ to $\C \setminus \R$. We do so by defining $j(z) := \overline{j(\bar{z})}$ for $z$ in the lower half-plane $\lh$. Using \eqref{eq:j-fourier-expansion} we see that $j(z)=j(-z)$ whenever $z \in \lh$, hence this equality holds for any $z \in \h$. This shows that $j$, as a function on $\h$, is invariant under the action of $\PGL_2(\Z)$, which is generated by $\PSL_2(\Z)$ and $\begin{pmatrix} -1 & 0 \\ 0 & 1 \end{pmatrix}$. The $q$-expansion also implies that $j(z) \in \R$ whenever $z \in \iu\R^{\times}$.

\iftrue
 \begin{figure}[htbp]
    \centering 
        \vspace*{-10pt}
    \includegraphics[width=1\linewidth]{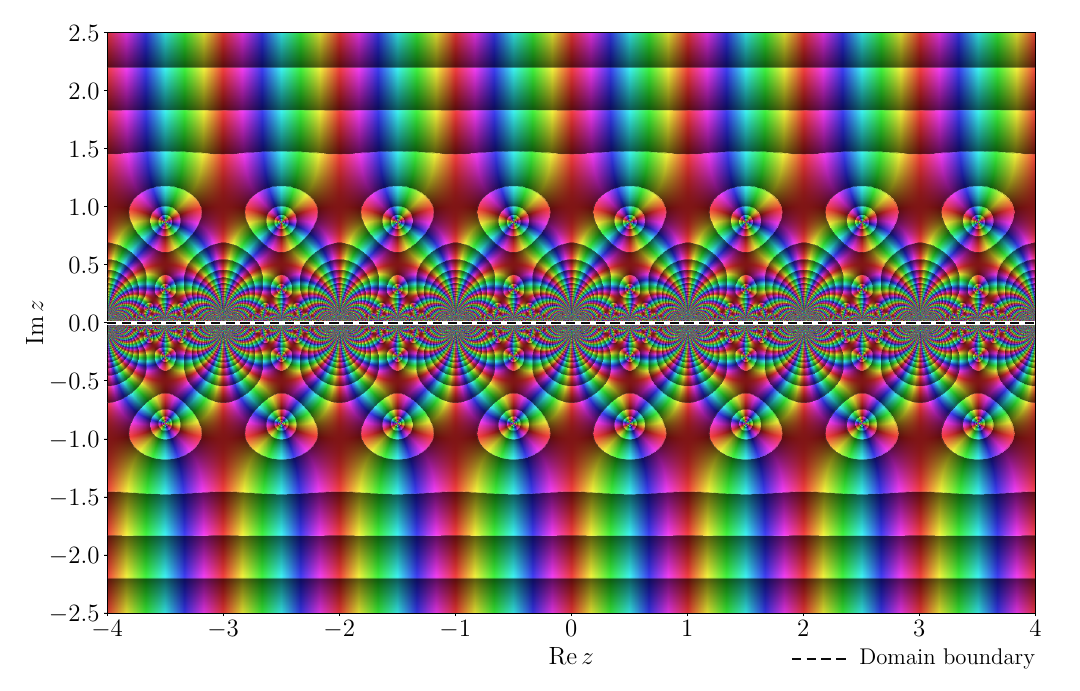} 
    \vspace*{-30pt}
    \caption{Visualisation of $j$ via domain colouring (see Appendix \ref{appendix})\protect\footnotemark}
    \label{fig:j}
\end{figure}
\fi
\footnotetext{Generated by ChatGPT 6 Astra.}

\subsection{Modular polynomials}

It is well known that there is a countable collection of irreducible polynomials $\Phi_N(X,Y) \in \Z[X,Y], N \in \Z_{>0}$, known as \emph{modular polynomials}, such that for any $z_1, z_2 \in \uh$ we have
\[ \exists g \in \PGL_2^+(\Q) : z_2= gz_1 \Longleftrightarrow \exists N \in \Z_{>0} :\Phi_N(j(z_1), j(z_2)) = 0. \]
This then implies that for any $z_1, z_2 \in \h$ we have
\begin{equation}\label{eqn: j GLQ vs modular}
    \exists g \in \PGL_2(\Q) : z_2= gz_1 \Longleftrightarrow \exists N \in \Z_{>0} : \Phi_N(j(z_1), j(z_2)) = 0.
\end{equation}

We refer the reader to \cite[Ch.~5, \S2]{Lang1987EllipticFunctions} for further details.

A subset of a field of characteristic 0 (which will be either $\C$ or a field of functions) is said to be $\PGL_2(\Q)$-independent (respectively, modularly independent) if for any two distinct elements $a, b$ from that subset no relation of the form $b=ga$ with $g \in \PGL_2(\Q)$ (respectively, a modular relation $\Phi_N(a,b)=0$ with $N\in \Z_{>0}$) holds.

\subsection{Ax--Schanuel}

The Ax--Schanuel theorem for the $j$-function, proved by Pila and Tsimerman in \cite{Pila-Tsim-Ax-j}, is used in the proof of \Cref{thm: EC full}.

\begin{fact}[Ax--Schanuel]\label{Ax-Schanuel functional}
    Let $\Omega \seq \C^m$ be some domain and let $f_1(\zbar), \ldots, f_n(\zbar) \colon \Omega \to \h$ be non-constant holomorphic functions such that no relation of the form $f_i = g f_k$ with $g \in \PGL_2(\Q)$ and $1\leq i < k \leq n$ holds. Then \[ \td_{\C} \C(f_1, \ldots, f_n, j(f_1), \ldots, j(f_n)) \geq n + \rk\left( \frac{\partial f_k}{\partial z_l} \right)_{k,l}. \]
\end{fact}

Note that the statement of Ax--Schanuel in \cite{Pila-Tsim-Ax-j} is only for $j$ defined on $\uh$ and the independence condition is about the values of $j$ being modularly independent rather than the inputs being $\PGL_2(\Q)$-independent. The above statement follows from \cite[Thm.~1.1 and \S2.6]{Pila-Tsim-Ax-j} by our extension of $j$ to $\h$ and by \eqref{eqn: j GLQ vs modular}. 

\subsection{Rouch\'e's theorem and the inverse function theorem}

Rouch\'e's theorem is the main complex analytic input in this paper.

\begin{fact}[Rouch\'e, see, e.g.,~{{\cite[Thm.~2, Ch.~IV, \S18.55]{Shabat}}}]
    \label{thm:rouche}
  Let $D \seq \C^m$ be a bounded domain with piecewise smooth boundary, and let $F,H \colon \Dbar \to \C^m$ be two continuous maps which are holomorphic on $D$. Assume the zeros of $F$ are isolated, and for all $\zbar \in \partial D$ we have $\| F(\zbar) \| > \| H(\zbar) \|$. Then $F$ and $F+H$ have the same number of zeros in $D$, counting multiplicities.
\end{fact}

\begin{corollary}\label{cor: Rouche}
    Let $D \seq \C^m$ be a bounded domain with piecewise smooth boundary, and let $F,F_s \colon \Dbar \to \C^m,\, s \in \Z_{>0},$ be continuous maps which are holomorphic on $D$. Assume the sequence $F_s$ converges uniformly to $F$ on $\Dbar$. Assume further that $F$ does not vanish on $\partial D$, has a zero in $D$, and all of its zeros are isolated. Then for all sufficiently large $s$, the map $F_s$ has a zero in $D$.
\end{corollary}
\begin{proof}
    Let $\delta := \min_{\zbar \in \partial D} \| F(\zbar) \|$. Since $F$ does not vanish on $\partial D$, we have $\delta>0$. Uniform convergence of $F_s$ to $F$ implies that for large enough $s$ we have \[ \| F_s(\zbar) - F(\zbar) \| < \delta \text{ for all } \zbar \in \Dbar. \]

    Since $F$ has a zero in $D$, by \Cref{thm:rouche}, $F_s = (F_s-F)+F$ also has a zero in $D$.
\end{proof}

In this paper when we refer to Rouch\'e's theorem, we actually use \Cref{cor: Rouche}.

We also need the classical inverse function theorem for holomorphic maps.

\begin{fact}[{\cite[Thm.~2.1, p.~45]{lee2024complex}}]\label{fact: inverse f thm}
    Let $V$ be a complex manifold of dimension $n$ and $F\colon V \to \C^n$ a holomorphic map. Let $\pbar \in V$ be a point where the Jacobian matrix has rank $n$. Then there is an open neighbourhood $U \seq V$ of $\pbar$ and a holomorphic inverse $H\colon F(U) \to U$ such that $F \circ H$ and $H \circ F$ are the identity maps on $F(U)$ and $U$, respectively.
\end{fact}

\section{Special cases and examples}\label{sec: special cases and examples}

In this section we look at three special cases of MEC to illustrate the main ideas of our proof of \Cref{thm: EC full} and develop some intuition before presenting the full argument in the next section. We follow the notation of \S\ref{sec: proof of main thm} to make the comparison between the special and general cases easier. In Appendix \ref{appendix} we illustrate the zeros and their distribution via domain colouring for two complex functions that we study below.

\subsection{MEC for $n=2$}\label{subsec: n=2 case}

In \cite[Ch.~5]{Gallinaro-thesis}, Gallinaro establishes MEC assuming $V$ is of the form $L \times W$ where $L$ is a M\"obius subvariety, i.e.\ defined by equations of the form $x_k = g x_i$ with $g \in \PSL_2(\R)$. Note that in this case we do not have to extend $j$ to the lower half-plane. We illustrate Gallinaro's proof with a simple example.

\begin{example}
    
Consider the equation $j(gz) = j(z)^2+1$ where $g \in \PSL_2(\R)$. Freeness means that $g \notin \PGL_2^+(\Q)$.  Pick any $u \in \C$ and let $\tau\in \uh$ be such that $j(\tau) = u$. There is $h \in \PSL_2(\R) \setminus \PGL^+_2(\Q)$ such that $j(h \tau) = u^2+1$.  Ratner's orbit closure theorem implies that $\Gamma g \Gamma$ is dense in $\PSL_2(\R)$ where $\Gamma:=\PSL_2(\Z)$. So let $\gamma_s, \zeta_s \in \Gamma$ be such that $\gamma_s g \zeta_s \to h$. Then the sequence of functions $j(\gamma_s g \zeta_s z) - j(\zeta_s z)^2 -1$ converges to $j(hz)- j(z)^2-1$ uniformly on some disc $D$ around $\tau$. Since $\tau$ is an isolated zero of $j(hz)-j(z)^2-1$, for all sufficiently large $s$, Rouch\'e's theorem supplies $\tau_s \in D$ such that $j(\gamma_s g \zeta_s \tau_s) - j(\zeta_s \tau_s)^2 -1 = 0$. So $\zeta_s \tau_s$ is a zero of $j(gz) - j(z)^2-1$. Gallinaro uses the argument principle instead of Rouch\'e's theorem in the last step of the proof, but Eterovi\'c--Zhao \cite{eterovic-zhao} replace that step with the above argument.
\end{example}

We can extend this approach to obtain the following result, which is already a new case of MEC.

\begin{proposition}\label{thm: EC n=2}
   Let $V \seq \C^2 \times \C^2$ be broad and free. Then $V \cap \gr(j) \neq \emptyset$.
\end{proposition}

\begin{proof}[Proof sketch]
    
When the projection of $V$ to the first two coordinates is dominant, the result follows from \cite[Thm.~1.1]{eterovic-herrero}.\footnote{Our general proof does not rely on this result; it recovers it as a special case.} So we assume the projection has dimension 1. We also assume $\dim V = 2$ as the problem is trivial otherwise. Using broadness and swapping the pairs of coordinates $(x_1,y_1)$ and $(x_2,y_2)$ if necessary, we may assume $x_1, y_2$ form a transcendence basis of the function field $\C(V)$, and find algebraic functions $P(x_1)$ and $Q(x_1, y_2)$ such that the set defined by
\[ x_2=P(x_1), y_1=Q(x_1,y_2) \]
is contained in $V$.

In this exposition, we ignore the subtleties arising from the domains of definition of $P$ and $Q$. We discuss them thoroughly in \S\ref{sec: proof of main thm}, and illustrate them with an example in \S\ref{subsec: less simple example}. So, for the sake of simplicity, the reader may assume $P$ and $Q$ are polynomials.

Then we want to solve the equation
\begin{equation}\label{eqn: example n=2}
    j(z) = Q(z, j(P(z))).
\end{equation}

\medskip\noindent\textbf{Case 1.} There is a real point $r$ where $P(r) \in \h$.

First, note that if $P(r)\in \h$ then a whole real neighbourhood of $r$ must map to $\h$, so $Q(z, j(P(z)))$ is defined on such a neighbourhood. Pick $\tau \in \h$ such that $j(\tau) = Q(r, j(P(r)))$.

Fix a disc $D$ around $\tau$ such that on the boundary $\partial D$ the function \[ f(z):= j(z) - Q(r, j(P(r))) \] does not vanish. Further, pick a sequence $\zeta_s \in \PSL_2(\Z)$ such that $\zeta_s z \to r$ uniformly for $z \in \Dbar$ as $s \to \infty$.  Consider the functions \[ f_s(z) := j(\zeta_s z) - Q(\zeta_s z, j(P(\zeta_s z))) .\]

Since $\zeta_s z \to r$ uniformly for $z \in \Dbar$, and $P$ is defined and holomorphic around $r$ with $P(r) \notin \R$, we conclude that $f_s \to f$ uniformly on $\Dbar$. We also know that $f(\tau)=0$; therefore, by Rouch\'e's theorem, for sufficiently large $s$, there is $\tau_s \in D$ such that $f_s(\tau_s) = 0$. Then $\zeta_s \tau_s$ solves the equation \eqref{eqn: example n=2} as required.

\medskip\noindent\textbf{Case 2.} $P$ maps all real points (where it is defined) to $\R$.

Fix $r\in \R$ and $\theta \in \C$, then choose $\tau \in \h$ such that $j(\tau) = Q(r, \theta)$. Let $h \in \PGL_2(\R)$ be such that $j(h \tau) = \theta$. By construction, $\tau$ is a zero of \[ f(z):=j(z) - Q(r, j(hz)).\]

We want to ensure that $f$ does not vanish identically for a suitable choice of $(r, \theta)$. To this end, observe that if $f\equiv 0$ then, by Ax--Schanuel, for some modular polynomial $\Phi_N$ we must have $\Phi_N(j(z), j(hz)) \equiv 0$. At this point, we could argue that then $h \in \PGL_2(\Q)$ and notice that we could in fact have chosen $h$ outside $\PGL_2(\Q)$ while maintaining $j(h \tau) = \theta$. However, we use a slightly different argument which is in line with what we do in the general case. The identity $\Phi_N(j(z), j(hz)) \equiv 0$ at $z=\tau$ implies $\Phi_N(Q(r,\theta), \theta) = 0$. For a fixed $N$, the set of all $(r, \theta) \in \R \times \C$ for which $\Phi_N(Q(r,\theta), \theta) = 0$ is a closed nowhere dense set. Otherwise $\Phi_N(Q(X,Y), Y) \equiv 0$, so  $\Phi_N(y_1,y_2) = 0$ on $V$, contradicting the freeness of $V$. Thus, any choice of $(r, \theta)$ with $\Phi_N(Q(r,\theta), \theta) \neq 0$ for all $N$ would work. The set of all such $(r, \theta)$ is co-meagre, so a generic choice works by the Baire category theorem.

Let $D$ be a disc around $\tau$ such that $f$ does not vanish on the boundary $\partial D$. By \Cref{thm: density of analytic orbits general version}, there are sequences $\gamma_s, \zeta_s \in \PGL_2(\Z),\, s \in \Z_{>0},$ such that $\zeta_s z \to r$  and $\gamma_s \circ P \circ \zeta_s \to h$ uniformly for $z \in \Dbar$. Let
\[ f_s(z):=  j(\zeta_s z) - Q(\zeta_s z, j(\gamma_s P(\zeta_s z))). \]

Since $\zeta_s z \to r$ and $\gamma_s P( \zeta_s z) \to h(z)$ uniformly on $\Dbar$, we conclude that $f_s\to f$ uniformly on $\Dbar$. Hence, by Rouch\'e's theorem, $f_s(z)$ has a zero $\tau_s \in D$ for all large enough $s$. Then $\zeta_s \tau_s$ is a solution of \eqref{eqn: example n=2}. This finishes the proof.
 \end{proof}

\subsection{A simple example}\label{subsec: simple example}

Let $V \seq \C^3 \times \C^3$ be given by the equations 
\begin{align*}
     & x_2 = \iu x_1,\\
     & x_3 = x_1 + y_2,\\
     & x_1^3 + y_1 +y_2^2 +y_3 = 0.
\end{align*}

We need to solve the following equation:
\begin{equation}\label{eqn: example 1}
    z^3+j(z) + j(\iu z)^2 + j(z+j(\iu z)) = 0.
\end{equation}

Pick an arbitrary non-zero real point $r$ and an arbitrary $\theta_3 \in \C$. Then pick $\tau \in \h$ with $j(\tau) = -r^3-j(\iu r)^2-\theta_3$ and $h \in \PGL_2(\R)$ with $j(h\tau) = \theta_3$. Then the function \[ f(z) := r^3+ j(z) + j(\iu r)^2+ j(hz)\]
vanishes at $\tau$. As in \S\ref{subsec: n=2 case}, we can show that it does not vanish identically for a suitable choice of $(r, \theta_3)$.

Let $F_3(z):= z+j(\iu z)$. Then $F_3$ maps non-zero reals to reals. By \Cref{thm: density of analytic orbits general version}, there are sequences $\gamma_s, \zeta_s \in \PGL_2(\Z), s \in \Z_{>0},$ such that $\zeta_s z \to r$  and $\gamma_s \circ F_3 \circ \zeta_s \to h$ locally uniformly on $\h$.

 Consider the sequence of functions
\begin{align*}
    f_s(z) := (\zeta_s z)^3 + j(\zeta_s z) + j(\iu \zeta_s z)^2 + j(\gamma_s F_3( \zeta_s z)), s \in \Z_{>0}.
\end{align*}

Let $D$ be a disc around $\tau$ such that $f$ does not vanish on the boundary $\partial D$.  Since $\zeta_s z \to r$ and $\gamma_s F_3( \zeta_s z) \to h(z)$ uniformly on $\Dbar$, we conclude that $f_s\to f$ uniformly on $\Dbar$. Hence, by Rouch\'e's theorem, $f_s(z)$ has a zero $\tau_s \in D$ for large $s$. Then $\zeta_s \tau_s$ is a solution of \eqref{eqn: example 1}. We conclude that \eqref{eqn: example 1} has infinitely many solutions.

\subsection{A less simple example}\label{subsec: less simple example}

We now look at another example and provide a detailed proof of the existence of solutions following the general strategy of \S\ref{sec: proof of main thm}. Most of the key ideas of the general proof appear here, but some technical aspects are simpler, which, along with explicit functions and equations, makes them easier to grasp.

    Let $V \seq \C^5 \times \C^5$ be given by the equations 
\begin{align*}
     & \iu x_1+\iu x_2 - x_3 = 0,\\
     & x_4^2 (x_1+x_2 + y_3 - y_3^2) - 2x_4y_3 - 1 =0,\\
     & x_1 + y_4 - x_5 = 0,\\
     & y_1 - x_1 - y_3 = 0,\\
     & x_2+y_4+y_5 - y_2 = 0.
\end{align*}

We choose a branch of the square-root function and write this as follows.  
\begin{align*}
     & x_3 = \iu (x_1+x_2),\\
     & x_4 = \frac{1}{\sqrt{x_1+x_2+y_3}-y_3},\\
     & x_5 = x_1 + y_4,\\
     & y_1 = x_1 + y_3,\\
     & y_2 = x_2+y_4+y_5.
\end{align*}

To be precise, we choose the branch of the square-root function mapping positive reals to positive reals defined on the complex plane with the real interval $(-\infty, 0]$ removed. In fact, we only need to define this branch locally. We discuss this in more detail below. 

Writing $x_i = z_i$, $y_i = j(z_i)$ for $1 \leq i \leq 5$, and $\zbar:=(z_1, z_2)$, we get
\begin{equation}\label{eqn: Example 2}
    \begin{aligned}
     & z_3 = \iu (z_1+z_2) =: Z_3(\zbar),\\
     & z_4 = \frac{1}{\sqrt{z_1+z_2 + j(Z_3(\zbar))} -j(Z_3(\zbar))} =: Z_4(\zbar),\\
     & z_5 = z_1 + j ( Z_4(\zbar)) =: Z_5(\zbar),\\
     & j(z_1) = z_1 + j(Z_3(\zbar)),\\
     & j(z_2) = z_2+ j (Z_4(\zbar)) + j (Z_5(\zbar)) .
\end{aligned}
\end{equation}

The first three equations define the functions $Z_3, Z_4, Z_5$ on a suitable domain, and we want to show that the system consisting of the last two equations has a solution.

As in the above examples, we will let $\zbar$ tend to a real point $\rbar \in \R^2$ to find limiting equations with a solution, and apply Rouch\'e's theorem to conclude that \eqref{eqn: Example 2} has a solution. The function $j(Z_3) = j(\iu (z_1+z_2))$ is well defined when $z_1+z_2$ is a non-zero real number. However, for any real $t \neq 0$ the value $t+j(\iu t)$ is real and positive, hence its square root is real. Thus, $Z_4$ maps real points to real points and so $j(Z_4(\rbar))$ is not defined. As above, we will instead approximate this function by a function of the form $j(hz_1)$ with $h \in \PGL_2(\R)$ (we could also approximate by $j(hz_2)$). To do this formally, and to illustrate the general proof strategy implemented in \S\ref{sec: proof of main thm}, we proceed as follows.

Write $\wbar:=(w_3, w_4, w_5)$, and define the following functions:
\begin{itemize}
    \item $F_3(\zbar, \wbar) = \iu (z_1+z_2)$, $G_3(\zbar, \wbar) = j(F_3(\zbar, \wbar)) = j(\iu (z_1+z_2))$;
    \item $F_4(\zbar, \wbar) = \frac{1}{\sqrt{z_1+z_2 + G_3} -G_3} = \frac{1}{\sqrt{z_1+z_2 + j(\iu (z_1+z_2))} -j(\iu (z_1+z_2))}$, $G_4(\zbar, \wbar) = w_4$;
    \item $F_5(\zbar, \wbar) = z_1 + G_4(\zbar, \wbar) = z_1 + w_4$, $G_5(\zbar, \wbar) = j(F_5(\zbar, \wbar))=  j(z_1+w_4)$.
\end{itemize}

Observe that since $Z_3$ is not real-valued on the reals, we have $F_3 = Z_3, G_3 = j(Z_3)$. But $F_4=Z_4$ is real-valued on real inputs, so instead of $j(Z_4)$ we define $G_4$ as an independent variable $w_4$. The reason is that we are not bound by the value $j(Z_4(\rbar))$, which is undefined, and are able to approximate any complex value (generic enough). Hence, we think of $j(Z_4)$ as an independent function/variable. Once we have a $G_4$ depending on $\wbar$, the subsequent $F_k$'s and $G_k$'s may also depend on $\wbar$. In this case, $F_5 = z_1+w_4$ and since we will let $z_1$ tend to a real point and $w_4$ tend to a generic complex point, the value of $F_5$ would then be non-real, so $G_5$ is defined as $j(F_5)$. Note that the $F_k$'s and $G_k$'s depend only on $\zbar, w_4$, so we will write $F_k(\zbar, w_4)$ and $G_k(\zbar, w_4)$.

Let us discuss where these functions are defined. Pick $\rbar = (r_1,r_2) \in \R^2$ and $\theta_4 \in \C$ such that
\[ t_0:= r_1+r_2 > 0, t_0+j(\iu t_0) \neq j(\iu t_0)^2, \theta_4 \notin \R. \]
Here we use the fact that for a real variable $t$ the function $t+j(\iu t) - j(\iu t)^2 \not\equiv 0$. While this is straightforward to verify in this case, in general we use the Ax--Schanuel theorem to avoid degeneracies. This is where the freeness of $V$ is used.

Then $j$ is defined in a neighbourhood of $\iu t_0$. Since $j$ is real and positive on the imaginary axis (outside the origin) and $t_0 > 0$,  we have $t_0+j(\iu t_0) > 0$. So, we can define a branch of the square-root function in a small neighbourhood of $t_0+j(\iu t_0)$ avoiding $0$. Then the function $\sqrt{z+j(\iu z)}$ is well defined in a small neighbourhood of $t_0$. Thus, the functions $F_k$ and $G_k$ are all defined on a sufficiently small polydisc $\Omega \seq \C^3$ containing $(\rbar, \theta_4)$.

Now pick $\tau_1, \tau_2 \in \h$ with
\begin{align*}
    & j(\tau_1) = r_1+ G_3(\rbar, \theta_4) = r_1+ j(\iu (r_1+r_2)),\\
    & j(\tau_2) = r_2 + G_4(\rbar, \theta_4) + G_5(\rbar, \theta_4) = r_2 + \theta_4 + j(r_1+ \theta_4).
\end{align*}

Further, pick $h_4 \in \PGL_2(\R)$ such that $j(h_4 \tau_1) = \theta_4$. Consider the functions
\begin{align*}
    & H_1(\zbar) := r_1+ G_3(\rbar, j(h_4 z_1)) = r_1+ j(\iu (r_1+r_2)),\\
    & H_2(\zbar) := r_2 + G_4(\rbar,  j(h_4 z_1)) + G_5(\rbar,  j(h_4 z_1)) = r_2 +  j(h_4 z_1) + j(r_1+  j(h_4 z_1)).
\end{align*}
By construction, $\taubar:= (\tau_1, \tau_2)$ is a solution to the system 
\begin{equation}\label{eqn: Ex 2 limit eqn}
    j(z_l) = H_l(\zbar), l=1,2.
\end{equation}
In fact, $\taubar$ is an isolated solution of this system of equations. This is obvious in this example as $H_1$ is constant. In general, we interpret the solutions to \eqref{eqn: Ex 2 limit eqn} as projections of intersections of a certain algebraic variety with the graph of $j$. For generic parameters, we apply the Ax--Schanuel theorem to deduce that the solution germ at the chosen $\taubar$ is 0-dimensional. We use the freeness of $V$ here.

The system \eqref{eqn: Ex 2 limit eqn} constitutes the limiting system that we will approximate using the original equations.

Recall that we want to show that the following system has a solution:
\begin{equation}\label{eqn: Example 2 -- two eqn to solve}
    \begin{aligned}
     & j(z_1) = z_1 + j(Z_3(\zbar)),\\
     & j(z_2) = z_2+ j (Z_4(\zbar)) + j (Z_5(\zbar)) .
\end{aligned}
\end{equation}

We write these equations in terms of the $G_k$'s:
\begin{equation}\label{eqn: Example 2 -- two eqn to solve G_k}
    \begin{aligned}
     & j(z_1) = z_1 + G_3(\zbar, j(F_4(\zbar))),\\
     & j(z_2) = z_2+ G_4(\zbar, j(F_4(\zbar))) + G_5(\zbar, j(F_4(\zbar))) .
\end{aligned}
\end{equation}
Here, by abuse of notation, we omit $w_4$ from $F_4(\zbar, w_4)$ as $F_4$ does not depend on $w_4$.

Now, since $F_4$ is real-valued at real points, we apply \Cref{thm: density of analytic orbits general version} (and \Cref{rem: variable selecting map and positive det}) and get sequences $\gamma_{4,s}, \zeta_{1,s}, \zeta_{2,s} \in \PGL_2(\Z)$ such that
\begin{itemize}

        \item for $l=1,2$, as  $s \to \infty$, we have $\zeta_{l,s} z \to r_l$  locally uniformly on  $\h$;

        \item as  $s \to \infty$, we have $\gamma_{4,s} F_4(\zeta_{1,s} z_1, \zeta_{2,s} z_2 ) \to h_4 z_1$ locally uniformly on $\h^2$.
\end{itemize}

We use these sequences to make a change of variables in \eqref{eqn: Example 2 -- two eqn to solve G_k}:
\begin{equation}\label{eqn: Example 2 -- two eqn to solve G_k -- cchange of var}
    \begin{aligned}
     & j(\zeta_{1,s} z_1) =\zeta_{1,s} z_1 + G_3(\zetabar_s\zbar, j(\gamma_{4,s}F_4(\zetabar_s\zbar))),\\
     & j(\zeta_{2,s}z_2) = \zeta_{2,s}z_2+ G_4(\zetabar_s\zbar, j(\gamma_{4,s}F_4(\zetabar_s\zbar))) + G_5(\zetabar_s\zbar, j(\gamma_{4,s}F_4(\zetabar_s\zbar))) .
\end{aligned}
\end{equation}
Here we used the shorthand $\zetabar_s\zbar:= (\zeta_{1,s} z_1, \zeta_{2,s}z_2)$.

Writing $H_{l,s}(\zbar), l=1,2,$ for the right-hand sides of the equations in \eqref{eqn: Example 2 -- two eqn to solve G_k -- cchange of var}, we see that $H_{l,s} \to H_l$ uniformly on a small neighbourhood of $\taubar$ as $s \to \infty$. Working in such a neighbourhood, and taking into account the $\PGL_2(\Z)$-invariance of $j$, we apply Rouch\'e's theorem and conclude that for sufficiently large $s$ the system \eqref{eqn: Example 2 -- two eqn to solve G_k -- cchange of var} has a solution $\taubar_s$ in this neighbourhood of $\taubar$. Then $\zetabar_s \taubar_s$ is a solution to \eqref{eqn: Example 2 -- two eqn to solve G_k}.

\section{Proof of the main theorem}\label{sec: proof of main thm}

In this section, assuming \Cref{thm: density of analytic orbits general version} (proved in \S\ref{sec: dynamics}), we prove our main theorem which we restate below for the convenience of the reader. \Cref{thm: secondary thm density} is proved in the last subsection.

\begin{customthm}{\ref*{thm: EC full}}
     Let $V \seq \C^n \times \C^n$ be a broad and free variety. Then $V \cap \gr(j) \neq \emptyset$.
\end{customthm}

We begin with a high-level outline of the proof to guide the reader through the main steps of the argument.

\begin{enumerate}
    \item We observe that in \Cref{thm: EC full} we may assume $\dim V = n$. 

    \item Using broadness, we choose a suitable transcendence basis of the coordinate functions of $V$ of the form $(x_1, \ldots, x_m, y_{m+1},\ldots, y_n)$. Recall that $(x_1, \ldots, x_n, y_1, \ldots, y_n)$ denote the coordinates on $\C^n \times \C^n$. We then write each coordinate $x_k\, (k>m)$ as an algebraic function $P_k(x_1, \ldots, x_m, y_{m+1}, \ldots, y_{k-1})$ and $y_l\, (1 \leq l \leq m)$ as an algebraic function $Q_l(x_1, \ldots, x_m, y_{m+1}, \ldots, y_n)$. Note that the coordinate basis is chosen so that $P_k$ does not depend on $y_k, \ldots, y_n$.

    \item Substituting $x_i = z_i, y_i = j(z_i)$, we get a system of $n$ equations in $n$ variables to be solved.  The equations allow us to express each $z_k\, (k>m)$ as an analytic function $Z_k(\zbar)$ where $\zbar:=(z_1, \ldots, z_m)$. These functions are defined recursively, and we need to ensure that at each step the functions substituted into the algebraic function $P_k$ take values in the domain of definition of $P_k$. This amounts to showing that the functions we construct are sufficiently independent from each other, which follows from the Ax--Schanuel theorem.

    \item Then we choose a sufficiently generic point $\rbar \in \R^m$. If, for some $k$, the function $Z_k(\zbar)$ is well defined around $\rbar$ and takes non-real values there, then $j(Z_k)$ is well defined. In this case we can let $\zbar$ approach $\rbar$ by a sequence of $\PGL_2(\Z)^m$-transformations. However, when $Z_k$ is real-valued on a real neighbourhood of $\rbar$, the composition $j(Z_k)$ is undefined (at real points). So, we actually construct a related but different collection of functions which encode the non-real $Z_k$'s and the real $Z_k$'s in different ways. Ax--Schanuel is used to avoid potential degeneracies in this construction.

    \item We separate all $Z_k$'s that are real-valued on real inputs and show that for suitable sequences $\gamma_{k,s}, \zeta_{l,s} \in \PGL_2(\Z)$ the composition $\gamma_{k,s} Z_k(\zeta_{1,s}z_1, \ldots, \zeta_{m,s} z_m)$ tends to a prescribed M\"obius transformation with respect to one of the variables. This is possible due to \Cref{thm: density of analytic orbits general version}.

    \item These approximations yield a sequence of systems of equations, obtained from the original one using modular transformations. This sequence tends to a limiting system with a solution by construction. We use the Ax--Schanuel theorem to show that this solution is isolated. We then apply Rouch\'e's theorem to conclude that the original system of equations has a solution. 
\end{enumerate}

\subsection{Reductions}

\begin{lemma}\label{lem: reduce dim}
    If $V \seq \C^{2n}$ is broad and free, then it contains a broad and free subvariety of dimension $n$. Hence, in \Cref{thm: EC full} we may assume $\dim V = n$.
\end{lemma}
\begin{proof}
    This is a well-known trick used frequently in Existential Closedness problems (see \cite{Kirby-semiab,Zilber-bicolored} and specifically \cite[Lem.~4.30]{Aslanyan-adequate-predim}). Nevertheless, we provide a proof for completeness. 
    
    If $\dim V > n$ then we intersect $V$ with generic hyperplanes and reduce its dimension, while preserving broadness and freeness. More precisely, we construct a broad and free $V' \seq V$ with $\dim V' = \dim V - 1$. We can then iterate this construction if necessary to reach dimension $n$.

    Let $V$ be defined over a finitely generated field $K$. Pick a Zariski generic point $\vbar \in V$ over $K$, that is, $\td_KK(\vbar) = \dim V$. Then pick $p_1, \ldots, p_{2n-1} \in \C$ algebraically independent over $K(\vbar)$. Let $p_{2n} := -\sum_{i=1}^{2n-1} \frac{v_i}{v_{2n}} p_i $. Since $p_1, \ldots, p_{2n-1}$ are algebraically independent over $K(\vbar)$, if $p_{2n}$ is algebraic over $K(p_1, \ldots, p_{2n-1})$ then all the coefficients $\frac{v_i}{v_{2n}}$ must be algebraic over $K$. This then implies $\td_KK(\vbar)\leq 1$, contradicting $\dim V > n \geq 1$. Thus, $\td_KK(\pbar)=2n$ where $\pbar:=(p_1, \ldots, p_{2n})$.

    We claim that the Zariski closure $V'$ of $\vbar$ over $K':=K(\pbar)^{\alg}$ satisfies the desired conditions. To this end, first observe that $V' \seq V$. Further, let $I \seq \{ 1, \ldots, n \}$ be a non-empty subset, $\tilde{I} := I \cup (I+n) \seq \{ 1, \ldots, 2n \}$, and let $\vbar_{\tilde{I}}$ denote the tuple $(v_i)_{i \in \tilde{I}}$. Then we have
    \[ \td_{K'}K'(\vbar_{\tilde{I}})  = \td_KK(\pbar, \vbar_{\tilde{I}}) - \td_KK(\pbar) = \td_KK(\vbar_{\tilde{I}})+ \td_{K(\vbar_{\tilde{I}})}K(\pbar, \vbar_{\tilde{I}})- 2n .\]
    
    Now, if $I=\{ 1, \ldots, n\}$ then this implies \[  \td_{K'}K'(\vbar) = \td_KK(\vbar) + \td_{K(\vbar)}K(\pbar, \vbar)- 2n = \dim V + (2n-1)-2n = \dim V - 1.\]
    Therefore, $\dim V' = \dim V - 1$ as claimed.

    Now assume  $I\subsetneq \{ 1, \ldots, n\}$. If  $\td_{K(\vbar_{\tilde{I}})}K(\pbar, \vbar_{\tilde{I}}) = 2n$, we have $\td_{K'}K'(\vbar_{\tilde{I}}) = \td_KK(\vbar_{\tilde{I}}) \geq |I|$ by the broadness of $V$. If $\td_{K(\vbar_{\tilde{I}})}K(\pbar, \vbar_{\tilde{I}}) < 2n$ then we can show as above that $\frac{v_i}{v_{2n}} \in K(\vbar_{\tilde{I}})^{\alg}$, hence $\vbar \in K(\vbar_{\tilde{I}})^{\alg}$, and so $\td_{K'}K'(\vbar_{\tilde{I}}) = \td_{K'}K'(\vbar) = \dim V - 1 \geq n$. This establishes the broadness of $V'$.

    The same argument shows that for any $i$ we have $v_i \notin K(\pbar)^{\alg}$, hence no coordinate is constant on $V'$. Now, if $V'$ is contained in a variety $S$ defined by a $\PGL_2(\Q)$-relation or a modular relation (on the appropriate set of coordinates) then $\vbar \in S$ and hence $V \seq S$, for $S$ is defined over $\Q$ and $\vbar$ is generic in $V$ over $K$. This contradicts the freeness of $V$. Thus, $V'$ is free, and we are done.
\end{proof}

The next combinatorial statement allows us to choose and work with a suitable set of coordinates on a broad variety.

\begin{lemma}\label{lem: broad pregeom}
    Let $\cl$ be a pregeometry on a set $M = \{ x_1, \ldots, x_n, y_1, \ldots, y_n \}$ with $\dim M = n$. Assume further that $M$ is broad, i.e.\ for any $I \seq \{ 1,\ldots, n\}$ we have $\dim \{ \xbar_I, \ybar_I\} \geq |I|$. Then, after reordering the coordinate pairs $(x_i,y_i)_{i=1}^n$, there is $m \in \{0, \ldots, n\}$ such that $\{ x_1, \ldots, x_m, y_{m+1},\ldots, y_n \}$ is a basis of $M$ and for any $m< k \leq n$ we have $x_{k} \in \cl (x_1,\ldots,x_m, y_{m+1}, \ldots, y_{k-1})$.

    Moreover, if $x_i, y_i \notin \cl(\emptyset)$ for all $i$ then $m\geq 1$.
\end{lemma}
\begin{proof}[Proof.\protect\footnotemark\nopunct]

\footnotetext{I first learnt about this statement (in the context of Exponential Closedness) from J. Kirby. This particular proof is based on an argument suggested by ChatGPT 6 Astra.}
We argue by induction on $n$, the $n=1$ case being trivial.

Let $n>1$. Consider two cases.

\medskip\noindent\textbf{Case 1.} There is a non-empty $I \subsetneq \{1,\ldots,n\}$ such that $\dim \{ \xbar_I, \ybar_I\} = |I|$.

Reordering $M$ we may assume that $I = \{ 1, \ldots, t \}$. Write $J:= \{ t+1, \ldots, n \}$. Since $t<n$, by the induction hypothesis there is $m_1 \in I \cup \{ 0 \}$ such that the set $\{ x_1, \ldots, x_{m_1}, y_{m_1+1},\ldots, y_t \}$ is a basis of $M_1:=\{ \xbar_I, \ybar_I \}$ and for any $m_1< k \leq t$ we have (possibly after reordering again) $x_{k} \in \cl (x_1,\ldots,x_{m_1}, y_{m_1+1}, \ldots, y_{k-1})$. 

Consider the set $M_2:= \{ \xbar_J, \ybar_J \}$ with the relative closure $\cl_{M_1} (A) = \cl(M_1 \cup A) \cap M_2$. We claim that $M_2$ is broad with respect to $\cl_{M_1}$. Indeed, for any $J' \seq J$ we have 
\[ \dim_{M_1} \{ \xbar_{J'}, \ybar_{J'} \} = \dim \{ \xbar_{I \cup J'}, \ybar_{I \cup J'} \} - \dim M_1 \geq |I \cup J'| - t = |J'|. \]

By the induction hypothesis again, after reordering $M_2$, there is $m_2 \in J \cup \{ t \}$ such that $\{ x_{t+1}, \ldots, x_{m_2}, y_{m_2+1},\ldots, y_n \}$ is a basis of $M_2$ and for any $m_2< k \leq n$ we have $x_{k} \in \cl_{M_1} (x_{t+1},\ldots,x_{m_2}, y_{m_2+1}, \ldots, y_{k-1})$. Then the tuple 
\[ (x_1, \ldots, x_{m_1}, x_{t+1}, \ldots, x_{m_2}, y_{m_1+1}, \ldots, y_t, y_{m_2+1}, \ldots, y_n) \]
is the desired basis for $M$ (after reordering once more).

\medskip\noindent\textbf{Case 2.} For all non-empty $I \subsetneq \{1,\ldots,n\}$ we have $\dim \{ \xbar_I, \ybar_I\} > |I|$.

Consider the set $M_2:= \{ x_2, \ldots, x_n, y_2, \ldots, y_n \}$ with the relative closure $\cl_{x_1}$. For any $J' \seq J:= \{ 2, \ldots, n \}$ we have $\dim_{x_1} M_2 = \dim (M_2 \cup \{ x_1 \} ) -1  = n-1$ and
\[ \dim_{x_1} \{ \xbar_{J'}, \ybar_{J'} \} = \dim \{ x_1, \xbar_{J'},\ybar_{J'} \} - 1 \geq |J'|. \]
This shows that $(M_2, \cl_{x_1})$ is broad. Hence, by the induction hypothesis, after reordering $M_2$, there is $m_2 \in \{ 1, \ldots, n \}$ such that $\{ x_{2}, \ldots, x_{m_2}, y_{m_2+1},\ldots, y_n \}$ is a basis of $M_2$ and for any $m_2< k \leq n$ we have $x_{k} \in \cl_{x_1} (x_{2},\ldots,x_{m_2}, y_{m_2+1}, \ldots, y_{k-1})$. Then the tuple 
\[ (x_1, \ldots, x_{m_2}, y_{m_2+1}, \ldots, y_n) \]
is the desired basis for $M$.

The ``moreover'' clause is immediate as otherwise we would have $x_1 \in \cl(\emptyset)$.
\end{proof}

\subsection{Algebraic maps}\label{subsec: alg func}

In this and the next subsections we construct various functions, domains, varieties, etc. For the convenience of the reader, we have gathered the relevant data in \Cref{tab:recursive-notation}, which can be referred to throughout the construction.

 Consider the function field $\C(V)$ of $V$. Algebraic closure over $\C$, restricted to $\C(V)$, defines a pregeometry. We apply \Cref{lem: broad pregeom} to the restriction of this pregeometry to the set of coordinate functions $\xbar, \ybar$. Broadness of $V$ and the absence of constant coordinates ensure the assumptions of the lemma are satisfied, so we get a set of algebraically independent coordinates $x_1, \ldots, x_m, y_{m+1}, \ldots, y_n$ on $V$ that form a transcendence basis of $\C(V)$. In other words, the projection of $V$ to these coordinates is dominant.

We can apply \cite[Prop.~3.5 and Cor.~3.6]{Aslanyan-Kirby-Mantova} to express the other coordinates as global algebraic functions of $x_1, \ldots, x_m, y_{m+1}, \ldots, y_n$ on $V$ defined outside a branch cut. However, we need only local algebraic branches. Write \[ \pi\colon \C^{n}\times \C^n \to \C^{m} \times \C^{n-m} , (\xbar, \ybar) \mapsto (x_1, \ldots, x_m, y_{m+1}, \ldots, y_n)\] for the projection map. Since $\pi\vert_V\colon V \to \C^n$ is dominant, there is a proper Zariski closed subset $E \seq  \C^n$  such that $V':= V \cap \pi^{-1}(\C^n \setminus E)$ is a smooth Zariski open dense subset of $V$ and the map $\pi\colon  V'  \to \C^n \setminus E$ is a surjective, finite-to-one map whose Jacobian has rank $n$ everywhere. By the inverse function theorem, at any point $\ubar \in \C^n \setminus E$, there are Euclidean neighbourhoods $\ubar \in U \seq \C^n \setminus E$ and $U'\seq V'$ with $U=\pi(U')$, and a holomorphic algebraic map $\Rbar\colon  U \to U'$ which is an inverse of $\pi$, that is, $\pi \circ \Rbar$ and $\Rbar \circ \pi$ are the identity maps on $U$ and $U'$ respectively.

Writing $\Rbar = (x_1,\ldots,x_m, P_{m+1}, \ldots, P_n, Q_1, \ldots, Q_m, y_{m+1}, \ldots, y_{n})$, the set defined by the equations
 \begin{equation}\label{eqn: algebraic}
\begin{aligned}
    & x_{k} = P_k(x_1,\ldots,x_m, y_{m+1}, \ldots, y_{k-1}),~ k= m+1, \ldots, n;\\
    & y_l = Q_l(x_1,\ldots,x_m, y_{m+1}, \ldots, y_{n}),~ l=1,\ldots,m
\end{aligned}
\end{equation}
is contained in $V$. Note that since \Cref{lem: broad pregeom} guarantees that $x_{k}$ is algebraic over $(x_1,\ldots,x_m, y_{m+1}, \ldots, y_{k-1})$ in $\C(V)$, the algebraic function $P_k$ does not depend on $y_k, \ldots, y_n$.

However, fixing the domain of all $P_k$'s and $Q_l$'s at the beginning causes issues in the next subsection where we recursively compose these algebraic functions with analytic functions. Therefore, we instead construct $P_k$'s recursively choosing a suitable domain and shrinking it at each step, after which we shrink the domain again and define all $Q_l$'s at once.

 \begin{table}[htbp]
\centering
\renewcommand{\arraystretch}{1.3}

\scalebox{0.85}{%
\begin{tabular}{|p{0.06\textwidth}|p{1.02\textwidth}|}
\hline

$P_k$
&
An algebraic function expressing $x_k$ in terms of
$x_1,\ldots,x_m,y_{m+1},\ldots,y_{k-1}$.
\\[0.5ex]

$Q_l$
&
An algebraic function expressing $y_l$ in terms of
$x_1,\ldots,x_m,y_{m+1},\ldots,y_n$.
The branches of the $P_k$'s and $Q_l$'s are chosen compatibly
so that their joint graph is contained in $V$.
\\[0.5ex]

$E_{k-1}$
&
The exceptional algebraic set in $\C^{k-1}$ excluded when constructing the branches of $P_k$.
\\[0.5ex]

$O_{k-1}$
&
A Zariski open dense subset of $\C^{k-1} \setminus E_{k-1}$ over which the previously selected branches of algebraic functions admit a compatible local lift.
\\[0.5ex]

$U_{k-1}$
&
An open subset of $O_{k-1}$ where $P_k$ is defined. $Q_l$'s are defined on $U:=U_n \seq O_n$.
\\[0.5ex]

$V_{k}$
&
The Zariski closure of the projection of $V$ to the coordinates \[ (x_1, \ldots, x_m, y_{m+1}, \ldots, y_{k-1}, x_{m+1}, \ldots, x_k).\]
The graph of $(P_{m+1}, \ldots, P_k)$ is Zariski dense in $V_{k}$.
\\[0.5ex]

$Z_k$
&
The analytic functions obtained by substituting $x_i = z_i, y_i = j(z_i)$, given by
\[
Z_k(\zbar)
=
P_k\bigl(\zbar,
j(Z_{m+1}(\zbar)),\ldots,j(Z_{k-1}(\zbar))\bigr),
\]
wherever the composition is defined.
\\[0.5ex]

$F_k$
&
The modified input functions, defined by
\[
F_k(\zbar,\wbar)
=
P_k\bigl(\zbar,G_{m+1}(\zbar,\wbar),\ldots,
G_{k-1}(\zbar,\wbar)\bigr).
\]
\\[-2ex]

$G_k$
&
The corresponding output functions:
$G_k=j(F_k)$ on a domain where $F_k$
avoids $\mathbb R$, and $G_k=w_k$ when $F_k$ is real-valued on $\Omega_{k-1}\cap(\mathbb R^m\times\mathbb C^{n-m})$.
\\[0.5ex]

$\Omega_{k-1}$
&
Suitable domain where $F_k, G_k$ are defined.
\\[0.5ex]

$T$
&
The set of indices $k$ for which $F_k$ is real-valued on
$\Omega_{k-1}\cap(\mathbb R^m\times\mathbb C^{n-m})$.
For these indices, $F_k$ is independent of $\wbar$,
and the auxiliary variable $w_k$ replaces $j(F_k)$.
\\

\hline
\end{tabular}%
}
\par\vspace{6pt}
\caption{Notation for the recursive construction.
Here $m<k\le n$, $1\le l\le m$,
$\zbar=(z_1,\ldots,z_m)$, and
$\wbar=(w_{m+1},\ldots,w_n)$.\protect\footnotemark}
\label{tab:recursive-notation}
\vspace*{-20pt}
\end{table}

\footnotetext{Including such a table was suggested by ChatGPT 6 Astra during the review of a previous version of the manuscript. The LaTeX code for the table was generated by ChatGPT.}

For $m<k \leq n$, let $V_k \seq \C^{2k-m-1}$ be the  Zariski closure of the projection of $V$ to the coordinates \[ (x_1, \ldots, x_m, y_{m+1}, \ldots, y_{k-1}, x_{m+1}, \ldots, x_k)\] and let $\pi_k$ be the projection map \[ \pi_k\colon (x_1, \ldots, x_m, y_{m+1}, \ldots, y_{k-1}, x_{m+1}, \ldots, x_k)\mapsto (x_1, \ldots, x_m, y_{m+1}, \ldots, y_{k-1}).\] 
A local section of $\pi_k$ is an algebraic map \[ \Pbar_{k}:= (x_1, \ldots, x_m, y_{m+1}, \ldots, y_{k-1}, P_{m+1}, \ldots, P_{k})\] defined on some Euclidean open set in $\C^{k-1}$ whose image is contained in $V_k$. We say $\Pbar_k$ extends $\Pbar_{k-1}$ if the projection of the graph of $\Pbar_k$ to the relevant coordinates is contained in the graph of $\Pbar_{k-1}$. We now construct these sections recursively.

As above, there is a proper Zariski closed subset $E_{k-1} \seq \C^{k-1}$ such that  $V'_k:= V_k \cap \pi_k^{-1}(\C^{k-1} \setminus E_{k-1})$ is a smooth Zariski open dense subset of $V_k$ and the map $\pi_k\colon V_k'  \to \C^{k-1} \setminus E_{k-1}$ is a surjective, finite-to-one map with Jacobian of full rank at all points. We also let $E_n:=E$.

Below we let $\rho_k\colon \C^k \to \C^{k-1}$ be the projection map to the first $k-1$ coordinates.

\begin{lemma}\label{lem: comp local branches}
    Assume $m<n$. For each $m\leq k \leq n$ there is a Zariski open dense subset $O_{k}\seq \C^{k} \setminus E_{k}$ with the following properties.
    \begin{itemize}
        \item For any $m<k\leq n$ we have $\rho_k(O_k) \seq O_{k-1}$.
        \item For any $\ubar_m \in O_m := \C^m \setminus E_m$ there are a polydisc $U_m \seq O_m$ containing $\ubar_m$ and a local section of $\pi_{m+1}$ \[ \Pbar_{m+1}:= (x_1, \ldots, x_m, P_{m+1}) \colon U_{m} \to V_{m+1}.\]
        \item Let $m+1<k \leq n$, and let $\Pbar_{k-1} \colon U_{k-2} \to V_{k-1}$ be a local section of $\pi_{k-1}$ on some polydisc $U_{k-2} \seq O_{k-2}$. Then for every $\ubar_{k-1} \in O_{k-1}$ with $\rho_{k-1}(\ubar_{k-1})\in U_{k-2}$, there are a polydisc $U_{k-1} \seq O_{k-1}$ around $\ubar_{k-1}$ with  $\rho_{k-1}(U_{k-1}) \seq U_{k-2}$, and a local section $\Pbar_k$ of $\pi_k$ extending $\Pbar_{k-1}$: \[ \Pbar_k = (x_1,\ldots,x_m, y_{m+1}, \ldots, y_{k-1},P_{m+1}, \ldots, P_{k-1}, P_k ) \colon U_{k-1} \to V_k. \]  
        
        \item Given any local section  $\Pbar_n \colon U_{n-1} \to V_n$ of $\pi_n$ on some polydisc $U_{n-1} \seq O_{n-1}$, for every $\ubar_n \in O_n$ with $\rho_{n}(\ubar_{n})\in U_{n-1}$, there are a polydisc $U_{n} \seq O_n$ around $\ubar_{n}$ with $\rho_{n}(U_{n})\seq U_{n-1}$, and a local section $\Rbar$ of $\pi$ extending $\Pbar_n$:
        \[ \Rbar=  (x_1,\ldots,x_m, P_{m+1}, \ldots, P_n, Q_1, \ldots, Q_m, y_{m+1}, \ldots, y_{n}) \colon U_n \to V.\]
    \end{itemize}
\end{lemma}

\begin{proof}
For $k=m+1$, we can work in a small neighbourhood $U_{m}$ of any point $\ubar_{m} \in O_m = \C^m \setminus E_m$ and define an algebraic map \[ \Pbar_{m+1}:= (x_1, \ldots, x_m, P_{m+1}) \colon U_{m} \to U'_{m} \] with $U'_{m} \seq V'_{m+1}$,  $U_{m}=\pi_{m+1}(U'_{m})$, and an algebraic function $P_{m+1}\colon  U_{m} \to \C$. Further, assume that $k>m+1$ and that the local section $\Pbar_{k-1} \colon  U_{k-2} \to V_{k-1}$ is already defined on $U_{k-2} \seq O_{k-2}$. 

The projection map from $V_{k}$ to $V_{k-1} \times \C$ given by
\begin{align*}
    (x_1, \ldots, x_m&,  y_{m+1}, \ldots, y_{k-1}, x_{m+1}, \ldots, x_k) \mapsto \\ & (x_1, \ldots, x_m, y_{m+1}, \ldots, y_{k-2}, x_{m+1}, \ldots, x_{k-1}, y_{k-1})
\end{align*}
is dominant and so its image contains a Zariski open dense subset $S_{k-1}$. The Zariski closure $S'_{k-1}$ of the projection of $(V_{k-1}\times \C) \setminus S_{k-1}$ to $(x_1, \ldots, x_m, y_{m+1}, \ldots, y_{k-1})$ has dimension $\leq k-2$. Define $O_{k-1} := (O_{k-2}\times \C) \setminus (S'_{k-1} \cup E_{k-1})$. 

Now, pick $\ubar_{k-1}:=(\ubar_{k-2}, u_{k-1}) \in O_{k-1}$ extending a point $\ubar_{k-2} \in U_{k-2}$, and set $\ubar'_{k-2} := \Pbar_{k-1}(\ubar_{k-2}) \in V_{k-1}$. Then $\ubar''_{k-1}:= (\ubar'_{k-2},u_{k-1}) \in S_{k-1}$ has a lift $\ubar'_{k-1} \in V_k$. Therefore,  $\ubar'_{k-1}$ extends a point from $U'_{k-2}:=\Pbar_{k-1}(U_{k-2})$ (namely, $\ubar'_{k-2}$) and  $\ubar_{k-1} = \pi_k(\ubar'_{k-1})$. Since $\ubar_{k-1} \in \C^{k-1}\setminus E_{k-1}$, we have $\ubar'_{k-1} \in V'_k$. Pick small neighbourhoods $\ubar_{k-1} \in U_{k-1} \seq O_{k-1}$ and $\ubar'_{k-1} \in U'_{k-1} \seq V'_{k}$ such that $U_{k-1}=\pi_{k}(U'_{k-1})$, and the projection of $U_{k-1}$ (respectively, $U'_{k-1}$) is contained in $U_{k-2}$ (respectively, $U'_{k-2}$). We pick these neighbourhoods small enough so that we can define an algebraic map $\Pbar_{k}\colon  U_{k-1} \to U'_{k-1}$, which is an inverse of $\pi_k$. Because in a small neighbourhood the local inverse is uniquely determined by the point $\ubar'_{k-1}$ whose projection to $V_{k-1}$ is in the image of $\Pbar_{k-1}$, it follows that 
\[ \Pbar_k = (x_1,\ldots,x_m, y_{m+1}, \ldots, y_{k-1},P_{m+1}, \ldots, P_{k-1}, P_k )  \] 
for some algebraic function $P_k$ defined on $U_{k-1}$.

Thus, we can define a branch of the algebraic function $P_k$ compatibly with the previously chosen branches, that is, we can ensure that the graph of $(P_{m+1}, \ldots, P_k)$ (which is the same as the image of $\Pbar_k$) is contained in $V_k$.

In the final step, apply the same argument to the dominant projection $V \to V_n \times \C$
\[ (x_1, \ldots, x_n, y_{1}, \ldots, y_{n}) \mapsto (x_1, \ldots, x_m, y_{m+1}, \ldots, y_{n-1}, x_{m+1}, \ldots, x_{n}, y_{n}). \] We obtain a Zariski open dense subset $O_n  \seq  \C^n \setminus E$ such that for every $\ubar_n \in O_n$ projecting into $U_{n-1}$, the section $\Pbar_n$ extends to a local section $\Rbar$ of $\pi$ on a polydisc $U:= U_n \seq O_n$ around $\ubar_n$ whose projection is contained in $U_{n-1}$.
\end{proof}

When $m=n$, the inverse function theorem directly gives a local section $\Rbar$ of $\pi$ near every point of $\C^n\setminus E$.

\begin{remark}
    One may try to find a basis of coordinates $x_1, \ldots, x_m, y_{m+1}, \ldots, y_n$ such that $m= \dim \pr_{\xbar} V$ and $x_k \in \C(x_1, \ldots, x_m)^{\alg}$ for all $m<k \leq n$. This would then make the equations \eqref{eqn: algebraic}, hence also the analytic equations we consider below, significantly simpler to deal with. That is indeed possible in the exponential setting (with linear/multiplicative change of variables), but not here. For instance, if $V \seq \C^3 \times \C^3$ is given by the equations
    \[ x_2=x_1^2, y_2 = x_1+x_3, y_1= x_2+x_3, \]
    then the required basis must contain $x_3$ and one of $x_1, x_2$, so it must be either $\{ x_1, x_3, y_2\}$ or $\{ x_2, x_3, y_1\}$ but neither of these sets is independent. A basis given by \Cref{lem: broad pregeom} is $\{ x_1, y_2, y_3\}$.
\end{remark}

\subsection{Analytic maps}\label{subsec: construct Zk Fk Gk}

The set defined by \eqref{eqn: algebraic} is not Zariski closed in general, but it is contained in $V$. So it suffices to show that this set intersects the graph of $j$.

Thus we want to solve the following system of equations (switching to the variables $z_1, \ldots, z_n$ indicates that we now deal with analytic, rather than algebraic, equations):
\begin{equation}\label{eqns after reduction}
\begin{aligned}
     & z_{k} = P_k(z_1,\ldots,z_m, j(z_{m+1}), \ldots, j(z_{k-1})),~ k= m+1, \ldots, n;\\
    & j(z_l) = Q_l(z_1,\ldots,z_m, j(z_{m+1}), \ldots, j(z_{n})),~ l=1,\ldots,m.
\end{aligned}
\end{equation}

As in \S\ref{sec: special cases and examples}, we will approximate this set of equations by a simpler one which has a solution by construction and apply Rouch\'e's theorem to conclude that \eqref{eqns after reduction} has a solution.

Let $\zbar:=(z_1, \ldots, z_m)$. The first group of equations in \eqref{eqns after reduction} (the first line) prompts us to define functions $Z_k(\zbar), m<k\leq n,$ recursively by \[ Z_{k}(\zbar) := P_{k}(\zbar, j(Z_{m+1}(\zbar)), \ldots, j(Z_{k-1}(\zbar))). \] However, we need to be careful about the domain of definition of $Z_k$: if $Z_i(\zbar)$ (for some $m<i<k$) is well defined on some domain but takes some real values there, then $j(Z_i)$ is not defined. So at each step of the recursive definition, we must shrink the domain to ensure that $Z_k$ avoids $\R$ on that domain. Then, our goal is to show that on some such domain the following system of equations has a solution:
\begin{equation}\label{eqn: j(zk) = Q(z, j(Z)...)}
    j(z_l) = Q_l(\zbar, j(Z_{m+1}(\zbar)), \ldots, j(Z_{n}(\zbar))),~ l=1, \ldots, m.
\end{equation}

On the other hand, when approximating \eqref{eqn: j(zk) = Q(z, j(Z)...)} by another system, we will let $\zbar$ approach a real point. So we need to keep track of the functions $Z_k$ that are defined at real points and we should distinguish whether $Z_k$ maps real points to real or non-real values. When real inputs are mapped to real values, we introduce a new variable for $j(Z_k)$ and later assign a generic value to this new variable which we can approximate thanks to \Cref{thm: density of analytic orbits general version}. Thus, we will recursively define a closely related but different collection of functions some of which will agree with some of the $Z_k$'s.

Recall that $\zbar:=(z_1, \ldots, z_m)$ and write $\wbar:=(w_{m+1}, \ldots, w_n)$. Pick $\ubar_{m}\in \R^m \setminus E_{m}$, and let $\Omega_m \seq \C^m \times \C^{n-m}$ be a small polydisc with \[ \ubar_m \in \pr_{\zbar} \Omega_m =: U_m \seq O_m =\C^m \setminus E_m\] such that, if $m<n$, a branch of $P_{m+1}$ is defined on $U_m$ with graph contained in $V_{m+1}$ (see \Cref{lem: comp local branches}). When $m=n$, we do not have any $P_k$'s and only define $Q_l$'s. 

For $m < k\leq n$ we recursively construct the following:
\begin{itemize}
    \item a polydisc $U_k \seq O_k \seq \C^k \setminus E_k$ whose projection is contained in $U_{k-1}$;
    \item for $k<n$, an algebraic function $P_{k+1}$ defined on $U_{k}$;
    \item a polydisc $\Omega_{k} \seq \Omega_{k-1} \seq \C^m \times \C^{n-m}$ with $\pr_{\zbar} \Omega_k \cap \R^m \neq \emptyset$;
    \item functions $F_k(\zbar, \wbar)$ and $G_k(\zbar, \wbar)$ defined on $\Omega_{k-1}$.
\end{itemize}

We will prove by simultaneous induction that
\begin{enumerate}
    \item the graph of $(P_{m+1}, \ldots, P_{k})$ is contained in $V_{k}$;
    \item the functions $z_1, \ldots, z_m, F_{m+1}(\zbar, \wbar), \ldots, F_{k}(\zbar, \wbar)$ are non-constant and pairwise $\PGL_2(\Q)$-independent;
    \item the functions $G_{m+1}(\zbar, \wbar), \ldots, G_{k}(\zbar, \wbar)$ are algebraically independent over $\C(\zbar)$;
    \item the image of $\Omega_{k}$ under the map \[ \chi_{k}\colon (\zbar, \wbar) \mapsto(\zbar, G_{m+1}(\zbar, \wbar), \ldots, G_{k}(\zbar, \wbar))\] is contained in $U_k$.
\end{enumerate}

Now, let $k>m$ and let $U_{i}, \Omega_{i}, P_{i+1}$ be given for all $m\leq i <k$ and $F_i, G_i$ be given for all $m<i<k$, satisfying the properties (1)--(4) above.

\subsection*{Defining $F_k$.}
Since  $\chi_{k-1}(\Omega_{k-1}) \seq U_{k-1}$ and $P_k$ is defined on $U_{k-1}$, we can define
 \[ F_k(\zbar, \wbar) := P_k(\zbar, G_{m+1}(\zbar, \wbar), \ldots, G_{k-1}(\zbar, \wbar))  \text{ for } (\zbar, \wbar) \in \Omega_{k-1}.\]

 \subsection*{Defining $G_k$.} To define $G_k$, we consider two cases.

 \medskip\noindent\textbf{Case 1. }  $F_k(\Omega_{k-1} \cap (\R^m \times \C^{n-m})) \not\subseteq \R$.

 In this case we shrink $\Omega_{k-1}$ so that $\pr_{\zbar} \Omega_{k-1} \cap \R^m \neq \emptyset$ and $F_k(\Omega_{k-1}) \seq \C \setminus \R$, and define
 \[ G_k(\zbar, \wbar) := j(F_k(\zbar, \wbar))  \text{ for } (\zbar, \wbar) \in \Omega_{k-1}.\]

 It is obvious that properties (1) and (4) survive the shrinking. For (2) and (3), we just use the identity theorem.

 \medskip\noindent\textbf{Case 2. } $F_k(\Omega_{k-1} \cap (\R^m \times \C^{n-m})) \seq \R$. 

 In this case we define
 \[ G_k(\zbar, \wbar) := w_k  \text{ for } (\zbar, \wbar) \in \Omega_{k-1}.\]

Case 2 above can happen only if $F_k$ does not depend on $\wbar$, as the image of a non-constant holomorphic function cannot be contained in $\R$. We let $T_k\seq \{ m+1, \ldots, k\}$ be the set of indices that fall under Case 2. In other words, \[ T_k :=\{ i: G_i(\zbar, \wbar) = w_i \}. \]

\subsection*{Independence of the functions.} We now turn to the inductive proof of the independence of the constructed functions claimed above. For the base case we just need to observe that $z_1, \ldots, z_m$ are algebraically independent. For $k>m$ assume by the induction hypothesis that $\zbar, F_{m+1}, \ldots, F_{k-1}$ are non-constant and $\PGL_2(\Q)$-independent and $\zbar, G_{m+1}, \ldots, G_{k-1}$ are algebraically independent.

\begin{claim}\label{claim: F_k indep and G_k indep}
    The functions $z_1, \ldots, z_m, F_{m+1}(\zbar, \wbar), \ldots, F_{k}(\zbar, \wbar)$ are non-constant and pairwise $\PGL_2(\Q)$-independent, and the functions $G_{m+1}(\zbar, \wbar), \ldots, G_{k}(\zbar, \wbar)$ are algebraically independent over $\C(\zbar)$.
\end{claim}
\begin{claimproof}
    We verify that $F_k$ is $\PGL_2(\Q)$-independent from $\zbar, F_{m+1}, \ldots, F_{k-1}$. To this end, observe that since $\zbar, G_{m+1}(\zbar, \wbar), \ldots, G_{k-1}(\zbar, \wbar)$ are algebraically independent over $\C$ (by the induction hypothesis), if $F_k=gF_i$ for some $m < i< k $ with $g \in \PGL_2(\Q)$ then $P_k = g P_i$ as algebraic functions. This then implies that $x_k = gx_i$ on $V$ which contradicts the freeness of $V$. Similarly, it is not possible to have $F_k = g z_l$ with $g \in \PGL_2(\Q),\, 1\leq l \leq m$, or $F_k$ constant.

    We now verify that  $\zbar, G_{m+1}, \ldots, G_{k}$ are algebraically independent over $\C$. Consider the sets
 \[ \Lambda_k:= \{ z_1, \ldots, z_m, w_{m+1}, \ldots, w_k \} \cup \{ F_i: m< i \leq k, F_i \notin \PGL_2(\Q)\{ w_{m+1}, \ldots, w_k \} \}\]
 and \[ \Pi_k:=\{  z_1, \ldots, z_m, w_{m+1}, \ldots, w_k\} \cup j(\Lambda_k).\]

 Note that here we temporarily shrink the domain $\Omega_{k-1}$ to ensure $j(\Lambda_k)$ is well defined; this does not affect algebraic relations between the functions.

We show that each $G_i$ is interalgebraic (via a modular relation) with a unique element $\mu_i \in \{ w_{m+1}, \ldots, w_k \} \cup j(\Lambda_k)$ and $\mu_{m+1}, \ldots, \mu_k$ are all distinct. 
  
  First, observe that for all $m<i \leq k$ we have $G_i \in \C(\Pi_{k})^{\alg}$. Indeed, if $i \in T_k$ then $G_i = w_i \in \Pi_k$ and if $i \notin T_k$ then $G_i = j(F_i)$ which either belongs to $j(\Lambda_k)$ when $F_i \notin \PGL_2(\Q)\{ w_{m+1}, \ldots, w_k \}$ or is in a modular relation with $j(w_t) \in \Pi_k$ for some $m<t \leq k$ otherwise. In the last case, note that for $i \neq i'$ it is not possible for $G_i$ and $G_{i'}$ to be in a modular relation with the same $j(w_t)$ as otherwise we would have $F_i, F_{i'} \in \PGL_2(\Q)w_t$ but we know that $F_i$ and $F_{i'}$ are in distinct $\PGL_2(\Q)$-orbits.

  Since $\Lambda_k$ is $\PGL_2(\Q)$-independent and contains no constant function, Ax--Schanuel (\Cref{Ax-Schanuel functional}) gives
 \[ \td_{\C} \C(\Lambda_k, j(\Lambda_k)) \geq |\Lambda_k|+k.  \]

 Taking into account that $F_i \in \C(\zbar, G_{m+1}, \ldots, G_{i-1})^{\alg} \seq \C(\Pi_k)^{\alg}$, we conclude that
 \[ \td_{\C} \C(\Pi_k) \geq | \Lambda_k| + k \geq |\Pi_k|.\]
  Thus $\Pi_k$ is algebraically independent. This then implies that $\zbar, G_{m+1}, \ldots, G_{k}$ are algebraically independent over $\C$.
\end{claimproof}

\subsection*{Defining $U_k$, $\Omega_k$, and $P_{k+1}$.}
The functions $\zbar, G_{m+1}(\zbar, \wbar), \ldots, G_k(\zbar, \wbar)$ are algebraically independent over $\C$, hence the image $I_k$ of $\Omega_{k-1} \cap (\R^m \times \C^{n-m})$ (and any open subset thereof) under the map \[ \chi_k\colon (\zbar, \wbar) \mapsto(\zbar, G_{m+1}(\zbar, \wbar), \ldots, G_k(\zbar, \wbar))\] is Zariski dense in $\C^k$. Since $E_k$ is algebraic of dimension $<k$, we can pick a sufficiently generic point $\ubar_{k} \in I_k \setminus E_k$. More precisely, pick $\ubar_k \in I_k \cap O_k$ where $O_k \subseteq \C^k \setminus E_k$ is a Zariski open set given by \Cref{lem: comp local branches} locally over which a compatible branch $P_{k+1}$ exists with joint graph contained in $V_{k+1}$ when $k<n$, and compatible branches of all $Q_l$'s exist with joint graph (with $P_k$'s) contained in $V$ when $k=n$.

Now, if $k<n$, let $U_k \seq O_k$ be a small polydisc around $\ubar_k$ with projection contained in $U_{k-1}$, such that (a branch of) the algebraic function $P_{k+1}$ given by \eqref{eqn: algebraic} is defined on $U_{k}$, chosen compatibly with the previously selected branches, with joint graph contained in $V_{k+1}$ (see \Cref{lem: comp local branches}). Let also $\Omega_k \seq \Omega_{k-1}$ be a small polydisc such that $\pr_{\zbar} \Omega_k \cap \R^m \neq \emptyset$ and $\ubar_k \in \chi_k(\Omega_k)\seq U_k$. 

In the last step (when $k=n$), there is no $P_{n+1}$ to be defined. Instead, we choose a small polydisc $U:=U_n \seq O_n$ as above containing a sufficiently generic point $\ubar_n \in I_n \cap O_n$, and define branches of algebraic maps $Q_l$  on $U$, compatible with the previously selected branches $P_{m+1}, \ldots, P_n$, such that the set defined by \eqref{eqn: algebraic} is contained in $V$.

Then we choose a sufficiently small polydisc $\Omega:=\Omega_n \seq \Omega_{n-1}$ with $\pr_{\zbar} \Omega \cap \R^m \neq \emptyset$ and $\ubar_n \in \chi_n(\Omega) \seq U$. Then the functions  \[ Q_l(\zbar, G_{m+1}(\zbar, \wbar), \ldots, G_{n}(\zbar, \wbar)),\, 1 \leq l \leq m, \]
are defined and holomorphic on $\Omega$.

This provides single-valued fixed algebraic maps $P_k, Q_l$ and functions $F_k, G_k$ and finishes the recursive construction. 

\begin{remark}\label{rem: F_k does not depend}
   Let $T:=T_n = \{ k \in \{m+1, \ldots, n\}: G_k = w_k \}$. When $k \in T$, the function $F_k(\zbar, \wbar)$ is independent of $\wbar$ and agrees with $Z_k(\zbar)$ wherever they are defined. Letting $\Omega_T$ be the projection of $\Omega$ to $(\zbar, \wbar_T)$, it is easy to see that all $Z_k$'s are defined on \[ \{\zbar \in \pr_{\zbar} \Omega \cap \h^m: F_i(\zbar) \in \h \text{ for } i \in T \text{ and } (\zbar, (j(F_i(\zbar)))_{i \in T}) \in \Omega_T \}. \]
   
   In general, $F_k(\zbar, \wbar)$ does not depend on $w_i$ with $i \notin T$ or $i\geq k$.
\end{remark}

\begin{remark}\label{rem: m=n}
    When $m=n$, we do not have any functions $F_k$ and $G_k$, so we do not need the recursive step and only need to define the algebraic functions $Q_l(\zbar),\, 1 \leq l \leq n,$ on $U:=U_m$. We also set $\Omega:=\Omega_m$. This is the dominant projection case considered in \cite{eterovic-herrero}, and our proof in this special case recovers theirs.

    When $T=\emptyset$, none of the functions depend on $\wbar$ and we do not need \Cref{thm: density of analytic orbits general version}. Only the input sequences $(\zeta_{l,s})_s$ are needed, and these are easy to construct.

    When $T = \{ m+1, \ldots, n\}$ then $G_k = w_k$ and $F_k = P_k(\zbar)$. In particular, Gallinaro's raising-to-powers case considered in \cite{Gallinaro-raising-to-powers} falls under this case. 
\end{remark}

We record the following consequence of \Cref{claim: F_k indep and G_k indep} as we refer to it later.

\begin{lemma}\label{lem: F_k indep and G_k indep}
    The functions $z_1, \ldots, z_m, F_{m+1}(\zbar, \wbar), \ldots, F_{n}(\zbar, \wbar)$ are non-constant and pairwise $\PGL_2(\Q)$-independent, and the functions $G_{m+1}(\zbar, \wbar), \ldots, G_{n}(\zbar, \wbar)$ are algebraically independent over $\C(\zbar)$.
\end{lemma}

\subsection{The limit equations}\label{subsec: limit eqns}

Pick a real tuple $\rbar:=(r_1, \ldots, r_m) \in \R^m$ and a complex tuple $\thebar :=(\theta_{m+1}, \ldots, \theta_n )$ such that $(\rbar, \thebar) \in \Omega$.

For $m<k\leq n$ we set $\alpha_k := F_k(\rbar, \thebar)$ and $\xi_k := G_k(\rbar, \bm{\theta})$, and $\xibar := (\xi_{m+1},\dots, \xi_n)$. Then $\alpha_k = P_k(\rbar, \xi_{m+1},\ldots, \xi_{k-1})$. Note that whenever $k \in T$, we have $\alpha_k \in \R$ and $\xi_k = \theta_k$. When $k \notin T$, we have $\alpha_k \notin \R$ and $\xi_k=j(\alpha_k)$. 

Now pick $\tau_l \in \h$ for $1\leq l \leq m$ such that
\begin{equation}\label{eqn: j(tau k ) = Qk (r, xi)}
    j(\tau_l) = Q_l(\rbar, \xibar).
\end{equation}
This is well defined since $(\rbar, \thebar) \in \Omega$ where $Q_l(\zbar, G_{m+1}(\zbar, \wbar), \ldots, G_{n}(\zbar, \wbar))$ is holomorphic. 

As we pointed out above, when $k \in T$ the function $F_k$ does not depend on $\wbar$, so we write $F_k(\zbar)$. Then the functions $(F_k(\zbar))_{k \in T}$ satisfy the conditions of \Cref{thm: density of analytic orbits general version}, so let $\sigma \colon T \to \{ 1, \ldots, m \}$ be the variable-selection map given by that theorem. Now, for $k \in T$ let $h_k \in \PGL_2(\R)$ be such that 
\begin{equation}\label{eqn: j(hk tau sigma k) = theta k}
    j(h_k \tau_{\sigma(k)}) = \theta_k.
\end{equation}

For $k \in T$ define $f_k(z_{\sigma(k)}):= j(h_k z_{\sigma(k)})$ and let $\fbar(\zbar) := (f_k(z_{\sigma(k)}))_{k\in T}$.

Now consider the functions
\begin{equation*}
    H_l (\zbar) := Q_l(\rbar, G_{m+1}(\rbar, \fbar(\zbar)), \ldots, G_n(\rbar, \fbar(\zbar))),~ 1\leq l \leq m.
\end{equation*}
Note that while formally $G_k(\zbar, \wbar)$ is a function of $n$ variables, it does not depend on $w_i$ with $i \notin T$, so we can think of it as a function of $\zbar$ and $\wbar_T:=(w_i)_{i\in T}$, hence the expression $G_k(\rbar, \fbar(\zbar))$ makes sense.

By \eqref{eqn: j(tau k ) = Qk (r, xi)} and \eqref{eqn: j(hk tau sigma k) = theta k},  $H_l$ is well defined and holomorphic around $\taubar:= (\tau_1, \ldots, \tau_m)$. Moreover, $\taubar$ is a solution to the system of equations
\begin{equation}\label{eqn: limit eqns}
    j(z_l) = H_l(\zbar),~ 1\leq l \leq m.
\end{equation}
 We will use \Cref{thm: density of analytic orbits general version} to show that, after a change of variables, the system \eqref{eqns after reduction} approximates \eqref{eqn: limit eqns}, and deduce by Rouch\'e's theorem that the former must have a solution because the latter does. In order to apply Rouch\'e's theorem we need to check first that $\taubar$ is an isolated solution of $\eqref{eqn: limit eqns}$.

 \begin{lemma}\label{lem: tau isolated sol}
    For a generic (outside a meagre set) choice of $(\rbar, \thebar) \in \Omega \cap (\R^m \times \C^{n-m})$, the point $\taubar$ is an isolated solution of \eqref{eqn: limit eqns}.
\end{lemma}
\begin{proof}
    Consider the following system of equations:
\begin{equation}\label{eqns tau isolated sol}
\begin{aligned}
     & x_{k} = P_k(\rbar, y_{m+1}, \ldots, y_{k-1}),~ k\in \{ m+1, \ldots, n \} \setminus T,\\
     & x_k = h_k x_{\sigma(k)},~ k\in T,\\
    & y_l = Q_l(\rbar, y_{m+1}, \ldots, y_n),~ l=1,\ldots,m.
\end{aligned}
\end{equation}
These are well defined whenever $(x_1, \ldots, x_m, y_{m+1}, \ldots, y_n) \in \tilde{U} = X \times Y $ where $X \seq \h^m$ is a polydisc containing $\taubar$ and $Y \seq \C^{n-m}$ is a polydisc containing $\xibar$ such that $\{ \rbar \} \times Y \seq U$. Let $\tilde{V} \seq \C^{2n}$ be the set  defined by these equations over $\tilde{U}$. Since $P_k, Q_l$ are algebraic functions, the coordinates $x_1, \ldots, x_m, y_{m+1}, \ldots, y_n$ form a transcendence basis for $\C(\xbar\vert_{\tilde{V}}, \ybar\vert_{\tilde{V}})$ where $(\xbar\vert_{\tilde{V}}, \ybar\vert_{\tilde{V}})$ are all the coordinate functions $(x_1, \ldots, x_n, y_1, \ldots, y_n)$ restricted to $\tilde{V}$.

Observe that locally the projection of $\tilde{V} \cap \gr(j)$ to the coordinates $(x_1, \ldots, x_m)$ is exactly the set of solutions of \eqref{eqn: limit eqns}. Let $A$ be a component of $\tilde{V} \cap \gr(j)$ whose projection contains $\taubar$. We need to show that $A$ is discrete. So assume, for contradiction, that $\dim A > 0$. On $A$, a coordinate $x_i$ is constant if and only if the corresponding coordinate $y_i$ is constant, for $A \seq \gr(j)$. Let $C \seq \{ 1, \ldots, n\}$ be the set of indices $i$ for which $x_i$ and $y_i$ are constant on $A$. We claim that \[ \td_{\C}\C(\xbar\vert_{A}, \ybar\vert_{A}) \leq n -\lvert C \rvert. \]
Indeed, the transcendence basis $x_1, \ldots, x_m, y_{m+1}, \ldots, y_n$ of $\C(\xbar\vert_{\tilde{V}}, \ybar\vert_{\tilde{V}})$ contains exactly one element from each pair $(x_i, y_i),\, i \in C,$ and the constant values of these coordinates are in the appropriate projection of the set $\tilde{U}$ where all the algebraic functions in \eqref{eqns tau isolated sol} are defined. All other coordinates are algebraic over the $n -\lvert C \rvert$ non-constant members of this transcendence basis. Hence, the inequality follows.

On the other hand, \[ \lvert \{ i\in \{ 1, \ldots, n \}: x_i, y_i \text{ are non-constant on } A\}\rvert = n -\lvert C \rvert > 0,\] since $\dim A > 0$. Therefore, we can apply the Ax--Schanuel theorem (\Cref{Ax-Schanuel functional}) to these non-constant coordinate functions of $A$ and conclude that some modular relation $\Phi_N(y_i, y_{i'}) = 0$ holds identically on $A$. Specialising this to the point in $A$ lifting $\taubar$, we get that the tuple
\[ \Xi:=  (Q_{1}(\rbar, \xibar), \ldots, Q_{m}(\rbar, \xibar), \xibar) \]
is modularly dependent, that is, two of its components are related by a modular equation. Recall that $\xi_k = G_k(\rbar, \thebar)$ for $m<k\leq n$. For a fixed modular polynomial $\Phi_N(y_i, y_{i'})$, the set of $(\rbar, \thebar) \in \Omega \cap (\R^m \times \C^{n-m})$ for which the coordinates $i$ and $i'$ of $\Xi$ are related by $\Phi_N$ is closed and has empty interior. Otherwise by \Cref{lem: F_k indep and G_k indep}, the coordinates $y_i$ and $y_{i'}$ would satisfy the relation $\Phi_N(y_i, y_{i'}) = 0$ on $V$ which contradicts the freeness of $V$.

We conclude that the set of $(\rbar, \thebar) \in \Omega \cap (\R^m \times \C^{n-m})$ for which some two coordinates of the tuple $(Q_{1}(\rbar, \xibar), \ldots, Q_{m}(\rbar, \xibar), \xibar)$ satisfy some modular relation is meagre. By the Baire category theorem, this set has empty interior, and it suffices to choose $(\rbar,\thebar)$ outside this set.
\end{proof}

\subsection{The final step}\label{subsec: final step}

We now complete the proof of Theorem~\ref{thm: EC full}. First, we choose $(\rbar, \thebar)$ as in \Cref{lem: tau isolated sol} to ensure  $\taubar$ is an isolated solution of \eqref{eqn: limit eqns}. As pointed out above, the functions $F_k(\zbar, \wbar)$ do not depend on $w_i$ with $i \notin T$ and do not depend on $\wbar$ at all whenever $k \in T$, so for $k \in T$ we write $F_k(\zbar)$ instead. The functions $(F_k(\zbar))_{k \in T}$ satisfy the conditions of \Cref{thm: density of analytic orbits general version}. Let $\gamma_{k,s}, \zeta_{l,s} \in \PGL_2(\Z),\, k \in T,\, 1\leq l \leq m,$ be the sequences given by that theorem, so that
\begin{itemize}

        \item for all $l$, as  $s \to \infty$, we have $\zeta_{l,s} z \to r_l$  locally uniformly on  $\h$;

        \item for all $k \in T$, as  $s \to \infty$, we have $\gamma_{k,s} F_k(\zeta_{1,s} z_1, \ldots, \zeta_{m,s} z_m ) \to h_k z_{\sigma(k)}$ locally uniformly on $\h^m$.
\end{itemize}

Recall that we want to find a solution to the system of equations
\begin{equation}\label{eqn: j(zk) = Q(z, j(Z)...) final step}
    j(z_l) = Q_l(\zbar, j(Z_{m+1}(\zbar)), \ldots, j(Z_n(\zbar))),~ l=1, \ldots, m.
\end{equation}

Instead of this, we consider the system of equations
\begin{equation}\label{eqn: j(zk) = Q(z, Gk...)}
    j(z_l) = Q_l(\zbar, G_{m+1} (\zbar, \wbar), \ldots, G_n (\zbar, \wbar))|_{\wbar_T = (j(F_k(\zbar)))_{k\in T}},~ l=1, \ldots, m.
\end{equation}

For $k \in T$ introduce new functions $W_{k,s}(\zbar):= j(\gamma_{k,s}F_k(\zetabar_s\zbar))$ and let $\Wbar_s(\zbar):=(W_{k,s})_{k \in T}$. We make a change of variables $\zbar \mapsto \zetabar_s \zbar$ in \eqref{eqn: j(zk) = Q(z, Gk...)} and get
\begin{equation}\label{eqn: original system with Gamma substitution}
     j(\zeta_{l,s}z_l) = Q_l(\zetabar_s\zbar, G_{m+1} (\zetabar_s\zbar, \Wbar_s), \ldots, G_n (\zetabar_s\zbar, \Wbar_s)),~ l=1, \ldots, m.
\end{equation}
Denote the right-hand side of \eqref{eqn: original system with Gamma substitution} by $H_{l,s}(\zbar)$.

Let $D \seq \h^m$ be a small polydisc containing $\taubar$ such that
\[ \{ (\rbar, \fbar(\zbar)): \zbar \in \Dbar \} \seq \Omega_T, \]
where $\Omega_T$ is the projection of $\Omega$ to $(\zbar, \wbar_T)$. Then all $H_l$'s are defined on an open set containing $\Dbar$ and, by shrinking it, we ensure that $\taubar$ is the only solution of \eqref{eqn: limit eqns} in this larger open set. In particular, the map $(j(z_1)-H_1(\zbar), \ldots, j(z_m) - H_m(\zbar))$ does not vanish on the boundary of $D$.

Since  $\gamma_{k,s} F_k(\zetabar_{s} \zbar) \to h_k z_{\sigma(k)}$ uniformly for $\zbar \in \Dbar$, and $j$ is holomorphic in $\h$, we conclude that $W_{k,s}=j(\gamma_{k,s} F_k(\zetabar_{s} \zbar)) \to j(h_k z_{\sigma(k)}) = f_k(z_{\sigma(k)})$ uniformly on $\Dbar$ for all $k \in T$. We also know that $\zetabar_s \zbar \to \rbar$, hence, eventually $(\zetabar_s \zbar, \Wbar_s(\zbar)) \in \Omega_T$, so $H_{l,s}(\zbar)$ is defined on $\Dbar$ and $H_{l,s}(\zbar) \to H_l(\zbar)$ uniformly on $\Dbar$. 

Since $\taubar$ is an isolated zero of $(j(z_1)- H_1(\zbar), \ldots, j(z_m) - H_m(\zbar))$ in $D$ and this map does not vanish on $\partial D$, by Rouch\'e's theorem (\Cref{cor: Rouche}), for all sufficiently large $s$ the map $(j(z_1)- H_{1,s}(\zbar), \ldots, j(z_m) - H_{m,s}(\zbar))$ has a zero $\taubar_s \in D$. Then $\zetabar_s \taubar_s$ is a solution to \eqref{eqn: j(zk) = Q(z, Gk...)}. 

Observe that, by our construction in \S\ref{subsec: construct Zk Fk Gk}, for all $m < k \leq n$ we have \[ Z_k(\zbar) = F_k (\zbar, \wbar)|_{\wbar_T = (j(F_i(\zbar)))_{i\in T}},~ j(Z_k(\zbar)) = G_k (\zbar, \wbar)|_{\wbar_T = (j(F_i(\zbar)))_{i\in T}}. \]

In particular, the functions $Z_k$ are defined around $\zetabar_s \taubar_s$ and, in fact, $\zetabar_s \taubar_s$ is a solution to \eqref{eqn: j(zk) = Q(z, j(Z)...) final step}. This finishes the proof of \Cref{thm: EC full}.

\subsection{Proof of \Cref{thm: secondary thm density}}\label{subsec: proof of density}

We now prove \Cref{thm: secondary thm density}, restated here for the convenience of the reader.

\secthm*
\begin{proof}
    Let $m:=\dim L,~ d:= n- \dim W \leq m$. By the broadness of $L\times W$ we may assume, after a permutation of the coordinate pairs $(x_i, y_i)_i$, that $x_1, \ldots, x_m, y_{d+1}, \ldots, y_n$ form a transcendence basis of $\C(L\times W)$. Then for $m<k\leq n$ we can express $x_k$ as an algebraic function of $\xbar:=(x_1, \ldots, x_m)$, i.e.\ $x_k = P_k(\xbar)$ where $P_k$ is an algebraic function defined in a small neighbourhood $\Omega$ of a point outside the exceptional set $E$ (see \S\ref{subsec: alg func}).\footnote{Note that here we do not need the recursive construction and can apply the inverse function theorem to define all $P_k$'s at once.} Since $L\cap \R^n$ is Zariski dense in $L$, the real points of $L$ cannot be contained in $E \times \C^{n-m}$, hence we can choose $\Omega$ to be a polydisc around a real point. Moreover, the Zariski density of real points in $L$ implies that $L$ is invariant under complex conjugation. We can therefore choose the algebraic branches $P_k$ near a real point lying over $\C^m \setminus E$ and ensure that all $P_k$'s map a small real box $B \seq \R^m$ to $\R$. 

   For $1 \leq l \leq d$ write $y_l = Q_l(\ybar)$ with $\ybar:=(y_{d+1}, \ldots, y_n)$ and $Q_l$ algebraic. For any point $\thebar \in W$ outside a proper Zariski closed set, we can choose the branches of $Q_l$ such that their joint graph contains $\thebar$.

   Let $\rbar \in B$. Choose $\taubar':=(\tau_1, \ldots, \tau_n) \in \h^n$ with $j(\taubar') = \thebar$. Working in a small neighbourhood $D \seq \h^m$ of $\taubar := (\tau_1, \ldots, \tau_m)$, we apply \Cref{thm: density of analytic orbits general version} to find sequences $\gamma_{k,s}, \zeta_{l,s} \in \PGL_2(\Z)$ such that the sequence $(\zetabar_s)_s$ contracts $D$ to the real point $\rbar$, and $\gamma_{k,s}P_k(\zetabar_s\zbar)$ approaches $h_k z_{\sigma (k)}$ for any prescribed matrices $h_{m+1}, \ldots, h_n \in \PGL_2(\R)$. Here $\zbar:=(z_1, \ldots, z_m)$. We choose these matrices such that $\tau_k = h_k \tau_{\sigma(k)}$ for $m<k \leq n$. Then for each $1 \leq l \leq d$, for $\zbar \in D$ the sequence of functions 
   \[ A_{l,s}(\zbar):= j(\zeta_{l,s}z_l) - Q_l(j(\zeta_{d+1,s}z_{d+1}), \ldots, j(\zeta_{m,s}z_{m}),j(P_{m+1}(\zetabar_s\zbar)), \ldots, j(P_n(\zetabar_s\zbar))) \]
   tends uniformly to
   \[ A_l(\zbar) := j(z_l) - Q_l(j(z_{d+1}), \ldots, j(z_{m}), j(h_{m+1}z_{\sigma(m+1)}), \ldots, j(h_n z_{\sigma(n)})). \]
   As in \Cref{lem: tau isolated sol}, $\taubar$ is an isolated zero of $(A_1, \ldots, A_d, z_{d+1}-\tau_{d+1}, \ldots, z_m - \tau_m)$, provided that $\thebar$ is outside a meagre set. Rouch\'e's theorem then supplies a sequence $\taubar_s \in D$ with $\taubar_s \to \taubar$ (by shrinking $D$) such that $(A_{1,s}, \ldots, A_{d,s})(\taubar_s) = 0$. Then ${\bm{\nu}}_s:=(\zetabar_s \taubar_s, P_{m+1}(\zetabar_s \taubar_s), \ldots, P_{n}(\zetabar_s \taubar_s)) \in L$, $j({\bm{\nu}}_s) \in W$, and $j(\bm{\nu}_s) \to j(\taubar') = \thebar$.
\end{proof}

\begin{remark}
    The condition of $L \cap \R^n$ being Zariski dense in $L$ is equivalent to $L$ being defined over $\R$ and containing a smooth real point. The statement is false without this assumption. For instance, we claim that $\Upsilon:= \{ (j(z), j(\iu z)): z \in \C \setminus (\R \cup \iu \R) \}$ is not Euclidean dense in $\C^2$. To this end, pick $\alpha, \beta \in \C \setminus (\R \cup \iu \R)$ algebraically independent over $\Q$ such that $j(\alpha), j(\beta) \notin \R$. We claim that $\overline{\Upsilon}$ does not contain $(j(\alpha), j(\beta))$. Suppose it does, i.e.\ there is a sequence $\tau_s \in \C \setminus (\R \cup \iu\R)$ such that $j(\tau_s) \to j(\alpha),\, j(\iu\tau_s) \to j(\beta)$. By passing to a subsequence, we may assume $\tau_s \to \tau \in \C \cup \{ \infty \}$. 
    
    If $\tau \in \C \setminus (\R \cup \iu\R)$ then $j(\tau) = j(\alpha),\, j(\iu\tau) = j(\beta)$, hence $\tau \in \PGL_2(\Z) \alpha, \iu \tau \in \PGL_2(\Z) \beta$ which is not possible, for $\alpha$ and $\beta$ are algebraically independent over $\Q$.

    If $\tau \in (\R \cup \iu\R) \setminus \{ 0 \}$ then either $j(\tau) \in \R$ or $j(\iu\tau) \in \R$. In either case, $(j(\tau_s), j(\iu\tau_s))$ cannot approach the point $(j(\alpha), j(\beta)) \in (\C \setminus \R) ^2$.

    If $\tau = \infty$ then $\max \{ \lvert \im (\tau_s)\rvert, \lvert \im (\iu\tau_s)\rvert \} \to +\infty$. So, $(j(\tau_s), j(\iu\tau_s))$ cannot approach the finite point $(j(\alpha), j(\beta))$.

    If $\tau = 0$ then replacing $\tau_s$ by $\tau_s^{-1}$ and using the $\PGL_2(\Z)$-invariance of $j$, we conclude as in the previous case.
\end{remark}

\section{A brief discussion of $j$ restricted to $\uh$}\label{sec: j restricted to upper half}

In this section we discuss the problem of intersecting an algebraic variety with $\gr_+(j)$ -- the graph of $j$ restricted to $\uh$. We establish some positive results and argue that one should not expect a general sufficient condition like broadness and freeness.

\begin{definition*}
    Let $(\xbar, \ybar)$ denote the coordinates of $\C^n \times \C^n$. Let $V \seq \C^n \times \C^n$, and let $L$ be the Zariski closure of $\pr_{\xbar} V$. A \textit{convenient basis of coordinates} on $V$ is a transcendence basis $B$  of $\C(V)$, contained in $\{ x_1, \ldots, x_n, y_1, \ldots, y_n \}$, such that $B \cap \{ x_1, \ldots, x_n \}$ is a transcendence basis of $\C(L)$, and $\{ y_i: x_i \notin B \} \seq B$.
  \end{definition*}

For instance, if $\dim L =n$ or $V=L\times W$ (with $W \seq \C^n$) is broad then $V$ has a convenient basis of coordinates.

In general, when a convenient basis of coordinates exists for some $V$ with $\dim V = n$, we may reorder the pairs $(x_i, y_i)$, and assume $x_1, \ldots, x_m, y_{m+1}, \ldots, y_n$ form a convenient basis where $m:=\dim L$. Then each $x_k, m<k\leq n,$ is algebraic over $\C(x_1, \ldots, x_m)$ on $L$. So, at a generic point of $L$, we can choose branches of algebraic functions $P_k(x_1, \ldots, x_m),\, m<k\leq n,$ such that the graph of the map $(P_{m+1}, \ldots, P_n)$ is contained in $L$. When $V$ is free, each $P_k$ is non-constant.

When $L$ is defined over $\R$ and contains a smooth $\R$-point $\pbar$, then it actually contains an $m$-dimensional subset of real points (by the implicit function theorem). After perturbing $\pbar$ inside the smooth real locus of $L$, we may assume that $P_k$'s are defined in a neighbourhood of $\rbar := \pr_{1,\ldots,m} \pbar \in \R^m$ and map the real points there to the reals.

\begin{proposition}\label{thm: weak MEC for upper half}
    Let $V, L$, and $P_k$ be as above, with $P_k$'s defined in a neighbourhood of a real point $\rbar \in \R^m$ and mapping the real points there to the reals. Assume further that for all $m<k\leq n, \, 1\leq l \leq m$ we have $\frac{\partial P_k}{\partial x_l}(\rbar) > 0$ whenever $\frac{\partial P_k}{\partial x_l}\not\equiv 0$. If $V$ is broad and free then $V \cap \gr_+(j) \neq \emptyset$.
\end{proposition}

We need the following lemma in the proof.

\begin{lemma}\label{lem: pos det upper half}
    Let $F$ be a non-constant holomorphic function defined on some domain $\Omega\seq \C^m$ containing a real point $\rbar \in \R^m$. Assume that $F(\Omega \cap \R^m) \seq \R$ and every partial derivative of $F$ that is not identically zero is strictly positive at $\rbar$. Then there is an open neighbourhood $U \seq \Omega$ of $\rbar$ such that $F(U \cap \uh^m) \seq \uh$.
\end{lemma}
\begin{proof}
    Let $U \seq \Omega$ be a small polydisc centred at $\rbar$ where $\re \frac{\partial F}{\partial z_l}$ is either identically zero or everywhere positive. Then, for $\zbar := \xbar + \iu \ybar \in  U \cap \uh^m$ we have
    \[ F(\zbar) - F(\xbar) = \int_0^1 \sum_{l=1}^m \iu y_l\frac{\partial F}{\partial z_l}(\xbar + t \iu \ybar)\,dt.\]
    Therefore, 
    \[ \im F(\zbar) =  \im (F(\zbar)- F(\xbar)) = \int_0^1 \sum_{l=1}^m  y_l \re \frac{\partial F}{\partial z_l}(\xbar + t \iu \ybar)\,dt > 0.\]
    The last inequality follows from the fact that all terms in the sum are non-negative with at least one positive, for $F$ is non-constant.
\end{proof}

\begin{proof}[Proof of \Cref{thm: weak MEC for upper half}]
    This is a straightforward consequence of the proof of \Cref{thm: EC full}. First, we can slightly perturb $\rbar$ and ensure it is generic for the purposes of the argument of \S\ref{sec: proof of main thm} while maintaining the positivity of the partial derivatives. Using the notation of \S\ref{sec: proof of main thm}, we have $T= \{ m+1, \ldots, n \}$ and $F_k(\zbar) = P_k(\zbar)$. The point $\taubar=(\tau_1, \ldots, \tau_m)$ in \S\ref{subsec: limit eqns}, which is an isolated solution of the limiting system of equations, can always be chosen from $\uh^m$, for $j$ is surjective on $\uh$. By \Cref{rem: variable selecting map and positive det} (with the target transformations $h_k$ chosen in $\PSL_2(\R)$), we can choose $\gamma_{k,s}, \zeta_{l,s} \in \PSL_2(\Z)$, so the argument of \S\ref{subsec: final step} goes through and produces a sequence $\taubar_s \to \taubar$ such that $\zetabar_s \taubar_s \to \rbar$ and $\zetabar_s \taubar_s$ is a solution to the desired system of equations arising from $V$. Then, for large enough $s$, we have $\taubar_s \in \uh^m$ and so $\zetabar_s \taubar_s \in \uh^m$. Since every non-zero partial derivative $\frac{\partial F_k}{\partial z_l}(\rbar)$ is strictly positive, by \Cref{lem: pos det upper half} there is a small neighbourhood $U$ of $\rbar$ such that $F_k(U \cap \uh^m) \seq \uh$. Since $\zetabar_s \taubar_s \to \rbar$, for large enough $s$ the point $\zetabar_s \taubar_s \in U \cap \uh^m$ and so $F_k(\zetabar_s \taubar_s) \in \uh$. Thus, we have
    \[ \tilde{\taubar}_s:= (\zetabar_s \taubar_s, F_{m+1}(\zetabar_s \taubar_s), \ldots, F_n(\zetabar_s \taubar_s)) \in  \uh^n\]
    and $(\tilde{\taubar}_s, j(\tilde{\taubar}_s)) \in V$.
\end{proof}

Note that the restriction $\dim V =n$ is only a convenience here. The proof still goes through without this assumption as in \S\ref{subsec: proof of density}. 

If $\dim L =n$, or $L$ is defined by $\PSL_2(\R)$-relations and $V=L\times W$ is broad and free for some $W \seq \C^n$, then $V \cap \gr_+(j) \neq \emptyset$, recovering the results of \cite{eterovic-herrero,Gallinaro-raising-to-powers} mentioned in \S\ref{subsec: extend to h counterex}. 

A simple combination of the proofs of \Cref{thm: secondary thm density} and \Cref{thm: weak MEC for upper half} establishes the following statement.

\begin{proposition}
    Let $L \seq \C^n$ be a product of curves defined over $\R$, and let $\pbar$ be a smooth real point of $L$. Assume the tangent space of $L$ at $\pbar$ meets $\R_{>0}^n$. Then for any $W$ for which $V:=L \times W$ is broad and free, $j(L \cap \uh^n) \cap W$ is dense in $W$ in the Euclidean topology.
\end{proposition}
\begin{proof}
    Since $V$ is broad and splits as a product, it has a convenient basis of coordinates. The algebraic functions locally parametrising $L$ are one-variable functions. A positive vector in the tangent space guarantees that the derivative of each algebraic function with respect to its variable is positive. So we can apply \Cref{thm: weak MEC for upper half} and deduce that $j(L \cap \uh^n) \cap W \neq \emptyset$. The density of this set in $W$ can be proven along the same lines as \Cref{thm: secondary thm density}.
\end{proof}

In particular, if $L$ is defined by $\PSL_2(\R)$-relations and $L\times W$ is broad and free then $j(L\cap \uh^n) \cap W$ is Euclidean dense in $W$. This recovers \cite[Thm.~3.17]{Gallinaro-raising-to-powers}. As another consequence, the set $\left\{ (j(z), j(z^2)): 0 < \arg z < \frac{\pi}{2} \right\}$ is Euclidean dense in $\C^2$.

Broadness and freeness are algebraic conditions. More precisely, broadness is definable in the field structure of $\C$ and so is the condition of non-constant coordinates. This means that in a parametric family $V_{\cbar}$ of algebraic varieties, where $\cbar$ ranges over some constructible set, the set of $\cbar$ for which $V_{\cbar}$ is broad is definable in $\C$. On the other hand, freeness is an intersection of countably many algebraic conditions and is not itself definable. For instance, the variety $x_2=x_1+c$ is free if and only if $c \notin \Q$, so the definability of freeness would imply the definability of $\Q$ in $\C$, which is not the case.

The following result shows that for a broad and free variety it is not possible to guarantee non-empty intersection with $\gr_+(j)$ by avoiding countably many semi-algebraic relations as in freeness.

\begin{proposition}\label{Prop: no semialg condition for upper half}
    Given any countable collection $\Sigma$ of semi-algebraic sets in $\C^{8}$ of positive real codimension, there is a broad and free algebraic variety $V \seq \C^4 \times \C^4$ such that $V \cap (\uh^4 \times \C^4) \neq \emptyset$ and $V$ is not contained in any member of $\Sigma$ and $V \cap \gr_+(j) = \emptyset$.
\end{proposition}
\begin{proof}
    Let $R>0$ be such that whenever $\im z> R$, we have $|j(z)|>1$. Let $V_{c_1, c_2, c} \seq \C^4 \times \C^4$ be given by \[ x_1 = x_3+ c_1,~ x_2 = x_4+c_2,~ y_1y_2 = c.\]
    Define \[ C:= \{ (c_1,c_2,c) \in \C^3: \im c_1 > R, \im c_2>R, 0<|c|<1 \}.\] Whenever $(c_1, c_2, c) \in C$, this variety is broad and free and does not meet $\gr_+(j)$. 

    For every non-empty open subset $C'$ of $C$, the real dimension of the set \[ \{ (c_1, c_2, c, x_1, \ldots, x_4, y_1, \ldots,y_4): (c_1, c_2, c ) \in C',  ( x_1, \ldots, x_4, y_1, \ldots,y_4) \in V_{c_1,c_2,c} \} \] is $16$. So, given a semi-algebraic set $S$ of real dimension $<16$, the set \[ C_S:=\{ (c_1, c_2, c ) \in C: V_{c_1,c_2,c} \seq S \}, \] which is semi-algebraic by the Tarski--Seidenberg theorem, cannot contain an open subset of $C$, hence its real dimension is $<6$. Thus $C_S$ is nowhere dense in $C$ and so $\bigcup_{S \in \Sigma} C_S$ is meagre in $C$. By the Baire category theorem we can choose $(c_1,c_2,c) \in C \setminus \bigcup_{S \in \Sigma} C_S$ and ensure that $V_{c_1,c_2,c}$ satisfies the desired properties.
\end{proof}

Note that the statement would be trivially false without the condition that the members of $\Sigma$ have positive codimension. Indeed, let $S:= \C^8 \setminus \{ \pbar \}$ where $\pbar \in \gr_+(j)$. Then $V \nsubseteq S$ implies $\pbar \in V$ and so $V \cap \gr_+(j) \neq \emptyset$. So the positive codimension assumption makes the condition $V \nsubseteq S$ meaningful.

\section{An approximation theorem for analytic functions}\label{sec: dynamics}

In this section we prove \Cref{thm: density of analytic orbits general version} -- a simultaneous approximation statement allowing us to start with arbitrary holomorphic functions with certain properties and use the action of $\PGL_2(\Z)$ to approximate M\"obius transformations. The theorem is restated and proved in \S\ref{subsec: density general case}. Below we present a roadmap to the proof to make the argument easier to follow.

\begin{enumerate}
   \item In the simplest case, we have a non-constant holomorphic function $F(z)$, which is not a $\PGL_2(\Q)$-transformation, defined around a real point $r$ and mapping a real neighbourhood of $r$ to $\R$. We want to show that for any $h \in\PGL_2(\R)$ we can find sequences $\gamma_s, \zeta_s \in  \PGL_2(\Z)$ such that $\zeta_s z \to r$ locally uniformly for $z \in \h$ and the composition $\gamma_s \circ F \circ \zeta_s$ converges locally uniformly to $h$.

    \item To that end, we approximate $F$ by a M\"obius transformation up to second order, i.e.\ we find $g \in \PGL_2(\R)$ such that $F^{(i)}(r) = g^{(i)}(r),\, i=0,1,2$. This is possible if $F'(r) \neq 0$, and if $F'(r)=0$ we approximate $r$ by real points where the derivative does not vanish and apply a diagonal argument. In fact, by replacing $F$ with $-F$, we may assume $F'(r) > 0$ and so $g \in \PSL_2(\R)$. Assuming $h \in \PSL_2(\R)$ allows us to work entirely in this group.

    \item We map $r$ to $\infty$ by a suitable M\"obius transformation and use this to conjugate the horocycle flow $u_t :=\begin{pmatrix}         1 & t \\ 0 & 1     \end{pmatrix},\, t \in \R$. It yields another unipotent flow whose orbits tend to $r$ instead of $\infty$. Then we apply Ratner's orbit closure and equidistribution theorems to the orbit of $(I,g)$ under the diagonal action of this flow in the quotient $\PSL_2(\Z)^2 \backslash \PSL_2(\R)^2$. This helps us establish the desired approximation for $g$ instead of $F$.

    \item Here, analysis takes over from dynamics. The second-order approximation of $F$ by $g$ results in a third-order error for $F-g$, which is amplified by a second-order factor coming from modular transformations. This leaves a first-order error tending to $0$, which implies that $\gamma_s \circ F \circ \zeta_s$ and $\gamma_s \circ g \circ \zeta_s$ must have the same limit.

    \item In the case of several functions of a single variable, we follow the same approach but work with higher powers of $\PSL_2(\R),\, \PSL_2(\Z)$ and the diagonal action of the horocycle flow.

    \item When we have a function of several variables, say $F(z_1, z_2)$ defined around $\rbar := (r_1, r_2) \in \R^2$, we apply the one-variable result to $F(z_1, r_2)$ and find $\gamma_s, \zeta_{1,s} \in \PGL_2(\Z)$ such that $\gamma_s F(\zeta_{1,s} z_1, r_2) \to h z_1$. Then we choose another sequence $\zeta_{2,s}$ such that $\zeta_{2,s} z - r_2$ approaches $0$ much faster than $\zeta_{1,s} z - r_1$, for $z$ in any compact subset of $\h$. This allows us to conclude that $\gamma_s F(\zeta_{1,s} z_1, \zeta_{2,s} z_2) \to h z_1$ locally uniformly on $\h^2$. Here we assume the partial derivatives of $F$ that do not vanish identically are non-zero at $\rbar$; otherwise we can approximate $\rbar$ by suitably generic points and use a diagonal argument.

    \item In the general case of several functions of several variables, we first choose a suitable variable for each function given by the map $\sigma$ in \Cref{thm: density of analytic orbits general version}, and follow the approach of the previous step to conclude by induction on the number of variables. 
\end{enumerate}

\subsection{Ratner's theorems and some consequences}

We now state Ratner's orbit closure and equidistribution theorems, which were first proved in \cite{Ratner1991OrbitClosure}. We refer the reader to \cite{Morris--Ratner} for an exposition of these theorems and their applications in various examples.

\begin{fact}[Ratner's orbit closure theorem, {\cite[Thm.~A]{Ratner1991OrbitClosure}, \cite[Rmk.~1.1.15]{Morris--Ratner}}]\label{Ratner orbit closure continuous}
Let $G$ be a Lie group, $\Gamma \seq G$ a lattice, and $U := \{ u_t: t \in \R \} \seq G$ a unipotent one-parameter subgroup. Let also $g \in G$. Then there is a connected closed subgroup $S$ of $G$ containing $U$ such that $g^{-1}\Gamma g \cap S$ is a lattice in $S$ and $\overline{\Gamma g U} = \Gamma g S$.    
\end{fact}

\begin{fact}[Ratner's equidistribution theorem, {\cite[Thm.~B]{Ratner1991OrbitClosure}, \cite[Thm.~1.3.4]{Morris--Ratner}}]\label{Ratner equidist thm}
    Let $G$ be a Lie group, $\Gamma \seq G$ a lattice, and $U := \{ u_t: t \in \R \} \seq G$ a unipotent one-parameter subgroup. Let also $g \in G$. Then the orbit $\Gamma g u_t$ is uniformly distributed in its closure. In particular, for any $T >0$ the set $\Gamma g \{ u_t: t \geq T \}$ is dense in $\overline{\Gamma g U}$.
\end{fact}

Recall that a \textit{lattice} in a Lie group is a discrete subgroup of finite covolume. For instance, $\Gamma:= \PSL_2(\Z)$ is a lattice in $G:=\PSL_2(\R)$, and $\Gamma^N$ is a lattice in $G^N$ for any $N \in \Z_{>0}$. We fix this notation for the rest of this section. Let $u_t := \begin{pmatrix} 1 & t \\ 0 & 1 \end{pmatrix}$, and let $U := \{ u_t: t \in \R \} \seq G$. We will apply Ratner's theorems to $G^N$ and $\Gamma^N$, with the unipotent flow given by diagonal powers of $U$, that is, \[ \Delta_N(U) := \{ (u_t, \ldots, u_t) \in U^N: t \in \R \}.\]

We also let $\mathcal{C} \seq G$ be the image of $\PGL_2^+(\Q)$ in $G$ under the map $g \mapsto \frac{1}{\sqrt{\lvert \det(g) \rvert}} g$. Then $\mathcal{C}$ is the commensurator of $\Gamma$, i.e.\ the subgroup of all $g \in G$ such that the intersection of $g^{-1}\Gamma g$ and $\Gamma$ has finite index in both groups (equivalently, $g^{-1}\Gamma g \cap \Gamma$ is a lattice in $G$).

We will need the following classification of connected subgroups of $G$ containing $U$.

\begin{lemma}\label{lem: connected subgps of SL_2(R)}
   The only connected subgroups of $G$ containing $U$ are $U$, the upper triangular subgroup $B$, and $G$.
\end{lemma}
\begin{proof}
    This is a well-known fact; it is mentioned, for instance, in \cite[Ex.~1.1.16]{Morris--Ratner} (without proof). If we assume the connected subgroup is a Lie subgroup, then we can use the Lie subalgebra--subgroup correspondence and apply the classification of Lie subalgebras of $\mathfrak{sl}_2(\R)$ (see, e.g., \cite{PW77}). However, it seems hard to find an exact reference (with proof) for our lemma, hence we give an elementary proof below.\footnote{This proof is due to ChatGPT 6 Astra.}

    Let $S \seq G$ be a connected subgroup containing $U$. The upper triangular subgroup $B$ consists of matrices $\begin{pmatrix} a & b \\ 0 & a^{-1}    \end{pmatrix}$ with $a \in \R_{>0}, b \in \R$.

    Consider two cases.

    \medskip\noindent\textbf{Case 1.} $S \seq B$.

    Consider the map $\varphi\colon  B \to \R$ given by $\begin{pmatrix} a & b \\ 0 & a^{-1}    \end{pmatrix} \mapsto \log a$. This is a continuous group homomorphism onto $(\R; +)$ with $\ker(\varphi) = U$. Since $S$ is connected, $\varphi(S)$ is a connected subgroup of $(\R; +)$, hence $\varphi(S)$ is either $\{ 0 \}$ or $\R$. Therefore, $S$ is either $U$ or $B$. 

    \medskip\noindent\textbf{Case 2.} $S \not\subseteq B$.

    Pick $g = \begin{pmatrix} a & b \\ c & d    \end{pmatrix} \in S$ with $c \neq 0$. Since $U \seq S$, we have \[ h:= u_{-\sfrac{a}{c}} g = \begin{pmatrix} 1 & -\sfrac{a}{c} \\ 0 & 1    \end{pmatrix}\begin{pmatrix} a & b \\ c & d    \end{pmatrix} = \begin{pmatrix} 0 & -\sfrac{1}{c} \\ c & d    \end{pmatrix} \in S. \]
    Then \[ hu_th^{-1} = \begin{pmatrix} 0 & -\sfrac{1}{c} \\ c & d    \end{pmatrix} \begin{pmatrix} 1 & t \\ 0 & 1    \end{pmatrix} \begin{pmatrix} d & \sfrac{1}{c} \\ -c & 0    \end{pmatrix} = \begin{pmatrix} 1 & 0 \\ -c^2t & 1    \end{pmatrix} \in S. \]

    Since $t \in \R$ is arbitrary, we conclude that $S$ contains the lower unitriangular subgroup $U^{\rm T} := \{ u_t^{\rm T}: t \in \R \}$. But $U$ and $U^{\rm T}$ generate $G$ (it is straightforward to check that $G = UU^{\rm T}UU^{\rm T}$), so $S=G$.
\end{proof}

\begin{lemma}[{cf. \cite[Thm.~2.6]{Hed36} and \cite[Thm.~2.3(3)]{McM07}}]\label{lem: Gamma g U dense}
    Let $g \in G$ be such that $g (\infty) \notin \p^1(\Q) = \Q \cup \{ \infty \}$. Then $\Gamma g U$ is dense in $G$.
\end{lemma}

\begin{proof}
    First, we show that if $g^{-1}\Gamma g \cap U$ is non-trivial then $g (\infty) \in \p^1(\Q)$. Indeed, if this intersection is non-trivial then for some $t \in \R^{\times}$ we have $g u_t g^{-1} \in \Gamma$. Writing $g = \begin{pmatrix}
        a & b \\ c & d
    \end{pmatrix}$ we get \[ g u_t g^{-1} =  \begin{pmatrix}
        1- act & a^2 t \\ -c^2 t & 1 + act
    \end{pmatrix}.\]

    If this matrix is in $\Gamma$ then all of its entries are integral, hence $g (\infty) = \frac{a}{c} = \frac{a^2 t}{act } \in \Q \cup \{ \infty \}$. Note that when $a=0$ or $c=0$ we have $g(\infty) \in \{ 0, \infty \}$.    Since we assumed $g (\infty) \notin \p^1(\Q) = \Q \cup \{ \infty \}$, we conclude that $g^{-1}\Gamma g \cap U$ is trivial. 
    
    By Ratner's orbit closure theorem (Fact~\ref{Ratner orbit closure continuous}), the closure of $\Gamma g U$ is of the form $\Gamma g S$ where $S$ is a connected closed subgroup containing $U$, and $g^{-1}\Gamma g \cap S$ is a lattice in $S$. By \Cref{lem: connected subgps of SL_2(R)}, $S$ is either $U$, or the upper triangular subgroup $B$, or all of $G$. The first case is not possible since $g^{-1}\Gamma g \cap U$ is trivial and so is not a lattice in $U$. On the other hand, $B$ does not have any lattices. To see this, we invoke \cite[Prop.~9.20]{EinsiedlerWard2011} which states that a closed linear group with a lattice must be unimodular -- the left and right Haar measures agree -- and \cite[Ex.~C.5(4)]{EinsiedlerWard2011} which shows that $B$ is not unimodular. We conclude that $S=G$, and $\overline{\Gamma g U} = \Gamma g G =  G$.
\end{proof}

\begin{proposition}[{cf. \cite[Thm.~2.6(5)]{McM07}}]\label{prop: Gamma g_0...g_N closure}
    Let $\gbar: = (g_1, \ldots, g_N) \in G^N$ be such that  $g_1(\infty), \ldots, g_N(\infty)$ are in distinct $\mathcal{C}$-orbits and not in $\p^1(\Q)$. Then $\Gamma^{N} \gbar \Delta_{N}(U)$ is dense in $G^{N}$.
\end{proposition}

\begin{proof}
    Let $S$ be the connected closed subgroup of $G^N$ given by Ratner's orbit closure theorem. Thus, $\gbar^{-1} \Gamma^N \gbar \cap S$ is a lattice in $S$, and $\overline{\Gamma^N \gbar \Delta_N(U)} = \Gamma^N \gbar S$. 
    
    The projection of $S$ to the $i$-th copy of $G$ need not a priori be closed, but it is a connected Lie subgroup containing $U$. Hence, by \Cref{lem: connected subgps of SL_2(R)}, each projection is either $U, B$, or $G$. Moreover, since $g_i^{-1}\Gamma g_i \cap \pr_i S \supseteq \pr_i (\gbar^{-1} \Gamma^N \gbar \cap S)$, and the latter is a lattice in $\pr_i S$ -- it is discrete as it is contained in $g_i^{-1}\Gamma g_i$, and we can push forward the finite invariant measure on the quotient along the projection -- so must be the former. Therefore, as in the proof of \Cref{lem: Gamma g U dense}, each projection is equal to $G$.

    We claim that $S = G^N$. If not, then we can apply Goursat's lemma (see \cite[Prop.~5.4.11]{Gallinaro-thesis} for details) to conclude that the projection of $S$ to some two coordinates is the graph of an automorphism of $G$. The Lie group automorphisms of $G$ are of the form $g \mapsto vgv^{-1}$ for some $v \in \PGL_2(\R)$ (see \cite[Ch.~IV, \S1, case~III, \S2, and \S6]{Die71}). Therefore, there are indices $k,k'$ and a matrix $v \in \PGL_2(\R)$ such that for any $\habar:=(h_1, \ldots, h_N) \in S$ we have $h_k  = v h_{k'} v^{-1}$. In particular, since $S$ contains $\Delta_N(U)$, we must have \[ u_t = v u_t v^{-1} \text{ for any } t.\] In other words, $v$ commutes with all $u_t$ and since the centraliser of $U$ in $\PGL_2(\R)$ is $U$, we conclude that $v \in U$.

   Writing $S_{k,k'}$ for the projection of $S$ to the $(k,k')$-coordinates, and taking into account that $\gbar^{-1} \Gamma^N \gbar \cap S$ is a lattice in $S$, we get that
    \[ (g_k,g_{k'})^{-1} \Gamma^2 (g_k, g_{k'}) \cap S_{k,k'} \]  is a lattice in  $S_{k,k'}$ (as above). Taking the projection to the $k$-coordinate, we conclude that \[ g_k^{-1}\Gamma g_k \cap v (g_{k'}^{-1} \Gamma g_{k'}) v^{-1}\] is a lattice in $G$. This then implies that $g_k^{-1}\Gamma g_k$ and $v g_{k'}^{-1} \Gamma (vg_{k'}^{-1})^{-1}$ are commensurable, i.e.\ their intersection has finite index in both groups, which happens if and only if $\kappa:= g_kvg_{k'}^{-1} \in \mathcal{C}$, for $\mathcal{C}$ is exactly the commensurator of $\Gamma$. Then $g_kv = \kappa g_{k'}$ and hence, taking into account that $v(\infty) = \infty$, \[ g_k(\infty) = g_kv(\infty) = \kappa g_{k'}(\infty) \] which contradicts our assumption that $g_k(\infty)$ and $g_{k'}(\infty)$ are in distinct $\mathcal{C}$-orbits.\footnote{In this proof, the steps where we show that $v \in U$ and that $g_k^{-1}\Gamma g_k \cap v (g_{k'}^{-1} \Gamma g_{k'}) v^{-1}$ is a lattice in $G$ were taken from an argument by ChatGPT 6 Astra.}
\end{proof}

Now, given $r\in \R \cup \{ \infty \}$, we want to find a subgroup $H_r \seq G$ fixing $r$. If $r = \infty$ then $H_{\infty} := U$ fixes $r$. Since the $G$-orbit of $\infty$ is exactly $\R \cup \{ \infty \}$, we can use this to explicitly find $H_r$ for any $r \in \R$. First, observe that the matrix \[ \mu_r := \begin{pmatrix} r & r-1 \\ 1 & 1 \end{pmatrix}\] maps $\infty$ to $r$. Hence, the following subgroup fixes $r$: \[ H_r = \mu_r H_\infty \mu_r^{-1} =  \left\{ \xi_{r,t}:= \mu_r u_t \mu_r^{-1} = \begin{pmatrix} 1 - rt & r^2t \\ -t & 1 + rt \end{pmatrix}: t \in \R \right\}. \]

\begin{lemma}\label{lem: xi_s from H_r to r}
    As $t\to \infty$, the orbit $\xi_{r,t} z$ tends to $r$ locally uniformly on $\h$ (i.e.\ on any compact subset of $\h$).
\end{lemma}
\begin{proof}
    This follows from the facts that $u_t (z) \to \infty$, and so $u_t (\mu_r^{-1} z) \to \infty$ locally uniformly, and that $\mu_r(\infty) = r$. The lemma can also be proved directly.
\end{proof}

For $N \in \Z_{>0}$ we let $\Delta_N(H_r):= \{ (\xi_{r,t}, \ldots, \xi_{r,t}) \in H_r^N: t \in \R \}$.

\begin{proposition}\label{prop: orbit closure irrational N}
    Let $r \in \R$ and $\gbar:= (I, g_1, \ldots, g_N) \in G^{N+1}$, and assume that  $r, g_1(r), \ldots, g_N(r)$ are in distinct $\mathcal{C}$-orbits and outside $\p^1(\Q)$. Then $\Gamma^{N+1} \gbar \Delta_{N+1}(H_r)$ is dense in $G^{N+1}$.
\end{proposition}
\begin{proof}
    This follows from \Cref{prop: Gamma g_0...g_N closure} and the fact that $ H_{r} = \mu_r U \mu_r^{-1}$, taking into account that $g_k \mu_r (\infty) = g_k(r)$.
\end{proof}

\begin{corollary}\label{cor: approx. g_0...g_N extra conditions}
    Let $r \in \R $, and $ g_0 = I, g_1, \ldots, g_N \in G$ be such that  $g_0(r), \ldots, g_N(r)$ are in distinct $\mathcal{C}$-orbits and not in $\p^1(\Q)$. Let $h_0, h_1, \ldots, h_N \in G$ be arbitrary. Then there are sequences of matrices $\gamma_{0,s}, \ldots, \gamma_{N,s} \in \Gamma,\, \eta_s \in H_r$ such that     
    \begin{itemize}
        \item $\eta_s = \xi_{r,t_s}$ with $t_s \to +\infty$ as $s \to \infty$;
        \item for all $0\leq k \leq N$ we have $\gamma_{k,s} g_k \eta_s \to h_k$ as $s \to \infty$.
    \end{itemize}
    In particular, $\eta_s(z) \to r$ locally uniformly on $\h$.
\end{corollary}
\begin{proof}
   This follows immediately from \Cref{prop: orbit closure irrational N} and Ratner's equidistribution theorem (\Cref{Ratner equidist thm}).
\end{proof}

\subsection{A single function of a single variable}
In this subsection we establish the $N=m=1$ case of \Cref{thm: density of analytic orbits general version}.

Recall from \S\ref{sec: prelim} that $G:= \PSL_2(\R)$ and $\Gamma := \PSL_2(\Z)$. Let  $\hat{G}:= \PGL_2(\R)$,  $\hat{\Gamma}:=\PGL_2(\Z)$, and $\hat{\mathcal{C}} := \left\{ \frac{1}{\sqrt{\lvert \det(g) \rvert}} g : g \in \GL_2(\Q) \right\} \seq \hat{G}$.

\begin{proposition}\label{prop: gamma F eta to h with  conditions}
     Let $\Omega \seq \C$ be a complex domain with $\Omega \cap \R \neq \emptyset$.   Let $F \colon  \Omega \to \C$ be a holomorphic function such that $F(\Omega \cap \R) \seq \R$. Assume $r \in \Omega \cap (\R\setminus \Q)$ is such that $F'(r) \neq 0$ and $F(r) \notin \hat{\mathcal{C}}r \cup \Q$. Then for every $h \in \hat{G}$ there are sequences $\gamma_s, \nu_s \in \hat{\Gamma},\, \eta_s \in H_r,\, s \in \Z_{>0},$ with the following properties:
 \begin{itemize}
        \item $\eta_s = \xi_{r,t_s}$ with $t_s \to +\infty$ as $s \to \infty$;

        \item $\nu_s \eta_s \to I$;
        \item $\lVert \gamma_{s} \rVert = O( t_{s} )$;
        \item $\gamma_{s} (F (\eta_s z )) \to hz$ as $s \to \infty$ locally uniformly on $\h$.
 \end{itemize} 
\end{proposition}

In this statement, for a matrix $g$ we let $\| g \|$ denote the maximum of the absolute values of its entries. We also remark that as $z$ varies in a compact subset of $\h$, for large $s$ we have $\eta_s(z) \in \Omega$, hence $F(\eta_s(z))$ is well defined.

To prove this, we approximate $F$ by a M\"obius transformation $g$ at $r$, use the corresponding density statement for $g$ and transfer back to $F$. A second-order approximation by $g$ leaves an error of order $t^{-3}$ while the modular transformations amplify this by $t^2$, so the total error is negligible.

We need the following lemma in the proof of this proposition.
\begin{lemma}\label{lem: bound on ||gamma_s||}
    Let $r \in \R$, $g, h \in G$, $\gamma_s\in \Gamma,\, \eta_s = \xi_{r,t_s}\in H_r,\, s \in \Z_{>0},$ with $t_s>0,\, t_s \to +\infty$. Assume that $\gamma_s g \eta_s$ converges to $h$ in $G$ as $s \to \infty$. Then $\| \gamma_s\| = O(t_s)$. 
\end{lemma}
\begin{proof}
    We use the inequality $\|g_1g_2\| \leq 2\|g_1\| \cdot \|g_2\|$. Observe that $\| \eta_s \| = O(t_s)$ and $\| \eta_s^{-1} \| = O(t_s)$. Since $\gamma_s g \eta_s \to h$, we also have $\| \gamma_s g \eta_s\| = O(1)$. Hence,
    \[  \| \gamma_s\| = \| (\gamma_s g \eta_s) \cdot (\eta_s^{-1} g^{-1}) \| \leq 2 \| \gamma_s g \eta_s\| \cdot 2\| \eta_s^{-1}\| \cdot \|g^{-1} \| = O(t_s). \qedhere \] 
\end{proof}

\begin{proof}[Proof of \Cref{prop: gamma F eta to h with  conditions}]

    First, we observe that it can be assumed without loss of generality that $F'(r) > 0$ and $h \in G$. When $F'(r) < 0$, we can replace $F$ by $-F$, and when $\det(h) < 0$ we can replace $h$ by $\begin{pmatrix}
        1 & 0 \\ 0 & -1 \end{pmatrix} h$.
    
    Since $F'(r)>0$, there is a unique $g \in G$ satisfying 
    \begin{equation}\label{eqn: approx F with g}
        F(r) = g(r), F'(r) = g'(r), F''(r) = g''(r).
    \end{equation}
    Explicitly, \[ g(z) = F(r)+\frac{F'(r)(z-r)}{1-\frac{F''(r)}{2F'(r)}(z-r)}.\]

    The conditions on $F$ guarantee that $g(r) \notin \hat{\mathcal{C}}r \cup \p^1(\Q)$. Applying \Cref{cor: approx. g_0...g_N extra conditions} with $N=1,\, g_1=g,\, h_0=I,\, h_1=h$ supplies sequences $\eta_s \in H_r$,  $\gamma_s := \gamma_{1,s} \in \Gamma$, and $\nu_s := \gamma_{0,s}\in \Gamma$ such that
    \begin{itemize}
        \item $\eta_s = \xi_{r,t_s} \in H_r$ with $t_s \to +\infty$;
        \item $\nu_s \eta_s \to I$;
        \item $\gamma_{s} g \eta_s \to h$ as $s \to \infty$.
    \end{itemize}

    We now estimate $\eta_s(z) - r$. Fix a compact set $K \seq \h$ and let $z$ vary in $K$.
    \begin{claim}\label{claim: eta z - r O}
      We have  $\eta_s(z) - r = O_K(t_s^{-1})$.
    \end{claim}
    \begin{claimproof}
           For $z \in K$ we have
     \begin{align*}
        \left| \xi_{r,t} z - r \right| & =   \left|\frac{(1-rt) z + r^2t}{-t z + (1+rt)} + \frac{1-rt}{t} - \frac{1}{t}\right|    \\ 
        & \leq  \frac{1}{\lvert -t z + (1+rt)\rvert t} + \frac{1}{t}  \leq   \frac{1}{t^2 \lvert\im z\rvert}+\frac{1}{t} = O_K\left( \frac{1}{t} \right). 
    \end{align*}
    
    Thus, $\eta_s(z) - r = O_K(t_s^{-1})$.
    \end{claimproof}

    Next we observe that $F(\eta_s z)$ is asymptotically very close to $g(\eta_s z)$.

    \begin{claim}\label{claim: F-g bound}
        There is a constant $B_K$ such that $|F(\eta_s z) - g(\eta_s z)| \leq B_K t_s^{-3}$ for all $z \in K$.
    \end{claim}
    \begin{claimproof}
    By looking at Taylor expansions, we get
    \begin{align*}
        F(\eta_s z) = F(r) + F'(r) (\eta_s z - r) + \frac{F''(r)}{2} \cdot (\eta_s z - r) ^2 + O_K(|\eta_s z - r|^3),
    \end{align*}
    and 
    \begin{align*}
        g(\eta_s z) = g(r) + g'(r) (\eta_s z - r) + \frac{g''(r)}{2} \cdot (\eta_s z - r) ^2 + O_K(|\eta_s z - r|^3).
    \end{align*}

    Hence, by \eqref{eqn: approx F with g} and \Cref{claim: eta z - r O},
    $ F(\eta_s z) - g(\eta_s z) =  O_K (t_s^{-3}).$ 
    \end{claimproof}

    Let $\gamma_s := \begin{pmatrix} a_s & b_s \\ c_s & d_s  \end{pmatrix}$. By \Cref{lem: bound on ||gamma_s||}, $\| \gamma_s \| = O(t_s)$, so  $a_s, b_s, c_s, d_s$ are all $O(t_s)$.

    \begin{claim}\label{claim: gamma_s^i bound}
         There is a constant $A_K$ such that for all $z \in K$ and for all $s, i \in \Z_{>0}$ \[ |\gamma_s^{(i)} (g (\eta_sz))| \leq  i! A_K^i t_s^{2i}. \]
    \end{claim}
   \begin{claimproof}
          Since $\gamma_s (g (\eta_s z))$ converges uniformly to $h(z)$ in an open neighbourhood of $K$, we can differentiate and get
    \[ \gamma'_s (g(\eta_s z)) \cdot g'(\eta_s z) \cdot \eta_s' (z) \to h'(z) \text{ for } z \in K.  \]

    Since $\eta_s = \xi_{r,t_s}$ we can compute its derivative: $\eta'_s (z) = \frac{1}{(-t_s z + 1+rt_s)^2 }.$ Taking into account that $g'$ is continuous at $r$ and, therefore, $g'(\eta_s z) \to g'(r) \neq 0$, we conclude
    \begin{equation}\label{eqn: gamma' ( f eta z))}
        \gamma'_s (g (\eta_s z)) \sim \frac{h'(z)}{g'(r)} \cdot (-t_s z + 1+rt_s)^2 = O_K(t_s^2).
    \end{equation}

   Since $\gamma_s' (w) = (c_sw+d_s)^{-2}$, \eqref{eqn: gamma' ( f eta z))} implies
   \begin{equation}\label{eqn: c_sg+d_s bound}
       |c_s g(\eta_s z) + d_s|^{-1} = O_K(t_s).
   \end{equation}

    Finally, for $i \in \Z_{>0}$ we have $\gamma_s^{(i)} (w) = (-1)^{i-1} i! (c_sw+d_s)^{-(i+1)} c_s^{i-1}$, so
    \[ |\gamma_s^{(i)} (g (\eta_sz))| = i! |c_s (g (\eta_sz)) + d_s|^{-(i+1)} |c_s|^{i-1} =i! O_K(t_s)^{2i}. \qedhere \]
   \end{claimproof}

   \bigskip

    Observe that $F(\eta_sz)$ is within the radius of convergence of the Taylor expansion of $\gamma_s(w)$ at $w=g(\eta_s z)$ (for large $s$). Indeed, if $c_s=0$, the radius of convergence is infinite; otherwise it is equal to the distance to the nearest pole, which is
    \[ \left\lvert w+ \frac{d_s}{c_s}\right\rvert = \frac{\lvert c_sw+d_s\rvert}{\lvert c_s \rvert} \geq C_Kt_s^{-2} \]
    for some constant $C_K$ by \eqref{eqn: c_sg+d_s bound} and \Cref{lem: bound on ||gamma_s||}. For large $s$, this dominates $|F(\eta_s z) - g(\eta_s z)|$ by \Cref{claim: F-g bound}. So, we conclude
    \begin{align*}
        & |\gamma_s(F(\eta_sz)) - \gamma_s(g(\eta_s z)) |  \leq  \sum_{i=1}^\infty \frac{|\gamma^{(i)}_s(g(\eta_s z))|}{i!}\cdot |F(\eta_s z) - g(\eta_s z)|^i \\
                                                        & \leq \sum_{i=1}^\infty A_K^i t_s^{2i}\cdot B_K^i t_s^{-3i} 
                                                        = \sum_{i=1}^\infty (A_KB_K t_s^{-1})^i = \frac{A_KB_Kt_s^{-1}}{1 - A_KB_Kt_s^{-1}}
    \end{align*}
    which tends to $0$ as $s\to \infty$.  Since $\gamma_s(g(\eta_s z))\to hz$ uniformly on $K$, it follows that $\gamma_{s} (F (\eta_s z )) \to hz$ uniformly on $K$.
\end{proof}

\begin{theorem}\label{thm: gamma F eta to h}
 Let $\Omega \seq \C$ be a complex domain with $\Omega \cap \R \neq \emptyset$.   Let $F \colon  \Omega \to \C$ be a holomorphic function such that $F(\Omega \cap \R) \seq \R$. Assume $F$ is non-constant and is not a linear fractional transformation with rational coefficients. Then for every $r \in \Omega \cap \R$ and for every $h \in \hat{G}$ there are sequences $\gamma_s, \zeta_s \in \hat{\Gamma},\, s \in \Z_{>0},$ with the following properties:
 \begin{itemize}
     \item $\zeta_s(z) \to r$  locally uniformly for $z$ in $\h$ as $s\to \infty$;
     \item $\gamma_s F(\zeta_s z) \to hz$ locally uniformly on $\h$ as $s\to \infty$.
 \end{itemize} 
\end{theorem}
\begin{proof} First, we establish the following claim.
\begin{claim}
    The set $R:= \{ r \in \Omega \cap (\R \setminus \Q): F'(r) \neq 0,\, F(r) \notin \hat{\mathcal{C}}r \cup \Q \}$ is dense in $\Omega\cap \R$. 
\end{claim}
\begin{claimproof}
     Since $F$ is non-constant, $F'$ does not vanish identically, so it has at most countably many zeros. For any fixed $g \in \hat{\mathcal{C}}$ the set $\{ x: F(x)=g(x) \}$ is at most countable, for $F$ is not identically equal to $g$, so the set \[ \{ x \in \Omega \cap \R: \exists g \in \hat{\mathcal{C}} \text{ such that } F(x) = g(x)\} \] is also countable. Since $F$ is non-constant, $F^{-1}(\Q)$ is countable too. Therefore, the complement of $R$ in $\Omega \cap \R$ is countable, so $R$ is dense in $\Omega\cap \R$.
\end{claimproof}

 Now let $r$ be an arbitrary point in $\Omega \cap \R$ and pick $r_s \in R$ such that $r_s \to r$ as $s \to \infty$. Fix a compact set $K \seq \h$. By \Cref{prop: gamma F eta to h with  conditions}, for each $s$ there are $\gamma^K_s, \nu^K_s \in \hat{\Gamma}$ and  $\eta^K_s = \xi_{r_s, t_s} \in H_{r_s}$ such that 
\begin{itemize}
    \item $|\eta^K_s(z) - r_s| < s^{-1}$ for all $z \in K$;
    \item $\|\nu^K_s \eta^K_s - I\| < s^{-1}$;
    \item $|\gamma^K_s(F(\eta^K_s z)) - h(z)| < s^{-1}$ for all $z \in K$.
\end{itemize}

Now let $K_1 \subseteq K_2 \subseteq \ldots$ be an increasing chain of compact subsets of $\h$ such that each $K_s$ is contained in the interior of $K_{s+1}$ and $\bigcup_{s=1}^{\infty} K_s = \h$. Let $\eta_s := \eta^{K_s}_s, \nu_s := \nu^{K_s}_s, \gamma_s := \gamma^{K_s}_s$. Since any compact subset of $\h$ must be contained in one of the $K_s$'s, the sequence $\gamma_{s} (F(\eta_s (z)))$ approaches $hz$ locally uniformly on $\h$, and $\eta_s(z)$ tends to $r$ locally uniformly. 

To finish the proof, we just observe that these statements remain valid if we replace $\eta_s$ with $\zeta_s := \nu_s^{-1}$. Indeed, for large $s$, the matrix $\nu_s \eta_s$ is close to $I$, so as $z$ varies in a compact set $K$, the image $(\nu_s \eta_s)^{-1}(z)$ varies in a slightly larger compact set $K' \seq \h$. Therefore, $\zeta_s(z) = \eta_s ( (\nu_s \eta_s)^{-1} z ) $ tends uniformly to $r$.

Similarly, $\gamma_s(F(\zeta_s z)) = \gamma_s(F(\eta_s ( (\nu_s \eta_s)^{-1} z )))$ converges uniformly to $hz$.
\end{proof}

\subsection{The general case}\label{subsec: density general case}

In this subsection we prove \Cref{thm: density of analytic orbits general version}, but we look at another special case first.

\begin{proposition}\label{prop: gamma F_1...F_N eta to h with conditions single variable}
     Let $\Omega \seq \C$ be a complex domain with $\Omega \cap \R \neq \emptyset$.   Let $F_1, \ldots,F_N \colon \Omega \to \C$ be holomorphic functions such that $F_k(\Omega \cap \R) \seq \R$ for every $k$. Assume $r \in \Omega \cap (\R \setminus \Q)$ is such that $F_k'(r) \neq 0$, $F_k(r) \notin \hat{\mathcal{C}}r \cup \Q$ for all $k$, and $F_k(r) \notin \hat{\mathcal{C}}F_{k'}(r)$ for all $k \neq k'$. Then for any $h_1, \ldots, h_N \in \hat{G}$ there are sequences $\gamma_{k,s}, \nu_s\in \hat{\Gamma},\, \eta_s \in H_r,\, k=1, \ldots,\, N, \, s \in \Z_{>0},$ with the following properties:
 \begin{itemize}
        \item $\eta_s = \xi_{r,t_s}$ with $t_s \to +\infty$ as $s \to \infty$;
        \item $\nu_s \eta_s \to I$;
        \item for all $k$, as $s \to \infty$, we have $\gamma_{k,s} (F_k (\eta_s z)) \to h_k z$ locally uniformly on $\h$.
 \end{itemize} 
\end{proposition}
\begin{proof}
    This can be proved exactly as \Cref{prop: gamma F eta to h with  conditions} by approximating each $F_k$ by a M\"obius transformation $g_k$ up to second order. The conditions on the values $F_k(r)$ allow us to apply \Cref{cor: approx. g_0...g_N extra conditions}.
\end{proof}

We now restate the main theorem of this section for the convenience of the reader.

\begin{customthm}{\ref*{thm: density of analytic orbits general version}}
      Let $F_1, \ldots, F_N$ be non-constant holomorphic functions of variables $z_1, \ldots, z_m$ defined on some open domain $\Omega \seq \C^m$ with $\Omega \cap \R^m \neq \emptyset$. Let $B \seq \Omega \cap \R^m$ be a non-empty open box. Assume the functions $z_1, \ldots, z_m, F_1, \ldots, F_N$ are in distinct $\hat{\mathcal{C}}$-orbits, and $F_k(B) \seq \R$ for all $k$. Then there is a map $\sigma \colon \{ 1, \ldots, N \} \to \{ 1, \ldots, m \} $ with the following property: for any $\rbar:=(r_1, \ldots, r_m) \in B$ and for any $h_k \in \hat{G},\, 1 \leq k \leq N,$ there are sequences of matrices $(\gamma_{k,s})_s, (\zeta_{l,s})_s \seq \hat{\Gamma},\, 1\leq k \leq N,\, 1\leq l \leq m,$ such that  
    \begin{itemize}
        \item for all $l$, as $s \to \infty$, we have $\zeta_{l,s} z \to r_l$ locally uniformly for $z \in \h$;

        \item for all $k$, as $s \to \infty$, we have $\gamma_{k,s} F_k(\zeta_{1,s} z_1, \ldots, \zeta_{m,s} z_m ) \to h_k z_{\sigma (k)}$  locally uniformly on $\h^m$.
    \end{itemize}
\end{customthm}

\begin{proof}
    For each $i=1,\ldots,m$ set \[ L_i:= \left\{ k : \frac{\partial F_k}{\partial z_i} \not\equiv 0 \right\},~ E_i := L_i \setminus \bigcup_{l<i} L_l.\]
    Then $E_1,\ldots,E_m$ give a partition of $\{ 1, \ldots, N \}$ (although some of these sets may be empty). Define a function $\sigma \colon \{ 1, \ldots, N \} \to \{ 1, \ldots, m \} $ mapping each element of $E_i$ to $i$. Then $\frac{\partial F_{k}}{\partial z_{\sigma(k)}} \not\equiv 0$ for all $k$. We will show that this $\sigma$ has the desired properties.

For each $i \in \{ 1, \ldots, m\},\, q \in \Q,\, g \in \hat{\mathcal{C}}$, consider the set
    \begin{align*}
        & R_{i,q,g}:=  \bigg\{  \rbar \in B:   \text{ for all } k \in L_i : \left( \frac{\partial F_k}{\partial z_i}(\rbar) \neq  0  \text{ and } F_k(\rbar) \neq q \right) \text{ and } \\
        &  \text{ for all } k \in L_i : \big(F_k(\rbar) \neq gr_i\big), \text{ and for all distinct } k,k'\in L_i: \big( F_k(\rbar) \neq gF_{k'}(\rbar)  \big)\bigg\}.
    \end{align*}
    Each $R_{i,q,g}$ is a finite intersection of open dense sets. By the Baire category theorem, $R:=(\R \setminus \Q)^m \cap \bigcap_{i,q,g} R_{i,q,g}$ is dense in $B$.

\begin{claim}\label{claim: approximation thm}
    Let $K \seq \h^m$ be compact. For every $\rbar \in R$ and every $\habar  \in \hat{G}^N$ there are sequences $\gamma_{k,s}, \nu_{l,s}\in \hat{\Gamma},\, \eta_{l,s} \in H_{r_l},\, k=1, \ldots,N,\, l=1, \ldots, m,\, s \in \Z_{>0},$ with the following properties:
 \begin{itemize}
        \item $\eta_{l,s} = \xi_{r_l,t_{l,s}}$ with $t_{l,s} \to +\infty$ as $s \to \infty$;
        \item $\nu_{l,s} \eta_{l,s} \to I$;
        \item as $s \to \infty$, we have $\gamma_{k,s} (F_k (\etabar_s \zbar)) \to h_kz_{\sigma(k)}$ uniformly for $\zbar \in K$, where $\etabar_s \zbar := (\eta_{1,s}z_1, \ldots, \eta_{m,s}z_m)$.
 \end{itemize} 
\end{claim} 
\begin{claimproof}
    We proceed by induction on $m$, the $m=1$ case being covered by \Cref{prop: gamma F_1...F_N eta to h with conditions single variable}.

    For $m> 1$ let $f_k(z_1, \ldots, z_{m-1}) := F_k(z_1, \ldots, z_{m-1}, r_m)$. If $k \in E:= \bigcup_{i<m} E_i$ then $f_k$ is non-constant. Since $\rbar \in R$, it follows that $(r_1, \ldots, r_{m-1})$ is in the analogue of $R$ for the map $(f_k)_{k \in E}$. Let the sequences $\gamma_{k,s}, \nu_{l,s}, \eta_{l,s}\, (k \in E,\, 1 \leq l < m)$ be given by the induction hypothesis for $f_k, h_k,\, k \in E$.

    The functions $F_k$ for $k \in E_m$ (assuming $E_m \neq \emptyset$) depend only on $z_m$. Let  the sequences $\nu_{m,s}, \eta_{m,s}$ and $\gamma_{k,s}$ for $ k \in E_m$ be given by \Cref{prop: gamma F_1...F_N eta to h with conditions single variable} for these functions and for $h_k,\, k \in E_m$. Since $\rbar \in R$, the point $r_m$ satisfies the conditions of \Cref{prop: gamma F_1...F_N eta to h with conditions single variable}. When $E_m = \emptyset$, we just pick arbitrary sequences $\nu_{m,s} \in \hat{\Gamma},\, \eta_{m,s} = \xi_{r_m, t_{m,s}}$ satisfying $\nu_{m,s} \eta_{m,s} \to I$ by \Cref{cor: approx. g_0...g_N extra conditions}.

    Note that since $f_k$ does not depend on $z_m$, the expression $f_k (\etabar_s\zbar)$ should be understood as $f_k (\eta_{1,s} z_1, \ldots, \eta_{m-1,s} z_{m-1})$, that is, we think of $f_k$ as a function of $z_1, \ldots, z_m$ which, however, does not depend on $z_m$. 

    The continuity of $\gamma_{k,s} \circ F_k$ and the compactness of $K \seq \h^m$ allow us to replace $\eta_{m,s}$ by a subsequence (and $\nu_{m,s}$, $\gamma_{k',s}$ for $k' \in E_m$ by the corresponding subsequences), such that for each $k \in E$ (which implies $\sigma(k) < m$) and for all large $s\in \Z_{>0}$
\[  \sup_{\zbar \in K} |\gamma_{k,s}(F_k(\eta_{1,s} z_1, \ldots, \eta_{m,s} z_m)) - \gamma_{k,s}(F_k(\eta_{1,s} z_1, \ldots, \eta_{m-1,s} z_{m-1}, r_m))| \leq s^{-1}. \]

The induction hypothesis ensures that $\gamma_{k,s} (f_k (\etabar_s \zbar)) \to h_kz_{\sigma(k)}$ uniformly for $\zbar \in K$ whenever $k \in E$. So, $\gamma_{k,s} (F_k (\etabar_s \zbar)) \to h_kz_{\sigma(k)}$ uniformly for $\zbar \in K$ for all $k \in E$.\footnote{The last step used a slightly different argument in an earlier version of the paper. This simpler approach was suggested by ChatGPT 6 Astra during a review.}

Note that replacing the sequences $\eta_{m,s}$, $\nu_{m,s}$, $\gamma_{k',s}\, (k' \in E_m)$ with subsequences does not affect $\eta_{l,s}$ for $l<m$ and $\gamma_{k,s}$ for $k \in E$, for $f_k$ does not depend on $z_m$.
\end{claimproof}

    By a diagonal argument as in the proof of \Cref{thm: gamma F eta to h} we can extend the conclusion to every point of $B$ by the density of $R$ in $B$, use a compact exhaustion of $\h^m$ to obtain locally uniform convergence, and  replace $\eta_{l,s}$ with $\zeta_{l,s} := \nu_{l,s}^{-1}$.
\end{proof}

\begin{remark}\label{rem: variable selecting map and positive det}
\begin{itemize}
    \item[]
        
    \item We have some freedom in choosing the variable-selection map $\sigma$. Given any linear ordering of $\{ 1, \ldots, m \}$, letting $\sigma(k)$ be the smallest element with respect to this ordering with  $\frac{\partial F_{k}}{\partial z_{\sigma(k)}}\not\equiv 0$ works.
    
    \item The proof shows that we can always choose $\zeta_{l,s}$ from $\Gamma$, rather than $\hat{\Gamma}$. The reductions at the beginning of the proof of \Cref{prop: gamma F eta to h with  conditions} about the signs of the derivative of $F$ and $\det h$ only affect $\gamma_{k,s}$. Moreover, when $\frac{\partial F_{k}}{\partial z_{\sigma(k)}}(\rbar)>0$ and $h_k \in G$, we can also choose $\gamma_{k,s} \in \Gamma$.
\end{itemize}
\end{remark}

\subsection{An orbit closure theorem}\label{subsec: orbit closure}

In this subsection, we apply \Cref{thm: density of analytic orbits general version} to establish an orbit closure theorem for analytic functions under the action of $\hat{\Gamma}$. Let $\Omega \seq \C$ be a disc centred at a real point and $\tilde{\Omega}$ a domain containing $\overline{\Omega}$. Let $F_1, \ldots, F_N \colon \tilde{\Omega} \to \C$ be holomorphic functions, and $D \seq \h$ a bounded domain with $\Dbar \seq \h$. Recall the definition of the orbit $\hat{\Gamma}^N \circ \Fbar \circ \hat{\Gamma}$:
     \[ \mathcal{O}(\Fbar):= \{ (\gamma_k \circ F_k \circ \zeta)_{k=1}^N : \gamma_k, \zeta \in \PGL_2(\Z),\, \zeta(D) \seq \Omega, \, F_k(\zeta(\Dbar)) \seq \h \text{ for all }k\}.  \]

\begin{proposition}\label{prop: orbit closure}
    With the above notation, assume $F_k(\overline{\Omega} \cap \R) \seq \R$ for all $k$ and $F_k'(r) \neq 0$ for all $k$ and for all $r \in \overline{\Omega} \cap \R$. Then \[ \overline{\mathcal{O}(\Fbar)}  \seq \mathcal{O}(\Fbar) \cup \hat{G}^N\] where the closure is taken with respect to the uniform norm on $C(\Dbar, \h^N)$. If we further assume that $z,F_1(z),\ldots,F_N(z)$ are in distinct $\hat{\mathcal{C}}$-orbits, then \[ \overline{\mathcal{O}(\Fbar)}  = \mathcal{O}(\Fbar) \cup \hat{G}^N.\]
\end{proposition}

The following elementary lemma will be used in the proof.

\begin{lemma}\label{lem: zetas infty to finite}
    Let $\zeta_s \in \hat{\Gamma}$ be a sequence of modular transformations such that $\lVert \zeta_s \rVert \to \infty$ and $\zeta_s(\infty) \to r \in \R \cup \{ \infty \}$ as $s \to \infty$. Then $\zeta_s(z)\to r$ locally uniformly on $\h$.
\end{lemma}
\begin{proof}
    Let $\zeta_s := \begin{pmatrix} a_s & b_s \\ c_s & d_s    \end{pmatrix}$. By our assumption, $\frac{a_s}{c_s}\to r$. If $c_s = 0$ for a subsequence, then $r = \infty$ and $a_s, d_s \in \{ \pm 1 \}$. Then $|b_s| \to \infty$ and $\zeta_s(z) \to \infty$ uniformly on compact sets. So assume that eventually $c_s \neq 0$. Then
    \begin{equation*}
        \left| \zeta_s z - \zeta_s(\infty) \right|  =   \left|\frac{a_s z + b_s}{c_s z + d_s} - \frac{a_s}{c_s}\right| = \frac{1}{|c_s(c_sz+d_s)|}  \leq \frac{1}{c_s^2 \lvert\im z\rvert}.
    \end{equation*}

    For a compact set $K \seq \h$, we have $\min_{z \in K}\lvert \im z \rvert > 0$. So the above quantity tends to $0$ provided $|c_s| \to \infty$. If this does not hold, then pick a subsequence of $\zeta_s$ along which $c_s$ is constant. If $|d_s| \to \infty$ along this subsequence then $\frac{1}{|c_s(c_sz+d_s)|} \to 0$ locally uniformly. Otherwise, pass to a subsequence and make $d_s$ constant too. Since $a_sd_s - b_sc_s = \pm 1$, $\lVert \zeta_s \rVert \to \infty$, and $c_s \neq 0$, we have $|a_s| \to \infty$ in which case $\zeta_s(z) \to \infty$ and $\zeta_s(\infty) \to \infty$.

    This shows that the conclusion of the lemma holds for some subsequence of $\zeta_s$. Starting with an arbitrary subsequence, we can thus extract a further one converging locally uniformly to $r$. So the original sequence also tends locally uniformly to $r$.
\end{proof}

 Recall that the \textit{Schwarzian derivative} of a function $f$ is defined as \[  \Schw(f) = \left(\frac{f''}{f'}\right)' - \frac{1}{2}\left(\frac{f''}{f'}\right)^2 = \frac{f'''}{f'} - \frac{3}{2}\left(\frac{f''}{f'}\right)^2. \] This can be used to identify linear fractional functions: assuming $f$ is non-constant, $f \in \PGL_2(\C)$ if and only if $\Schw (f) \equiv 0$. The $\Rightarrow$ implication follows from straightforward calculations, and the $\Leftarrow$ direction is based on the theory of ordinary differential equations. Moreover, direct calculations also show that $\Schw$ satisfies the following chain rule:
     \[ \Schw(f \circ g) = ((\Schw f) \circ g) \cdot (g')^2 + \Schw(g).  \]

    We refer the reader to \cite[Ch.~II, \S\S1.1--1.2]{lehto1987univalent} for more details on the Schwarzian derivative.

\begin{lemma}\label{lem: limit one function mobius}
    Let $\Omega, \tilde{\Omega}, D$ be as above. Let $F \colon \tilde{\Omega} \to \C$ be holomorphic, taking real values on $\overline{\Omega} \cap \R$, such that $F'$ does not vanish on $\overline{\Omega} \cap \R$. Let $\gamma_s, \zeta_s \in \hat{\Gamma},\, s \in \Z_{>0},$ be such that $\{ \zeta_s: s \in \Z_{>0} \}$ is infinite and $\zeta_s(D) \seq \Omega$. If $h_s := \gamma_s \circ F \circ \zeta_s \to h$ uniformly on $\Dbar$, then $h \in \hat{G}$ or $h$ is a constant in $\R \cup \{ \infty \}$.
\end{lemma}
\begin{proof}
    First, pass to a subsequence such that $\lVert\zeta_s\rVert \to \infty$. Then, choose a subsequence such that $\zeta_s(\infty)$ has a limit $r$ in $\p^1(\R)$. Then, by \Cref{lem: zetas infty to finite}, $\zeta_s(z) \to r$ on any compact $K \seq \h$. In particular, $r \in \overline{\Omega} \cap \R$.    Suppose $h \not\equiv \infty$.

    Since $F$ is non-constant, we have
     \[ \Schw(h_s) = ((\Schw F) \circ \zeta_s) \cdot (\zeta_s')^2. \]

     Since $\zeta_s(z) \to r$ and $F'(r) \neq 0$, we have $((\Schw F) \circ \zeta_s)(z) \to (\Schw F)(r) \in \R$. On the other hand, $\zeta_s'(z) \to 0$, hence $\Schw(h_s) \to 0$.  Then either $h$ is constant or $\Schw(h)=0$, in which case $h \in \PGL_2(\C)$. 

     We need to show that $h$ is real. To this end, we claim that 
     \begin{equation}\label{eqn: im h and h'} 
     \lvert\im h(z)\rvert = \lvert h'(z)\rvert \cdot \lvert\im z\rvert \text{ for } z \in D .
     \end{equation}

     Note that this equality holds for any element of $\hat{G}$, including $\gamma_s$ and $\zeta_s$. Therefore, writing $w = \zeta_s(z)$, we get
     \[ \lvert\im h_s(z)\rvert = |\gamma_s'(F(w))| \cdot \lvert \im F(w) \rvert. \]
     
     On the other hand,
     \[ h_s'(z) = \gamma_s'(F(w)) \cdot F'(w) \cdot \zeta_s'(z). \]

     We conclude
     \begin{equation}\label{eqn: hs by hs'}
         \frac{\lvert\im h_s(z)\rvert}{|h_s'(z)|} = \frac{\lvert \im F(w)\rvert}{|F'(w) \cdot \zeta_s'(z)|} = \frac{\lvert \im F(w)\rvert \cdot \lvert \im z \rvert}{|F'(w) | \cdot \lvert\im w\rvert}.
     \end{equation}

    Since $F$ is real-valued on the reals, $\overline{F(\bar{u})} = F(u)$ for real $u \in\Omega$, hence also for any $u \in \Omega$ by the identity theorem. So, $\im F(u) = \frac{1}{2\iu}(F(u) - F(\bar{u})).$ 
     As $w = \zeta_s(z) \to r \in \R$, we have $F'(w)\to F'(r)$. Hence, 
     \begin{equation}\label{eqn: Im F(w) by Im w}
         \frac{ \im F(w)}{\im w} = \frac{F(w) - F(\bar{w})}{w - \bar{w}}  = \int_0^1F'(\bar{w}+ t(w-\bar{w}))\, dt \to F'(r).
     \end{equation}

     Now, \eqref{eqn: hs by hs'} and \eqref{eqn: Im F(w) by Im w} yield
     \[ \frac{\lvert\im h_s(z)\rvert}{|h_s'(z)|}  \to \lvert \im z \rvert.\]

     Thus, \eqref{eqn: im h and h'} holds as claimed. If $h$ is constant then this immediately implies $h \in \R$. 
     
     So suppose $h$ is a M\"obius map.  The expressions in \eqref{eqn: im h and h'} are not real analytic where some of the functions inside the absolute value are 0. So we take squares and get
     \[ (\im h(z))^2 = |h'(z)|^2 \cdot (\im z)^2,\]
     where all the expressions are real analytic away from poles. Since this equality between real analytic functions holds on $D$, by the identity theorem, it holds identically wherever $h$ is defined. Since $h$ is a M\"obius map, it is defined on the real line with the possible exception of at most one point. So, \eqref{eqn: im h and h'} holds for real $z$ which implies that $h(z) \in \R$ whenever $z \in \R$ and $h(z) \neq \infty$. Hence, $h$ must be a real M\"obius map as required.
\end{proof}

\begin{proof}[Proof of \Cref{prop: orbit closure}]
    Suppose for some sequences $\gamma_{k,s}, \zeta_s \in \hat{\Gamma},\, 1\leq k \leq N,\, s \in \Z_{>0},$ we have $\gamma_{k,s} \circ F_k \circ \zeta_s \to h_k \in C(\Dbar, \h)$ uniformly on $\Dbar$.

    \medskip\noindent\textbf{Case 1.} $\{ \zeta_s: s \in \Z_{>0} \}$ is finite. 
    In this case, we can choose a subsequence along which  $\zeta_s$ is constant. If for some $k$ the sequence $\gamma_{k,s}$ takes infinitely many values, then pass to a subsequence along which $\lVert \gamma_{k,s}\rVert \to \infty$ and $\gamma_{k,s}(\infty)$ converges to a point in $\p^1(\R)$. By \Cref{lem: zetas infty to finite}, the only limit points of $\gamma_{k,s}(F_k(\zeta_s(z)))$ can be real points or $\infty$. This means the limit $h_k$ is a constant function with value in $\R \cup \{ \infty \}$, which is not possible since $h_k \in C(\Dbar, \h)$. 
    
    If $\gamma_{k,s}$ takes only finitely many values for all $k$, then we simply have $\habar \in \mathcal{O}(\Fbar)$.

    \medskip\noindent\textbf{Case 2.} $\{ \zeta_s: s \in \Z_{>0} \}$ is infinite. For each $k$ we apply \Cref{lem: limit one function mobius} and conclude that $h_k \in \hat{G}$. Therefore, $\habar \in \hat{G}^N$.

    Finally, \Cref{thm: density of analytic orbits general version} immediately implies that, assuming $z,F_1(z),\ldots,F_N(z)$ are in distinct $\hat{\mathcal{C}}$-orbits, the restriction of any $\habar \in \hat{G}^N$  to $\Dbar$ can be approximated by maps of the form  $(\gamma_k \circ F_k \circ \zeta)_{k=1}^N$ with $\zeta(\Dbar)$ contained in a neighbourhood of some fixed real point $r \in \Omega \cap \R$, hence also contained in $\Omega$. This shows that $\overline{\mathcal{O}(\Fbar)} \supseteq \mathcal{O}(\Fbar) \cup \hat{G}^N$.
\end{proof}

    If we consider the orbit inside $C(\Dbar, \C^N)$, then we also get real-valued constant limits for each individual coordinate-function.

    One can also characterise the orbit closure without assuming the functions are $\hat{\mathcal{C}}$-independent, by considering various cases of dependencies between $F_k$'s. For example (cf. \cite[Lem.~5.4.7]{Gallinaro-thesis}), given a single function $F$ as above, we have
    \[ \overline{\mathcal{O}(F)} = \begin{cases} \mathcal{O}(F) & \text{ if } F \text{ is a } \hat{\mathcal{C}}\text{-transformation,}\\
    \mathcal{O}(F) \cup \hat{G} & \text{ otherwise.}
    \end{cases}   \]
    To show this, suppose $F=g \in \hat{\mathcal{C}}$. By \Cref{prop: orbit closure}, $\overline{\mathcal{O}(g)} \seq \mathcal{O}(g) \cup \hat{G} \seq \hat{G}$. So we need to show that $\mathcal{O}(g)$ is closed in $\hat{G}$. 
    
    As we remarked above, $\hat{\mathcal{C}}$ is the commensurator of $\hat{\Gamma}$, hence $g^{-1}\hat{\Gamma} g$ is commensurable with $\hat{\Gamma}$. In particular, $\hat{\Gamma}$ contains $g^{-1}\hat{\Gamma} g \cap \hat{\Gamma}$ as a subgroup of finite index $p$. Then $g^{-1}\hat{\Gamma} g\hat{\Gamma}$ is covered by at most $p$ translates of $g^{-1}\hat{\Gamma} g$, so it is discrete and closed. Hence, $\hat{\Gamma} g\hat{\Gamma}$ is discrete and closed, therefore $\mathcal{O}(g) \seq \hat{\Gamma} g\hat{\Gamma}$ is also closed.
    
    It is also possible to obtain analogous orbit-closure results for the action of $\Gamma = \PSL_2(\Z)$.

\appendix

\section{Visualising complex functions and their zeros}\label{appendix}

Complex functions of a single variable can be visualised via domain colouring. We represent a complex number by a colour where the hue, chosen from a colour wheel, encodes the argument, while the brightness encodes the modulus. For a function $L$, the hue assigned to $z$ is determined by $\arg L(z)$ using the colour wheel in \Cref{fig:key}. Selected level curves (contours of modulus) where $|L(z)|$ is constant are marked, with brightness changing and resetting as $z$ moves from one contour to the next, as shown in the diagram in \Cref{fig:key}. Within each band, points with larger values of $|L(z)|$ appear brighter. If two points within the same band have the same hue and brightness then the values of $L$ at these points are equal. This provides a visual representation of the symmetries of $L$. See \cite[\S2.5, p.~29]{wegert2012visual} for more details.
 
\iftrue

 \begin{figure}[htbp]
    \centering 
    \includegraphics[width=0.7\linewidth]{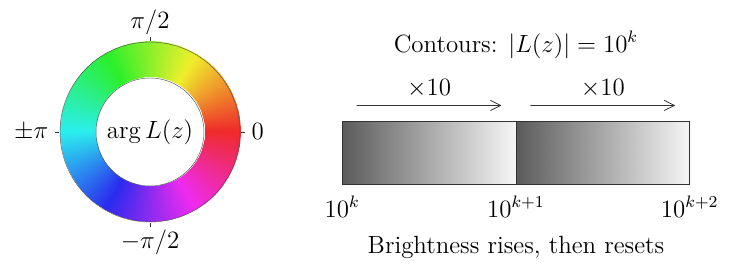} 
    \caption{Colour key}\label{fig:key}
\end{figure}

\Cref{fig:j} in \S\ref{sec: prelim} plots the $j$-function on a bounded portion of $\h$. The colour patterns match the partition of $\h$ into fundamental domains displayed in \Cref{fig:fundamental-domain}, which is expected, for $j$ is invariant under the action of $\PGL_2(\Z)$. The roughly horizontal level curves demonstrate the exponential growth of $|j(z)|$ as $\lvert \im z \rvert \to \infty$.

Below we plot two functions considered in \S\ref{sec: special cases and examples}.\footnote{The Python code used to produce these and the other images in the paper was generated and run by ChatGPT 6 Astra.} \Cref{fig:j(z)-j(z^2)-1} plots $L(z):=j(z)-j(z^2)-1$ which is a particular case of the functions considered in \S\ref{subsec: n=2 case}. \Cref{fig:z^3+j(z)+...} plots the function $L(z):= z^3+j(z) + j(\iu z)^2 + j(z+j(\iu z))$ studied in \S\ref{subsec: simple example}.

 \begin{figure}[htbp]
    \centering 
    \includegraphics[width=1\linewidth]{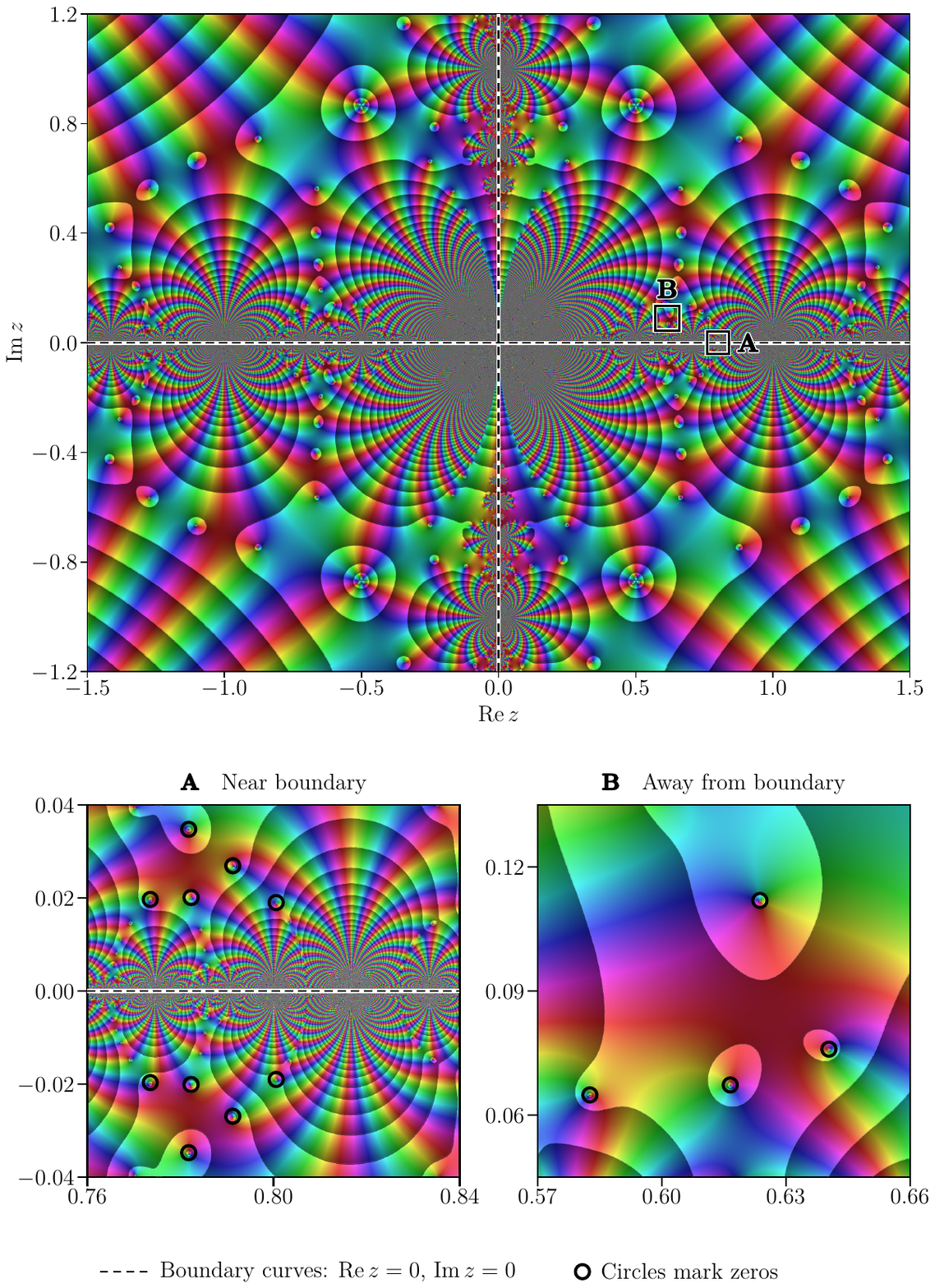} 
    \caption{$L(z)= j(z) - j(z^2)-1$}\label{fig:j(z)-j(z^2)-1}
\end{figure}

 \begin{figure}[htbp]
    \centering 
    \includegraphics[width=1\linewidth]{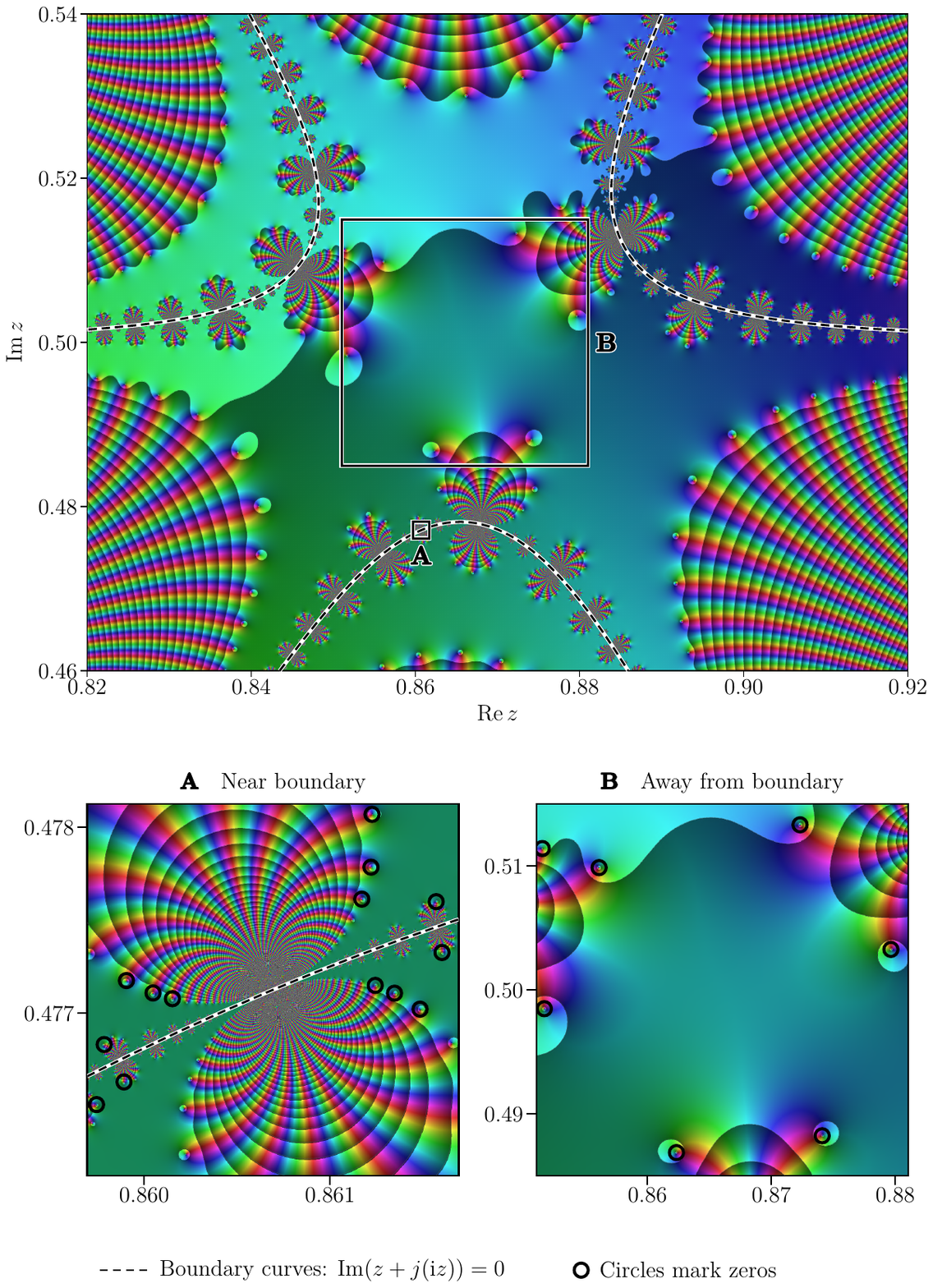} 
    \caption{$L(z)= z^3+j(z) + j(\iu z)^2 + j(z+j(\iu z))$}
    \label{fig:z^3+j(z)+...}
\end{figure}

\fi
As we can see in both figures, the behaviour of the function is wild near the boundary curves where $L$ is undefined, i.e.\ where the input of at least one instance of $j$ is real. The magnified regions display more detail near the boundary and away from it, also marking some of the zeros in these regions.

In \Cref{fig:j(z)-j(z^2)-1} the boundary curves are the real and imaginary axes where $z \in \R$ or $z^2 \in \R$.  Our proof of \Cref{thm: EC full} shows that broad and free systems of equations have infinitely many solutions near real points (where the values of the independent variables accumulate at real points). This can be seen in this figure -- every regular point where different colours (forming a full cycle of hues) converge is a zero of $L$. 

In \Cref{fig:z^3+j(z)+...}, the boundary curves are the real and imaginary axes and the curves where $\im (z+j(\iu z)) = 0$. The plotted window does not include the real and imaginary axes though, as the rapid oscillation of the function near these lines does not allow for a meaningful image with enough detail. Nevertheless, the plot still suggests that in the displayed region the zeros accumulate near the boundary curves $\im (z+j(\iu z)) = 0$.

\subsection*{Acknowledgements} I am grateful to Sebastian Eterovi\'c, Francesco Gallinaro, Jonathan Kirby, Vincenzo Mantova, and Ronnie Nagloo for numerous discussions around Modular Existential Closedness over the last few years. Eterovi\'c also made useful comments on a previous version of this paper that helped improve the presentation. I also thank Lasse Rempe and Richard Webb for discussions on homogeneous dynamics.

\bibliographystyle{alphaurl}
\bibliography{references.bib}

\end{document}

%% file: gl2z-domains-tikz.tex

\begin{tikzpicture}[
  x=2.5cm,y=2.5cm,
  tile/.style={
    draw=black!65,
    line width=0.22pt,
    line join=round
  },
  domain boundary/.style={
    draw=domainRed,
    line width=0.8pt
  },
  excluded boundary/.style={
    domain boundary,
    line cap=round,
    dash pattern=on 0pt off 2.4pt,
    preaction={draw=white,solid,line width=1.35pt}
  },
  axis/.style={
    line width=0.55pt,
    -{Stealth[length=6pt,width=4.5pt]}
  },
  tick label/.style={
    font=\scriptsize\boldmath,
    text=labelBlue,
    inner sep=0.35pt,
    outer sep=0pt,
    anchor=north,
    yshift=-8pt
  },
  font=\small
]

\definecolor{tileA}{RGB}{200,217,210}
\definecolor{tileB}{RGB}{180,220,250}
\definecolor{domainRed}{RGB}{173,38,61}
\definecolor{domainFill}{RGB}{252,229,233}
\definecolor{labelBlue}{RGB}{13, 36, 23}
\definecolor{labelRed}{RGB}{128,20,45}

\def\Cmax{8}
\def\Dmax{60}
\def\Xmax{3.45}
\def\Ymax{2.05}

\newcommand{\Geodesic}[5]{%
  \pgfmathparse{abs((#1)-(#3)) < 0.00002}%
  \ifdim\pgfmathresult pt>0pt
    \edef#5{-- (#3,#4)}%
  \else
    \pgfmathsetmacro{\gcentre}{%
      ((#1)+(#3))/2
      + (((#2)-(#4))/((#1)-(#3)))*(((#2)+(#4))/2)}%
    \pgfmathsetmacro{\gradius}{%
      veclen((#1)-\gcentre,#2)}%
    \pgfmathsetmacro{\gstart}{%
      atan2(#2,(#1)-\gcentre)}%
    \pgfmathsetmacro{\gend}{%
      atan2(#4,(#3)-\gcentre)}%
    \edef#5{arc[
      start angle=\gstart,
      end angle=\gend,
      radius=\gradius
    ]}%
  \fi
}

\newcommand{\FiniteTile}[4]{%
  \pgfmathsetmacro{\cusp}{#1/#3}%
  \pgfmathsetmacro{\denL}{#3*#3-#3*#4+#4*#4}%
  \pgfmathsetmacro{\denR}{#3*#3+#3*#4+#4*#4}%
  \pgfmathsetmacro{\xL}{%
    (#1*#3+#2*#4-(#1*#4+#2*#3)/2)/\denL}%
  \pgfmathsetmacro{\xR}{%
    (#1*#3+#2*#4+(#1*#4+#2*#3)/2)/\denR}%
  \pgfmathsetmacro{\yL}{sqrt(3)/(2*\denL)}%
  \pgfmathsetmacro{\yR}{sqrt(3)/(2*\denR)}%

  \Geodesic{\cusp}{0}{\xL}{\yL}{\edgeA}%
  \Geodesic{\xL}{\yL}{\xR}{\yR}{\edgeB}%
  \Geodesic{\xR}{\yR}{\cusp}{0}{\edgeC}%

  \pgfmathsetmacro{\tilelinewidth}{%
    max(0.025,min(0.22,14/(#3*#3+#4*#4)))}%

  \foreach \n in {-4,...,4}{%
    \pgfmathtruncatemacro{\parity}{%
      mod(#1+#2+#3+#4+#1*#4+1+\n,2)}%
    \ifnum\parity=0
      \def\tilecolour{tileA}%
    \else
      \def\tilecolour{tileB}%
    \fi
    \begin{scope}[shift={(\n,0)}]
      \path[
        tile,
        line width=\tilelinewidth pt,
        fill=\tilecolour
      ]
        (\cusp,0) \edgeA \edgeB \edgeC -- cycle;
    \end{scope}%
  }%
}

\foreach \sgn in {1,-1}{
  \begin{scope}[yscale=\sgn]
    \clip (-\Xmax,0) rectangle (\Xmax,\Ymax);

    \path[fill=tileA!50!tileB,draw=none]
      (-\Xmax,0) rectangle (\Xmax,\Ymax);

    \foreach \n in {-4,...,4}{
      \pgfmathtruncatemacro{\parity}{mod(\n,2)}
      \ifnum\parity=0
        \def\tilecolour{tileA}
      \else
        \def\tilecolour{tileB}
      \fi
      \path[tile,fill=\tilecolour]
        ({\n-0.5},\Ymax)
        -- ({\n-0.5},{sqrt(3)/2})
        arc[start angle=120,end angle=60,radius=1]
        -- ({\n+0.5},\Ymax)
        -- cycle;
    }

    \foreach \c in {1,...,\Cmax}{
      \pgfmathtruncatemacro{\amax}{\c-1}
      \pgfmathtruncatemacro{\Dlocal}{%
        max(\Cmax,ceil(\Dmax/\c))}

      \foreach \d in {-\Dlocal,...,\Dlocal}{
        \foreach \a in {0,...,\amax}{
          \pgfmathtruncatemacro{\residue}{%
            mod(\a*\d-1,\c)}
          \ifnum\residue=0
            \pgfmathtruncatemacro{\b}{(\a*\d-1)/\c}
            \FiniteTile{\a}{\b}{\c}{\d}
          \fi
        }
      }
    }

    \path[fill=domainFill]
      (-0.5,\Ymax)
      -- (-0.5,{sqrt(3)/2})
      arc[start angle=120,end angle=60,radius=1]
      -- (0.5,\Ymax)
      -- cycle;

    \path[
      pattern=north east lines,
      pattern color=domainRed!30
    ]
      (-0.5,\Ymax)
      -- (-0.5,{sqrt(3)/2})
      arc[start angle=120,end angle=60,radius=1]
      -- (0.5,\Ymax)
      -- cycle;

    \ifnum\sgn=1
      \tikzset{
        left boundary/.style={domain boundary},
        right boundary/.style={excluded boundary}
      }
    \else
      \tikzset{
        left boundary/.style={excluded boundary},
        right boundary/.style={domain boundary}
      }
    \fi

    \draw[left boundary]
      (-0.5,\Ymax)
      -- (-0.5,{sqrt(3)/2})
      arc[start angle=120,end angle=90,radius=1];

    \draw[right boundary]
      (0,1)
      arc[start angle=90,end angle=60,radius=1]
      -- (0.5,\Ymax);
  \end{scope}
}


\draw[axis]
  (-3.78,0) -- (3.78,0)
  node[right] {\phantom{$\operatorname{Re}z$}}
  node[right,overlay] {\scalebox{1.5}{$\operatorname{Re}z$}};

\node[left] at (-3.78,0)
  {\phantom{$\operatorname{Re}z$}};

\draw[axis]
  (0,-2.18) -- (0,2.22)
  node[above] {\scalebox{1.5}{$\operatorname{Im}z$}};

\foreach \n in {-3,-2,-1,1,2,3}{
  \draw (\n,0.035) -- (\n,-0.035);
  \node[tick label] at (\n,0) {\scalebox{1.7}{$\n$}};
}

\node[
  tick label,
  anchor=north east,
  xshift=-1.5pt
] at (0,0) {\scalebox{1.7}{$0$}};

\fill (0,1) circle[radius=0.8pt];
\fill (0,-1) circle[radius=0.8pt];

\node[
  font=\scriptsize\boldmath,
  text=labelBlue,
  inner sep=0.5pt,
  above left
] at (-0.02,1) {\scalebox{1.7}{$\iu$}};

\node[
  font=\scriptsize\boldmath,
  text=labelBlue,
  inner sep=0.5pt,
  below left
] at (-0.02,-1) {\scalebox{1.7}{$-\iu$}};

\node[
  font=\scriptsize\boldmath,
  text=labelRed,
  inner sep=1pt
] at (0.26,1.58) {\scalebox{1.7}{$\mathcal F_+$}};

\node[
  font=\scriptsize\boldmath,
  text=labelRed,
  inner sep=1pt
] at (0.26,-1.58) {\scalebox{1.7}{$\mathcal F_-$}};

\node at (3.03,1.77)
  {\scalebox{1.7}{$\mathbb H_{+}$}};

\node at (3.03,-1.77)
  {\scalebox{1.7}{$\mathbb H_{-}$}};

\end{tikzpicture}